\documentclass{amsart}

\usepackage{mathrsfs}
\usepackage{stmaryrd}
\usepackage{amsmath,amssymb,amsthm}
\usepackage[colorlinks=true,hyperindex, linkcolor=magenta, pagebackref=false, citecolor=cyan]{hyperref}
\usepackage[arrow, matrix, curve]{xy}
\usepackage{enumerate}
\usepackage{verbatim}
\usepackage[utf8]{inputenc}
\usepackage{xcolor}

\usepackage{lmodern}
\DeclareFontShape{OMX}{cmex}{m}{n}{
	<-7.5> cmex7
	<7.5-8.5> cmex8
	<8.5-9.5> cmex9
	<9.5-> cmex10
}{}
\SetSymbolFont{largesymbols}{normal}{OMX}{cmex}{m}{n}
\SetSymbolFont{largesymbols}{bold}  {OMX}{cmex}{m}{n}

\newtheorem{lem}{Lemma}[section]
\newtheorem{cor}[lem]{Corollary}
\newtheorem{prop}[lem]{Proposition}
\newtheorem{thm}[lem]{Theorem}
\newtheorem{def-prop}[lem]{Definition-Proposition}
\newtheorem{def-lem}[lem]{Definition-Lemma}
\theoremstyle{definition}
\newtheorem{definition}[lem]{Definition}
\newtheorem{example}[lem]{Example}
\newtheorem{convention}[lem]{Convention}

\theoremstyle{definition}
\newtheorem{remark}[lem]{Remark}
\newtheorem{notation}[lem]{Notation}
\newtheorem{observation}[lem]{Observation}

\newenvironment{mat}
{\left(\begin{smallmatrix}}
	{\end{smallmatrix}\right)}

\newcommand{\Hom}{\operatorname{Hom}}
\newcommand{\End}{\operatorname{End}}

\newcommand{\Gal}{\operatorname{Gal}}
\newcommand{\supp}{\operatorname{supp}}
\newcommand{\Ind}{\operatorname{Ind}}
\newcommand{\cInd}{\operatorname{c-Ind}}
\newcommand{\GL}{\operatorname{GL}}
\newcommand{\bQ}{\mathbf{Q}}
\newcommand{\bZ}{\mathbf{Z}}
\newcommand{\bF}{\mathbf{F}}
\newcommand{\Fpbar}{\overline{\mathbf{F}}_p}
\newcommand{\Zpbar}{\overline{\mathbf{Z}}_p}
\newcommand{\Qpbar}{\overline{\mathbf{Q}}_p}
\newcommand{\wt}{\widetilde}
\newcommand{\ol}{\overline}

\newcommand{\cH}{\mathcal{H}}
\newcommand{\cI}{\mathcal{I}}
\newcommand{\cT}{\mathcal{T}}

\newcommand{\Ext}{\operatorname{Ext}}
\newcommand{\Mod}{\operatorname{Mod}}
\newcommand{\sm}{\operatorname{sm}}
\newcommand{\Gtilde}{\widetilde{G}}
\newcommand{\sq}{\operatorname{sq}}
\newcommand{\coker}{\operatorname{coker}}
\newcommand{\Sym}{\operatorname{Sym}}
\newcommand{\Hhat}{\widehat{H}}
\newcommand{\soc}{\operatorname{soc}}

\newcommand\noloc{%
	\nobreak
	\mspace{6mu plus 1mu}
	{:}
	\nonscript\mkern-\thinmuskip
	\mathpunct{}
	\mspace{2mu}
}

\begin{document}
	
	\title{The mod-$p$ representation theory of the metaplectic cover of $\GL_2(\bQ_p)$}
	\author{Robin Witthaus}
	\email{robin.witthaus@stud.uni-due.de}
	\date{\today}

	\begin{abstract}
		Half-integral weight modular forms are naturally viewed as automorphic forms on the so-called metaplectic covering of  $\GL_2(\mathbf{A}_{\mathbf{Q}})$ -- a central extension by the roots of unity $\mu_2$ in $\mathbf{Q}$. For an odd prime number $p$, we give a complete classification of the smooth irreducible genuine mod-$p$ representations of the corresponding covering of $\GL_2(\bQ_p)$ by showing that the functor of taking pro-$p$-Iwahori-invariants and its left adjoint define a bijection onto the set of simple right modules of the pro-$p$ Iwahori Hecke algebra. As an application of our investigation of the irreducible subquotients of the universal module over the spherical Hecke algebra depending on some weight, we prove that being of finite length is equivalent to being finitely generated and admissible. Finally, we explain a relation to locally algebraic irreducible unramified genuine principal series representations.
	\end{abstract}
	
	\maketitle
	
	\tableofcontents
	
	\section{Introduction} 
		The Shimura correspondence \cite{shimura} is a map $\operatorname{Sh}$ from the space of half-integral weight cusp forms to the space of even-integral weight modular forms, satisfying $\operatorname{Sh}\circ T_{p^2} = T_p\circ \operatorname{Sh}$ for each prime $p$, where $T_{p^2}$ is the half-integral weight Hecke operator defined by Shimura and $T_p$ is the usual integral weight Hecke operator, see also \cite{Purkait}. This correspondence has been reformulated in terms of automorphic representations by Gelbart and Piatetski-Shapiro \cite{Gelbart_and_co} as well as, independently and more generally, by Flicker \cite{flicker}: letting $\mathbf{A}_{\bQ}$ denote the ring of rational ad\`{e}les, they consider the \textit{metaplectic cover}
		\begin{equation*}
			1\to \mu_2\to \wt{\GL}_2(\mathbf{A}_{\bQ})\to \GL_2(\mathbf{A}_{\bQ})\to 1,
		\end{equation*}
		which is a central topological extension by the roots of unity $\mu_2=\{\pm 1\}$ in $\bQ$. It is constructed from corresponding local metaplectic covers
		\begin{equation*}
			1\to \mu_2\to \wt{\GL}_2(\bQ_p)\to \GL_2(\bQ_p)\to 1
		\end{equation*}
		as $p$ ranges over all places, see \cite{Gelbart_book}. These local extensions are defined using the degree-two Hilbert symbol, and they split over the maximal compact open subgroup if $2<p<\infty$. The quadratic reciprocity law implies that the global covering splits over $\GL_2(\bQ)$, and Shimura's correspondence is then incarnated as a map
		\[
		\left\{\begin{array}{c}
			\text{Genuine complex}\\
			\text{automorphic representations}\\
			\text{of $\wt{\GL}_2(\mathbf{A}_{\bQ})$}
		\end{array}\right\} \to \left\{\begin{array}{c}
			\text{Complex}\\
			\text{automorphic representations}\\
			\text{of $\GL_2(\mathbf{A}_{\bQ})$}
		\end{array}\right\},
		\]
		where ``genuine'' means that the subgroup $\mu_2$ acts non-trivially. This map is defined by means of local Shimura correspondences relating those smooth irreducible admissible complex representations of $\wt{\GL}_2(\bQ_p)$ which are genuine to those of $\GL_2(\bQ_p)$ (here we restrict to finite places $p$).\\

		Given a prime $\ell$, it is natural to ask whether there exists an $\ell$-adic analogue of the local Shimura correspondence for the group $\wt{\GL}_2(\bQ_p)$, attaching an irreducible admissible unitary $\ell$-adic Banach space representation of $\GL_2(\bQ_p)$ to a genuine one of the covering group. If $\ell \neq p$, then passing to smooth vectors still remembers a large part of the representation (the smooth vectors will be dense). In fact, Vign\'{e}ras \cite{Vigneras_Banach} proved that taking smooth vectors and remembering the commensurability class of the smooth part of an invariant unit ball sets up an equivalence onto a suitable category. Since smooth representations do not see the topology of the coefficient field, one may then try to use the existing complex local correspondence to define an $\ell$-adic one. The case $\ell=p$ is more subtle (and therefore more exciting to study) as there are usually no smooth vectors at all and it is not clear how to possibly formulate a $p$-adic local Shimura correspondence. By passing to the mod-$p$ reduction of an invariant unit ball, any such correspondence should give rise to a semi-simple mod-$p$ local Shimura correspondence (for $p>2$), which seems easier to attack. Hence, instead of directly looking for a $p$-adic correspondence, it is reasonable to first study the mod-$p$ situation. In order to do this, it will be necessary---or at least very useful---to have a good understanding of the smooth irreducible admissible mod-$p$ representations of $\wt{\GL}_2(\bQ_p)$. On such an object, the central subgroup $\mu_2$ either acts trivially, i.e.\ the representation descends to $\GL_2(\bQ_p)$, or it acts via the non-trivial character in which case the representation is called genuine.\\

		The smooth irreducible $\Fpbar$-linear representations of $\GL_2(\bQ_p)$ (admitting a central character) are well-understood by the work of Barthel--Livn\'{e} \cite{Barthel-Livne} and Breuil \cite{Breuil_I}: Letting $Z\cong \bQ_p^{\times}$ denote the center of $\GL_2(\bQ_p)$, their approach is based on the spherical Hecke algebras $\End_{\GL_2(\bQ_p)}(\cInd^{\GL_2(\bQ_p)}_{\GL_2(\bZ_p)Z}(\sigma))$, where $\sigma$ is an irreducible representation of $\GL_2(\bF_p)$ viewed as a representation of $\GL_2(\bZ_p)Z$ via inflation and letting the central element $p$ act via a fixed non-zero scalar. Namely, Barthel and Livn\'{e} show that the Hecke algebra is a polynomial ring $\Fpbar[T]$ in a single Hecke operator $T$, corresponding to the double coset $\GL_2(\bZ_p)Z\begin{mat}1 & 0\\0 & p^{-1}\end{mat}\GL_2(\bZ_p)Z$ in a suitable sense, and they observe that any of the irreducible objects admitting a central character is a quotient of the cokernel of a linear polynomial in $T$, for some $\sigma$. For $\lambda\in \Fpbar^{\times}$, they compute the cokernel $\coker(T-\lambda)$ to be, up to semi-simplification, equal to a principal series, which is generically irreducible and otherwise of length two; they called the remaining irreducible representations, i.e.\ those arising as a quotient of $\coker(T)$ for some $\sigma$, supersingular. It was Breuil who proved the cokernel of $T$ to be in fact irreducible by computing the subspace of vectors fixed by the standard pro-$p$ Iwahori subgroup $I_1$, implying that any smooth irreducible $\Fpbar$-representation of $\GL_2(\bQ_p)$ admitting a central character is admissible -- Berger \cite{Berger_central} later proved the existence of central characters for any smooth irreducible $\Fpbar$-representation of $\GL_2(\bQ_p)$. 
		
		The consideration of $I_1$-invariants was studied more generally by Vign\'{e}ras \cite{Vigneras} who proved that passing to $I_1$-invariants defines a bijection between the set of isomorphism classes of smooth irreducible (admissible) $\Fpbar$-representations of $\GL_2(\bQ_p)$ and the set of isomorphism classes of (finite dimensional) simple right modules of the pro-$p$ Iwahori Hecke algebra $\cH_{\GL_2(\bQ_p)}=\End_{\GL_2(\bQ_p)}(\cInd^{\GL_2(\bQ_p)}_{I_1}(\mathbf{1}))$, where $\mathbf{1}$ denotes the trivial $\Fpbar^{\times}$-valued character of $I_1$. We emphasize that her result relies on Breuil's previously mentioned computation. Ollivier \cite{Ollivier} took this further by establishing that, at least after fixing a central character, the adjoint pair $(-\otimes_{\cH_{\GL_2(\bQ_p)}} \cInd^{\GL_2(\bQ_p)}_{I_1}(\mathbf{1}), (-)^{I_1})$ defines an equivalence between the category of right $\cH_{\GL_2(\bQ_p)}$-modules and the category of smooth $\Fpbar$-representations of $\GL_2(\bQ_p)$ generated by their $I_1$-invariants.\\

	On the other hand, to the best of the author's knowledge, the smooth \textit{genuine} mod-$p$ representation theory of $\wt{\GL}_2(\bQ_p)$ has not yet been investigated. The aim of this work is to fill this gap in the literature and prove a complete classification result for the irreducible objects. It is the point of view of Vign\'{e}ras and Ollivier that we will use to tackle this task. We also obtain a clear picture in terms of the approach of Barthel-Livn\'{e} in the genuine setting. Despite not finding any existing results for the covering of the general linear group, we point out that Peskin \cite{peskin} studied the smooth mod-$p$ representation theory of the metaplectic covering of $\operatorname{SL}_2(\bQ_p)$ by considering spherical Hecke algebras; however, it only allowed her to obtain a classification up the supersingular representations. We are optimistic that, by studying the restriction of the irreducible $\wt{\GL}_2(\bQ_p)$-representations to the covering of $\operatorname{SL}_2(\bQ_p)$, one should be able to recover and complete her classification.\\

	Finally, we note that Breuil \cite{Breuil_II} computed, depending on small weights, the mod-$p$ reduction of locally algebraic unramified principal series, the unitary completions of which were later \cite{berger_breuil} shown to be irreducible admissible $p$-adic Banach space representations of $\GL_2(\bQ_p)$.
	
	 The metaplectic counterpart of this task is interesting for two reasons: 1) it might shed some light on how to ``lift'' a mod-$p$ Shimura correspondence to a $p$-adic one and 2) his computations led Breuil to a formulation of a semi-simple mod-$p$ local Langlands correspondence (LLC) for $\GL_2(\bQ_p)$, which ultimately led to a $p$-adic LLC \cite{cdp}, and so a possible semi-simple mod-$p$ \textit{metaplectic} LLC for $\GL_2(\bQ_p)$ might give some valuable information about a $p$-adic version.
	 
	 We check that Breuil's computations also go through in our setting if the weight is at most $p-1$. While we do not attempt to formulate a mod-$p$ Shimura correspondence, we do establish a relation between the smooth genuine mod-$p$ representation theory of the metaplectic covering group of $\GL_2(\bQ_p)$ and Galois representations in \cite{witthaus_montreal}.

	\subsection{Results} Fix an odd prime number $p$ and let $\Fpbar$ be our field of coefficients.\footnote{In the main text we work over an arbitrary algebraically closed field of characteristic $p$.} Put $G=\GL_2(\bQ_p)$ so that the metaplectic covering takes the form
	\[
	1\to \mu_2\to \wt{G}\to G\to 1.
	\]
	It splits over the maximal compact open subgroup $K=\GL_2(\bZ_p)$ and thus uniquely over the pro-$p$ group $I_1$, so that we can unambiguously identify $\wt{I}_1=I_1\times \mu_2$ as groups. Let $\iota\colon \mu_2\hookrightarrow \bF_p^{\times}\subset \Fpbar^{\times}$ denote the non-trivial character, and let $\Mod^{\sm}_{\Gtilde,\iota}(\Fpbar)$ denote the category of smooth $\Fpbar$-representations of $\Gtilde$, which are genuine, i.e.\ on which the central subgroup $\mu_2$ acts via $\iota$. Let $\cH=\End_{\Gtilde}(\cInd^{\Gtilde}_{\wt{I}_1}(\iota))$ be the genuine version of the pro-$p$ Iwahori Hecke algebra. Denoting by $\Mod_{\cH}$ the category of right $\cH$-modules, we have the adjoint pair $(\cT,\cI)$,
	\[
	\cT(-)=(-)\otimes_{\cH} \cInd^{\Gtilde}_{\wt{I}_1}(\iota)\colon \Mod_{\cH} \rightleftarrows \Mod^{\sm}_{\Gtilde,\iota}(\Fpbar) \noloc (-)^{I_1}=\cI(-).
	\]
	
	\begin{thm}
		The adjoint pair $(\cT,\cI)$ defines a bijection
		\[
		\left\{\begin{array}{c}
			\text{Simple right} \\ \text{$\cH$-modules}\end{array}\right\} \cong \left\{\begin{array}{c}
			\text{Smooth irreducible genuine}\\
			\text{$\Fpbar$-representations of $\wt{G}$}
		\end{array}\right\},
		\]
		where both sides are considered up to isomorphism.
	\end{thm}

	\noindent This result is a consequence of theorems \ref{thm_ps}, \ref{thm_ss} and \ref{fg+adm=fl}, the last of which says that being of finite length as a smooth $\wt{G}$-representation is equivalent to being finitely generated and admissible, as well as a careful analysis of the objects appearing on the left-hand side. In order to describe the latter, we follow the approach of Vign\'{e}ras in \cite{Vigneras} and decompose the Hecke algebra into a direct sum of smaller algebras whose centers we compute to be either isomorphic to $\Fpbar[X,Y,Z^{\pm 1}]/(XY)$ (regular case) or to $\Fpbar[\mathfrak{Z},Z^{\pm 1}]$ (non-regular case). We show that a simple $\cH$-module is finite dimensional and thus admits a central character. It will be called supersingular if it is killed by $X$ and $Y$ (resp.\ by $\mathfrak{Z}$), otherwise we call the module principal. This leads to the following result, which is a summary of propositions \ref{simple_regular_ss}, \ref{simple_regular_principal} and \ref{simple_non-regular}.
	
	\begin{thm}
		The simple right $\cH$-modules fall into two subclasses:
		\begin{enumerate}
			\item[{\rm (a)}] Principal modules
			\item[{\rm (b)}] Supersingular modules.
		\end{enumerate}
		Each of these subclasses is in itself described explicitly.
	\end{thm}

	As a consequences of the two theorems, we obtain an explicit description of the smooth irreducible genuine $\Fpbar$-representations of $\wt{G}$. In particular, there are only two kinds of such objects: principal series and supersingular representations (i.e.\ those not occuring as a subquotient of a principal series). From the point of view of modules over the pro-$p$ Iwahori Hecke algebra, the supersingular objects naturally come in pairs. Moreover, it turns out that a genuine principal series is always irreducible; they are in bijective correspondence with pairs of smooth characters on the subgroup $S$ of squares in $\bQ_p^{\times}$: We have an equality of groups $\wt{B}_2=B_2\times \mu_2$, where $B_2$ is the subgroup of $G$ consisting of those upper-triangular matrices having elements of even valuation on the diagonal, and the correspondence is given by $(\chi_1,\chi_2)\mapsto \Ind^{\wt{G}}_{\wt{B}_2}(\chi_1'\otimes \chi_2' \boxtimes \iota)$ for any extension $\chi_i'$ of $\chi_i$ to the subgroup $p^{2\bZ}\bZ_p^{\times}$.\\
	
	As for the spherical Hecke algebras, having fixed an explicit splitting $\wt{K}\cong K\times \mu_2$, we may form the smooth irreducible genuine $\Fpbar$-representation $\Sym^r(\Fpbar^2)\boxtimes \iota$ of $\wt{K}$ for $0\leq r\leq p-1$.  The center $Z(\Gtilde)$ being equal to the product $Z^{\sq}\times \mu_2$, we may extend it to a smooth genuine $\wt{K}Z(\Gtilde)$-representation by letting the central element $p^2$ act trivially. Inside the spherical Hecke algebra $\End_{\Gtilde}(\cInd^{\Gtilde}_{\wt{K}Z(\wt{G})}(\Sym^r(\Fpbar^2)\boxtimes \iota))$ one then has the Hecke operator $\wt{T}$ corresponding to the double coset represented by a lift of $\begin{mat}
		1 & 0\\0 & p^{-2}
	\end{mat}$, which also exists in characteristic $0$, i.e.\ if $\Fpbar$ is replaced by $\Qpbar$ or $\Zpbar$. We summarize propositions \ref{Hecke_poly}, \ref{prop_lattice} and Theorem \ref{thm_JH} as follows.

\begin{thm}
	Assume that $0\leq r\leq p-1$ and let $a_p\in \ol{\bZ}_p$.
	
	\begin{enumerate}
		\item[{\rm (i)}] The spherical Hecke algebra $\End_{\Gtilde}(\cInd^{\Gtilde}_{\wt{K}Z(\wt{G})}(\Sym^r(\Fpbar^2)\boxtimes \iota))=\Fpbar[\wt{T}]$ is a polynomial ring in the Hecke operator $\wt{T}$.
		
		\item[{\rm (ii)}] The canonical map
		\[
		\cInd^{\wt{G}}_{\wt{K}Z(\wt{G})}(\Sym^r(\ol{\bZ}_p^2)\boxtimes \iota)/(\wt{T}-a_p)\hookrightarrow 	\cInd^{\wt{G}}_{\wt{K}Z(\wt{G})}(\Sym^r(\ol{\bQ}_p^2)\boxtimes \iota)/(\wt{T}-a_p)
		\]
		is injective and defines a $\wt{G}$-stable bounded lattice.
		
		\item[{\rm (iii)}] Let $\bar{a}_p\in \ol{\bF}_p$ denote the reduction of $a_p$ modulo the maximal ideal in $\ol{\bZ}_p$.
		
		\begin{itemize}
			\item If $\bar{a}_p$ is non-zero, then $\cInd^{\wt{G}}_{\wt{K}Z(\wt{G})}(\Sym^r(\ol{\bF}_p^2)\boxtimes \iota)/(\wt{T}-\bar{a}_p)$ is a principal series representation.
			\item If $\bar{a}_p=0$, then $\cInd^{\wt{G}}_{\wt{K}Z(\wt{G})}(\Sym^r(\ol{\bF}_p^2)\boxtimes \iota)/(\wt{T})$ is an extension of two distinct supersingular representations, which is split if and only if $r=\frac{p-1}{2}$.
		\end{itemize}
	\end{enumerate}
\end{thm}

\noindent The pairs of genuine supersingular representations appearing in part (iii) are precisely those alluded to above in terms of the classification of simple $\cH$-modules.

In Corollary \ref{loc_alg_descr}, we also give a description of the locally algebraic representation appearing in part (ii) of the theorem in terms of locally algebraic \textit{irreducible} $\Qpbar$-representations -- in fact, it will generically be a locally algebraic irreducible unramified principal series. This is deduced from the corresponding statement \cite{Breuil_II} for the group $G$ by using Savin's local Shimura correspondence \cite{Savin} formulated in terms of Iwahori Hecke algebras.

While we have not investigated the existence (let alone its mod-$p$ reduction) of a $\wt{G}$-stable bounded lattice in such a locally algebraic representation in case $p\leq r\leq 2(p-1)$, we still expect the metaplectic version of \cite[Proposition 5.3.4.1]{Breuil_II} to be true. We can prove a necessary consequence of such a result, namely the non-vanishing of the Ext-space in part (iii) of the following proposition, see Proposition \ref{extension_PS}. For two smooth characters $\chi_1,\chi_2$ of $A=p^{2\bZ}\bZ_p^{\times}$, define the corresponding normalized principal series by $\tilde{\pi}(\chi_1,\chi_2)=\Ind^{\wt{G}}_{\wt{B}_2}((\chi_1\otimes \chi_2 \omega^{-1})\boxtimes \iota)$ only depending on the restriction of $\chi_1,\chi_2$ to the subgroup $S$ of squares in $\bQ_p^{\times}$, where $\omega \colon \bQ_p^{\times}\twoheadrightarrow \bF_p^{\times}$ is the projection onto the torsion part.

	\begin{prop}
	Let $\chi_1,\chi_2$ and $\chi_1',\chi_2'$ be two pairs of smooth characters $A\to \ol{\bF}_p^{\times}$.
	\begin{enumerate}
		\item[{\rm (i)}] The space
		\[
		\Ext^1_{\wt{G}}(\tilde{\pi}(\chi_1,\chi_2),\tilde{\pi}(\chi_1',\chi_2'))=0
		\]
		vanishes
		unless $(\chi'_1|_S,\chi'_2|_S)$ is equal to $(\chi_1|_S,\chi_2|_S)$ or $(\chi_2|_S,\chi_1|_S)$.
		
		\item[{\rm (ii)}] The map
		\[
		\Ext^1_{T_2}(\chi_1\otimes \chi_2\omega^{-1},\chi_1\otimes \chi_2\omega^{-1})\xrightarrow{\cong} 		\Ext^1_{\wt{G}}(\tilde{\pi}(\chi_1,\chi_2),\tilde{\pi}(\chi_1,\chi_2)),
		\]
		induced by the exact functor $\Ind^{\wt{G}}_{\wt{B}_2}(-\boxtimes \iota)$, is an isomorphism, where $T_2=\begin{mat}
			A & 0\\0 & A
		\end{mat}$.
		
		\item[{\rm (iii)}] If $\chi_1|_S \neq \chi_2|_S$, then
		\[
		\dim_{\ol{\bF}_p}	\Ext^1_{\wt{G}}(\tilde{\pi}(\chi_1,\chi_2),\tilde{\pi}(\chi_2,\chi_1))=1.
		\]
	\end{enumerate}
\end{prop}

\noindent This proposition is proved by comparing it to the situation for the group $G$, where the statement is known, see for example \cite{ord_parts_II}.\\

\noindent \textbf{Acknowledgements.} This work is part of the author's PhD thesis and he would like to heartily thank his advisor Vytautas Pa\v{s}k\={u}nas for suggesting the topic at hand and for fruitful discussions during the past years. Special thanks go
to Florian Herzig for his lectures \cite{herzig_notes} at the Spring School on the mod-$p$ Langlands Correspondence,
which took place in Essen (online) in April 2021; Section \ref{basic_lemma_section} is more or less an excerpt of his material. This work was partially funded by the DFG
Graduiertenkolleg 2553.

\subsection{Conventions and notations}\label{conv_notations} Fix an odd prime number $p$ and an algebraically closed field $k$ of characteristic $p$. Denote by $\bQ_p$ the completion of the rationals $\bQ$ with respect to the $p$-adic valuation $v\colon \bQ_p\to \bZ\cup \{\infty\}$. We use the standard notations
\[
G=\GL_2(\bQ_p) \supset B=\begin{pmatrix}
	\bQ_p^{\times} & \bQ_p\\0& \bQ_p^{\times}
\end{pmatrix}\supset T=\begin{pmatrix}
	\bQ_p^{\times} & 0\\0 & \bQ_p^{\times}
\end{pmatrix}\supset Z,
\]
where $Z\cong \bQ_p^{\times}$ denotes the center of $G$, and we let $K=\GL_2(\bZ_p)$ be the standard maximal compact open subgroup of $G$.
It will be convenient to also introduce the following subgroups of $\bQ_p^{\times}$:
\[
\mu_2=\{\pm 1\}, \hspace{0.1cm} A=p^{2\bZ}\bZ_p^{\times} \supset S=\{x^2 : x\in \bQ_p^{\times}\}.
\]

\noindent For a topological group $\mathscr{G}$, we denote by $\Mod^{\sm}_{\mathscr{G}}(k)$ the category of $k$-linear smooth representations of $\mathscr{G}$. Here, smooth means that the stabilizer of each vector is open.

\smallskip
If $\mathscr{H}\subset \mathscr{G}$ is a closed subgroup and $\pi\in \Mod^{\sm}_{\mathscr{H}}(k)$, we denote the smooth induced representation by $\Ind^{\mathscr{G}}_{\mathscr{H}}(\pi)$. It is right adjoint to the restriction functor $(-)|_{\mathscr{H}}$, which we will refer to as Frobenius reciprocity. 

If $\mathscr{H}$ is an open subgroup, we denote the  compact induction of $\pi$ by $\cInd^{\mathscr{G}}_{\mathscr{H}}(\pi)$. It is the subrepresentation of $\Ind^{\mathscr{G}}_{\mathscr{H}}(\pi)$ consisting of those functions having finite support modulo $\mathscr{H}$, and defines a left adjoint to the restriction functor $(-)|_{\mathscr{H}}$, which we will also refer to as Frobenius reciprocity.

For $\sigma_1,\sigma_2\in \Mod^{\sm}_{\mathscr{H}}(k)$, we will identify $\Hom_{\mathscr{G}}(\cInd^{\mathscr{G}}_{\mathscr{H}}(\sigma_1),\cInd^{\mathscr{G}}_{\mathscr{H}}(\sigma_2))$ with the $k$-vector space
\[
\left\{\varphi\colon \mathscr{G}\to \Hom_{k}(\sigma_1,\sigma_2) : \begin{array}{l}\diamond \hspace{0.1cm} \varphi(h_2gh_1)=\sigma_2(h_2)\varphi(g)\sigma_1(h_1), \forall h_1,h_2\in \mathscr{H}, g\in \mathscr{G}\\
	\diamond \hspace{0.1cm} |\mathscr{H}\setminus \supp(\varphi)/\mathscr{H}|<\infty\end{array}\right\}
\]
with left $\End_{\mathscr{G}}(\cInd^{\mathscr{G}}_{\mathscr{H}}(\sigma_2))$-, resp.\ right $\End_{\mathscr{G}}(\cInd^{\mathscr{G}}_{\mathscr{H}}(\sigma_1))$-module structure given by the convolution product
\[
(\varphi_1 \varphi_2)(g)=\sum_{x\in \mathscr{H}\setminus \mathscr{G}} \varphi_1(gx^{-1})\varphi_2(x), \text{ for all $g\in \mathscr{G}$.}
\]

For a subgroup $\mathscr{H}\subset \mathscr{G}$, $\pi\in \Mod^{\sm}_{\mathscr{H}}(k)$ and $g\in \mathscr{G}$, we denote by $\pi^{g}=\pi(g(-)g^{-1})\in \Mod^{\sm}_{g^{-1}\mathscr{H}g}(k)$ the conjugated representation.

\smallskip

If $\mathscr{G}$ is abelian, we write $\mathscr{G}^{\sq}=\{g^2 : g\in \mathscr{G}\}$ for the subgroup of squares.

\smallskip

For a ring $R$, we denote by $\Mod_R$ the category of right $R$-modules and by $M_{n\times n}(R)$ the ring of $n\times n$-matrices with entries in $R$, for $n\in \bZ_{\geq 1}$.

\section{The metaplectic covering group}

The metaplectic covering group is defined by an explicit cocycle relying on the Hilbert symbol, which we briefly remind the reader of following \cite[V. §3]{Neukirch_AZ}.

\subsection{Hilbert symbol}

Let $L=\bQ_{p}(\sqrt{\smash[b]{\bQ_p^{\times}}})=\bQ_{p^2}(\sqrt{\smash[b]{p}})$ be the unique abelian extension of $\bQ_p$ with norm group $S$.  By local class field theory and Kummer theory, respectively, we have isomorphisms $\Gal(L/\bQ_p)\cong \bQ_p^{\times}/S$ and $\Hom(\Gal(L/\bQ_p),\mu_2)\cong \bQ_p^{\times}/S$. In particular, the evaluation map $\Gal(L/\bQ_p)\times \Hom(\Gal(L/\bQ_p),\mu_2)\to \mu_2$
is identified with a non-degenerate bilinear map
\[
(-,-)\colon \bQ_p^{\times}/S\times \bQ_p^{\times}/S\to \mu_2,
\]
which is called the \textit{Hilbert symbol}. We will usually view it as a bilinear map on $\bQ_p^{\times}\times \bQ_p^{\times}$ via inflation.

Denote by $\omega\colon \bQ_p^{\times}\twoheadrightarrow [\bF_p^{\times}]\cong \bF_p^{\times}$ the projection onto the torsion part, where $[-]$ denotes the Teichmüller lift. By \cite[V. Satz (3.4)]{Neukirch_AZ}, we then have
\begin{equation}\label{HS}
	(a,b)=\omega\left((-1)^{v(a)v(b)} \frac{b^{v(a)}}{a^{v(b)}}\right)^{\frac{p-1}{2}} \text{ for all $a,b\in \bQ_p^{\times}$.}
\end{equation}
From this one deduces the following properties of the Hilbert symbol.

\begin{lem}\label{Hilbert_symbol_identities}
	Let $a,a',b,b'\in \bQ_p^{\times}$. Then
	\begin{enumerate}
		\item[{\rm (i)}] $(aa',b)=(a,b)(a',b)$;
		\item[{\rm (ii)}] $(a,bb')=(a,b)(a,b')$;
		\item[{\rm (iii)}] $(a,b)^{-1}=(b,a)$;
		\item[{\rm (iv)}] $(a,1-a)=1$ and $(a,-a)=1$;
		\item[{\rm (v)}] $(a,x)=1$ for all $x\in \bQ_p^{\times}$ if and only if $a\in S$;
		\item[{\rm (vi)}] $(a,b)=(b,a)$;
		\item[{\rm (vii)}] If $a,b\in A$, then $(a,b)=1$.
	\end{enumerate}
\end{lem}

\begin{proof}
	Parts (i)-(v) are contained in \cite[V. Satz (3.2)]{Neukirch_AZ}, while (vi) and (vii) are immediate from (\ref{HS}) noting that an element in $\mu_2$ is self-inverse.
\end{proof}

\subsection{Definition and first properties}

The \emph{metaplectic covering of $G$} is the central topological extension
\begin{equation}\label{metaplectic_covering}
	1\to \mu_2\to \Gtilde \to G\to 1
\end{equation}
defined by the two-cocycle $\sigma\colon G\times G\to \mu_2$,
\begin{equation}\label{cocycle}
	\sigma(g_1,g_2)=\left(\frac{\mathfrak{c}(g_1g_2)}{\mathfrak{c}(g_1)},\frac{\mathfrak{c}(g_1g_2)}{\mathfrak{c}(g_2)} \det(g_1)\right) \text{ for all $(g_1,g_2)\in G\times G$},
\end{equation}
where $\mathfrak{c}\colon G\to \bQ_p^{\times}$ is the function $\mathfrak{c}\left(\begin{mat}
	a & b\\ c & d
\end{mat}\right)=\begin{cases}
	c \text{ if $c\neq 0$}\\
	d \text{ if $c=0$,}
\end{cases}$ cf.\ \cite[p.\ 41]{Kazhdan_Patterson}.

	In other words, $\Gtilde = G\times \mu_2$ and the product of two elements $(g_i,\zeta_i)$, $i=1,2$, is equal to $(g_1,\zeta_1)(g_2,\zeta_2)=(g_1g_2,\zeta_1\zeta_2 \sigma(g_1,g_2))$.

\begin{notation} Given a subset $D\subset G$, we denote by $\wt{D}$ its preimage in $\wt{G}$ under the projection map $\wt{G}\to G$. For $g\in G$, we put $\tilde{g}=(g,1)$.
\end{notation}

\begin{example}\label{law_torus}
	The multiplication law on the covering group $\wt{T}$ of the torus $T$ is given by
	\[
	\left(\begin{mat}
		x_1 & 0\\0 & y_1
	\end{mat},1\right)	\left(\begin{mat}
		x_2 & 0\\0 & y_2
	\end{mat},1\right)=	\left(\begin{mat}
		x_1x_2 & 0\\0 & y_1y_2
	\end{mat},(y_2,x_1)\right),
	\]
	for all $x_i,y_i\in \bQ_p^{\times}$, $i=1,2$. In particular, it is not abelian, cf.\ Lemma \ref{centers} (ii) below. We also note that for all $x,y\in \bQ_p^{\times}$,
	\[
	\left(\begin{mat}
		x & 0\\0 & y
	\end{mat},1\right)^{-1}=\left(\begin{mat}
		x^{-1} & 0\\0 & y^{-1}
	\end{mat}, (y,x)\right) \text{ and } \left(\begin{mat}
		0 & x\\ y & 0
	\end{mat},1\right)^{-1}=\left(\begin{mat}
		0 & y^{-1}\\x^{-1} & 0
	\end{mat},1\right).
	\]
\end{example}

\begin{lem}\label{split_subgroups}
	The metaplectic covering (\ref{metaplectic_covering}) splits over the subgroups $K=\GL_2(\bZ_p)$, $B_2=\begin{mat}
		A & \bQ_p\\0 & A
	\end{mat}$ and $B_S=\begin{mat}
		\bQ_p^{\times} & \bQ_p\\0 & S
	\end{mat}$. More precisely:
	\begin{enumerate}
		\item[{\rm (i)}] The map $K\to \mu_2$ given by 
		\[
		g=\begin{pmatrix}
			a & b\\c & d
		\end{pmatrix}\mapsto \begin{cases}(c, d {\det(g)}^{-1}) \text{ if } c\in p\bZ_p\setminus \{0\}\\ 1 \text{ otherwise}\end{cases}
		\]
		defines a splitting $K\times \mu_2\cong \wt{K}$.
		\item[{\rm (ii)}] We have equalities of groups $\wt{B}_2=B_2 \times \mu_2$, $\wt{B}_S=B_S\times \mu_2$.
	\end{enumerate}
\end{lem}

\begin{proof}
	Part (i) is contained in \cite[p.\ 43]{Kazhdan_Patterson}, while part (ii) follows from the cocycle $\sigma$ being trivial when restricted to the subgroups $B_2$ and $B_S$, respectively, which is a consequence of Lemma \ref{Hilbert_symbol_identities}.
\end{proof}

\begin{convention} From now on we will fix the splitting on $K$ defined in part (i) of the lemma and view subgroups of $K$ as subgroups of $\wt{K}\subset \wt{G}$.
\end{convention}

	\begin{lem}\label{centers}
	We have the following equalities of groups describing the centers of $\wt{G}$ and $\wt{T}$, respectively:
	\begin{enumerate}	
		\item[{\rm (i)}] $Z(\wt{G}) = Z^{\sq}\times \mu_2$;
		\item[{\rm (ii)}] $Z(\wt{T})= T^{\sq}\times \mu_2$.
	\end{enumerate}
\end{lem}

\begin{proof}
	Note that by Lemma \ref{split_subgroups} the groups on the right-hand side of the equation are indeed subgroups of $\wt{G}$. Property (v) of the Hilbert symbol in Lemma \ref{Hilbert_symbol_identities} implies that they are contained in the respective center. For proving the converse inclusions, we first consider part (ii): Let $(t,\zeta)\in Z(\wt{T})$. We may assume that $\zeta=1$ as $\mu_2$ is a central subgroup. By the formula for the multiplication law in Example \ref{law_torus}, the commuting of $\tilde{t}$ with elements in $K\cap T$ forces the diagonal entries of $t$ to lie in $A$, and the commuting with  $\left(\begin{mat}
		p & 0\\0 & 1
	\end{mat},1\right)$ and $\left(\begin{mat}
		1 & 0\\0 & p
	\end{mat},1\right)$ then forces $t\in T^{\sq}$, which finishes the proof of (ii). Part (i) now follows from the inclusion $Z(\wt{G})\subset Z(\wt{T})\cap \wt{Z}$.
\end{proof}

	\subsection{Genuine representations and further notations}\label{gen_reps_and_notations} Fix the non-trivial character $\iota\colon \mu_2\hookrightarrow \bF_p^{\times}\subset k^{\times}$. Let $\mathscr{G}\subset G$ be a closed subgroup.

\begin{definition}
	We let $\Mod^{\operatorname{sm}}_{\wt{\mathscr{G}},\iota}(k)$ denote the full subcategory of $\Mod^{\sm}_{\wt{\mathscr{G}}}(\Fpbar)$ consisting of \textit{genuine} representations, i.e.\ those objects on which $\mu_2$ acts via the non-trivial character $\iota$.
\end{definition}

Since $p$ is prime to the order of the central subgroup $\mu_2$, we have a decomposition of categories
\[
\Mod^{\operatorname{sm}}_{\wt{\mathscr{G}}}(k)=\Mod^{\operatorname{sm}}_{\wt{\mathscr{G}},\iota}(k) \times \Mod^{\operatorname{sm}}_{\mathscr{G}}(k),
\]
where the second factor is viewed as a full subcategory via inflation.\\

Recall that a smooth representation $\pi\in \Mod^{\operatorname{sm}}_{\wt{\mathscr{G}}}(k)$ is admissible if the subspace $\pi^{\mathscr{H}}$ of $\mathscr{H}$-fixed vectors is finite dimensional for every open subgroups $\mathscr{H}\subset \mathscr{G}$. Since the identity in $\mathscr{G}$ admits a basis of open neighbourhood consisting of pro-$p$ groups and we are working with mod-$p$ coefficients, admissibility can be checked for a single open pro-$p$ subgroup $\mathscr{H}$, see \cite[Proposition 36]{herzig}. Examples of open pro-$p$ subgroups of $K$ (and hence of $\wt{G}$) are given by the congruence subgroup of level $m\geq 1$ and the pro-$p$ Iwahori subgroup:
\[
K_m=1+p^{m}M_{2\times 2}(\bZ_p), \hspace{0.1cm} I_1=\begin{pmatrix}
	1+p\bZ_p & \bZ_p\\p\bZ_p & 1+p\bZ_p
\end{pmatrix}.
\]
Let $I=\begin{pmatrix}
	\bZ_p^{\times} & \bZ_p\\ p\bZ_p & \bZ_p^{\times}
\end{pmatrix}$ denote the Iwahori subgroup and set $H=\begin{pmatrix}
	[\bF_p^{\times}] & 0\\0 & [\bF_p^{\times}]
\end{pmatrix}$, we have the semi-direct product decomposition $I=I_1 \rtimes H$.\\

Finally, let us agree on the following notation: given a subgroup $C\subset G$ together with a splitting $\wt{C}\cong C\times \mu_2$, we denote the image of an object $\pi$ under the induced equivalence $\Mod^{\sm}_C(k) \cong \Mod^{\sm}_{\wt{C},\iota}(k)$ by $\pi\boxtimes \iota$.

\section{The pro-$p$ Iwahori Hecke algebra}

In this section, we study the pro-$p$ Iwahori Hecke algebra by decomposing it into smaller algebras which we then describe explicitly, enabling us to determine the simple modules. 

By Lemma \ref{split_subgroups} (i), the metaplectic covering splits over the maximal compact open subgroup and thus also over the pro-$p$ Iwahori subgroup $I_1$. Since $p$ is odd, any such splitting is unique and we in fact have an equality of groups $\wt{I}_1=I_1\times \mu_2$ because $1+p\bZ_p\subset S$ so that the splitting in the cited lemma is trivial on $I_1$. We will simply view the elements of $I_1$ as elements of $\wt{I}_1$ and drop the $\mu_2$-component (which is $1$) from the notation. Taking the trivial character $\mathbf{1}$ on $I_1$, we obtain the genuine $k^{\times}$-valued character $\mathbf{1}\boxtimes \iota$ on $\wt{I}_1$, which is just the inflation of $\iota$.

\begin{definition}
	The \textit{pro-$p$ Iwahori Hecke algebra $\cH$} is the endomorphism ring of the compact induction $\cInd^{\wt{G}}_{\wt{I}_1}(\mathbf{1}\boxtimes \iota)$ in the category of $\wt{G}$-representations, i.e.\ $\cH=\End_{\wt{G}}(\cInd^{\wt{G}}_{\wt{I}_1}(\mathbf{1}\boxtimes \iota))$.
\end{definition}

As agreed upon in our list of conventions and notations \ref{conv_notations}, we will identify $\cH$ with the $k$-algebra
\[
\left\{\varphi\colon I_1\setminus \Gtilde/I_1\to k : \begin{array}{l}\diamond \hspace{0.1cm} \varphi(g \zeta)=\zeta \varphi(g), \forall g\in \Gtilde, \zeta\in \mu_2\\
	\diamond \hspace{0.1cm} |\supp(\varphi)|<\infty\end{array}\right\},
\]
the product being given by convolution.

\subsection{Generators and relations} In this subsection, we prove some useful identities of Hecke operators by following \cite{coeffsystems}, which treats the case of $\GL_2$ over a local field. We start by writing down a basis of $\cH$ as a $k$-vector space. Let $N_G(T)$ be the normalizer of $T$ in $G$. The following two elements of $N_G(T)$ will play a prominent role in this work:
\[
s=\begin{pmatrix}
	0 & 1\\ 1 & 0
\end{pmatrix}\text{ and } \Pi =\begin{pmatrix}
	0 & 1\\p & 0
\end{pmatrix}.
\]
Putting $W_1:=N_G(T)/(I_1\cap T)$, we have the decomposition
\begin{equation}\label{pro-p_decomp}
	\wt{G}=\bigsqcup_{w\in W_1} \wt{I}_1 \tilde{w} \wt{I}_1,
\end{equation}
see also \cite[Lemma 2.0.4]{coeffsystems}. The ``genuine characteristic functions'' of the double cosets appearing in this decomposition will define a $k$-basis of $\cH$ -- provided they are well-defined, which we now check. 

\begin{lem}\label{wd_operator}
	For all $g\in I_1$ and $n\in N_G(T)$, $\tilde{n}g(\tilde{n})^{-1}=(ngn^{-1},1)$. In particular, for all $n\in N_G(T)$, $\mu_2\cap I_1 \tilde{n}I_1(\tilde{n})^{-1}=\{1\}$.
\end{lem}

\begin{proof}
	Recall that  $N_G(T)=T\sqcup sT$. Let $n\in N_G(T)$ and write $g=\begin{mat}
		a & b\\c & d
	\end{mat}\in I_1$. The end of Example \ref{law_torus} gives a formula for the inverse $(\tilde{n})^{-1}$ of $\tilde{n}=(n,1)$. Distinguishing between the cases $c=0$ and $c\neq 0$, if $n\in T$, and the cases $b=0$ and $b\neq 0$, if $n\in sT$, the lemma follows from the definition of the cocycle (\ref{cocycle}) and the identities of the Hilbert symbol in Lemma \ref{Hilbert_symbol_identities}, bearing in mind that $1+p\bZ_p\subset S$.
\end{proof}

\begin{definition}
	For $w\in W_1$, let $T_w\in \mathcal{H}$ be the unique function with
	\[
	\supp(T_w)=\wt{I}_1\tilde{w} \wt{I}_1 \text{ and } T_w(\tilde{w})=1.
	\]
	If $n\in N_G(T)$, then we also write $T_n$ for the operator defined by the class of $n$ in $W_1$.
\end{definition}

\begin{cor}
	The operators $T_w$, for $w\in W_1$, are well-defined and the collection $\{T_w\}_{w\in W_1}$ is a $k$-basis of $\cH$.
\end{cor}

\begin{proof}
	For $t\in T\cap I_1$, we have $\tilde{g}\tilde{t}=\wt{gt}$ for all $g\in G$, so the condition $T_w(\tilde{w})=1$ is independent of the choice of a representative of $w$. By Lemma \ref{wd_operator}, the operator $T_w$ is then well-defined. The decomposition (\ref{pro-p_decomp}) implies the assertion about the basis.
\end{proof}

\begin{remark}\label{conv_basis}
	Let $n_1,n_2\in N_G(T)$. Writing $I_1n_2I_1=\sqcup_{x\in I_1/(I_1 \cap n_2^{-1}I_1n_2)} I_1n_2x$, the convolution product of $T_{n_1}$ and  $T_{n_2}$ reads
	\begin{equation}\label{T_{n_1}T_{n_2}}
		(T_{n_1}T_{n_2})(g)=\sum_{x\in I_1/(I_1\cap n_2^{-1}I_1n_2)} T_{n_1}(g x^{-1} (\tilde{n}_2)^{-1}) \text{ for all $g\in \wt{G}$.}
	\end{equation}
	In particular, if $n_2$ or $n_1$ normalizes $I_1$, then $T_{n_1}T_{n_2}=\zeta T_{n_1n_2}$, where $\zeta\in \mu_2$ is such that $(n_1n_2,1)(n_2,1)^{-1}=(n_1,\zeta)$.
\end{remark}	

\begin{definition}
	Let $\Hhat=\Hom(H,\bF_p^{\times})=\Hom(H,k^{\times})$ be the group of $k^{\times}$-valued characters of $H$. For $\chi\in \Hhat$, define the Hecke operator
	\[
	e_{\chi}=\frac{1}{|H|} \sum_{h\in H} \chi(h) T_h.
	\]
\end{definition}

\begin{lem}\label{identities_Hecke_operators_I} The following identities hold true for all characters $\chi,\chi'\in \Hhat$:
	\begin{enumerate}
		\item[{\rm (i)}] $e_{\chi}^2 = e_{\chi}$;
		\item[{\rm (ii)}] $e_{\chi} e_{\chi'} = 0$ if $\chi\neq \chi'$;
		\item[{\rm (iii)}] $\sum_{\chi\in \Hhat} e_{\chi} = 1$;
		\item[{\rm (iv)}] $T_s e_{\chi} = e_{\chi^s}T_s$;
		\item[{\rm (v)}] $T_s^2 e_{\chi} = \begin{cases}
			0 \text{ if } \chi\neq \chi^s\\
			-\chi(\begin{mat}
				1 & 0\\ 0 & -1
			\end{mat}) T_se_{\chi} \text{ if } \chi=\chi^s.
		\end{cases}$
	\end{enumerate}
\end{lem}

\begin{proof}
	The identities can be checked in the endomorphism ring of the induction $\Ind^{\wt{K}}_{\wt{I}_1}(\mathbf{1}\boxtimes \iota)$. Via the splitting $\wt{K}\cong K\times \mu_2$, this is identified with the endomorphism ring of $\Ind^{K}_{I_1}(\mathbf{1})$, so the lemma follows from \cite[Lemma 2.0.10, Lemma 2.0.12]{coeffsystems}. 
\end{proof}

\begin{definition}\label{twisted_char}
	Given a character $\chi\in \Hhat$ and a pair $(i,j)\in \bZ/2\bZ$, define the twisted character $\chi[i,j]\in \Hhat$ by
	\begin{align*}
		\chi[i,j]\left(\begin{mat}
			[\lambda] & 0\\0 & [\mu]
		\end{mat}\right)=\chi\left(\begin{mat}
			[\lambda] & 0\\0 & [\mu]
		\end{mat}\right) (\lambda^{i} \mu^{j})^{\frac{p-1}{2}}, \text{for all $\lambda,\mu\in \bF_p^{\times}$.}
	\end{align*}
\end{definition}

Recall that we put $\Pi=\begin{mat}
	0 &1\\p & 0
\end{mat}\in N_G(T)$, which  normalizes $I_1$. We will make use of Remark \ref{conv_basis} as well as the identities in Example \ref{law_torus} during the proof of the following lemma.

	\begin{lem}\label{identities_Hecke_operators_II} Let $m\in \bZ_{\geq 0}$. The following assertions hold true.
	\begin{enumerate}
		\item[{\rm (i)}] $T_{\Pi}$ is invertible with inverse $T_{\Pi^{-1}}$;
		\item[{\rm (ii)}] $T_{\Pi}^2=T_{\Pi^2}$;
		\item[{\rm (iii)}] $T_{\Pi}^4=(-1)^{\frac{p-1}{2}}T_{\Pi^4}$ is central, and $T_nT_{\Pi^4}=T_{n\Pi^{4}}=T_{\Pi^4n}$ for all $n\in N_G(T)$;
		\item[{\rm (iv)}] $T_{\Pi^2}T_s=T_{\Pi^2 s}=T_{s\Pi^2}=(-1)^{\frac{p-1}{2}}T_sT_{\Pi^2}$;
		\item[{\rm (v)}]  for all $\chi\in \Hhat$, $T_{\Pi} e_{\chi} = e_{\chi^s[1,0]} T_{\Pi}$;
		\item[{\rm (vi)}] $(T_s T_{\Pi})^m = T_{(s\Pi)^m}$;
		\item[{\rm (vii)}]  $(T_{\Pi}T_s)^m=T_{(\Pi s)^m}$;
		\item[{\rm (viii)}] $T_{(s\Pi)^m}T_s=T_{(s\Pi)^m s}=T_sT_{(\Pi s)^m}$;
		\item[{\rm (ix)}] $T_{\Pi} T_{(s\Pi)^m}=T_{\Pi (s\Pi)^m}=T_{(\Pi s)^m\Pi}=T_{(\Pi s)^m}T_{\Pi}$;
	\end{enumerate}
\end{lem}

\begin{proof}
	(i) This follows from $(\Pi,1)^{-1}=(\Pi^{-1},1)$.
	
	(ii) This follows from $(\Pi^2,1)(\Pi^{-1},1)=(1,1)$.
	
	(iii) The first equality follows from $(\Pi^2,1)^2=(\Pi^4,(p,p))=(\Pi^4,(-1)^{\frac{p-1}{2}})$. This element is central element in $\wt{G}$ by Lemma \ref{centers} (i), which implies the other equalities.
	
	(iv) The first equality follows from $(\Pi^2 s,1)(s,1)=(\Pi^2,1)$, the second one is trivial, and the last equality follows from $(s\Pi^2,1)(\Pi^2,1)^{-1}=(s\Pi^2,1)(\Pi^{-2},(-1)^{\frac{p-1}{2}})=(s,(-1)^{\frac{p-1}{2}})$.
	
	(v) Let  $h=\begin{mat}
		[\lambda] & 0\\ 0 & [\mu]
	\end{mat}\in H$ and denote by $h^s=shs$ its conjugate by $s$. The two equalities $(\Pi h,1)(h^{-1},1)=(\Pi,1)$ and $(h^s \Pi,1)(\Pi^{-1},1)=(h^s,\mu^{\frac{p-1}{2}})$ translate into
	\[
	T_{\Pi} T_h = T_{\Pi h}= T_{h^s\Pi}=\mu^{\frac{p-1}{2}} T_{h^s} T_{\Pi},
	\]
	which implies the assertion.
	
	(vi) The case $m=1$ follows from the equality $(s\Pi,1)(\Pi^{-1},1)=(s,1)$. It now suffices to prove $T_{(s\Pi)^m}T_{s\Pi}=T_{(s\Pi)^{m+1}}$ for all $m\geq 1$. Since the support of the product of two operators is contained in the product of the supports (as follows from the definition of the convolution product), the left-hand side has support contained in the preimage in $\wt{G}$ of
	\[
	I_1\begin{mat}
		p^m & 0\\0 & 1
	\end{mat}I_1 \begin{mat}
		p & 0\\0 &1
	\end{mat}I_1=\bigcup_{\lambda\in \bF_p} I_1 \begin{mat}
		p^m & 0\\0 & 1
	\end{mat} \begin{mat}
		1 & [\lambda]\\0 & 1
	\end{mat} \begin{mat}
		p & 0\\0 & 1
	\end{mat}I_1=I_1\begin{mat}
		p^{m+1} &0\\0 & 1
	\end{mat}I_1.
	\]
	The decomposition $I_1 s\Pi I_1 = \sqcup_{\lambda\in \bF_p} I_1 s\Pi \begin{mat}
		1 & 0\\p[\lambda] & 1
	\end{mat}$ yields that the evaluation of $T_{(s\Pi)^m}T_{s\Pi}$ at $((s\Pi)^{m+1},1)$ is equal to
	\[
	\sum_{\lambda\in \bF_p} T_{(s\Pi)^m}\left(((s\Pi)^{m+1},1) \begin{mat}
		1 & 0\\ -p[\lambda] & 1
	\end{mat} (s\Pi,1)^{-1}\right).
	\]
	If $\lambda\neq 0$, then $(s\Pi)^{m+1} \begin{mat}
		1 & 0\\ -p[\lambda] & 1
	\end{mat} (s\Pi)^{-1}$ does not lie in $I_1(s\Pi)^mI_1$, so only the summand corresponding to $\lambda=0$ contributes and gives the value $1$.
	
	(vii) This is analogous to (vi).
	
	(viii) Since $I_1 (s\Pi)^m I_1 sI_1=I_1(s\Pi)^m sI_1$, it is enough to evaluate at $((s\Pi)^m s,1)$. Now $I_1 sI_1=\sqcup_{\lambda\in \bF_p} I_1 s\begin{mat}
		1 & [\lambda]\\0 & 1
	\end{mat}$, so
	\[
	T_{(s\Pi)^m}T_s(((s\Pi)^m s,1))=\sum_{\lambda\in \ol{\bF}_p} T_{(s\Pi)^m}\left(((s\Pi)^ms,1)\begin{mat}
		1 & [\lambda]\\0 & 1
	\end{mat}(s,1)\right).
	\]
	If $\lambda\neq 0$, then $(s\Pi)^m \begin{mat}
		1 & 0\\ [\lambda] & 1
	\end{mat}$ does not lie in $I_1 (s\Pi)^m I_1$, so only the summand corresponding to $\lambda=0$ contributes and gives the value $1$. This proves the first equality. The second one is proved similarly.
	
	(ix) The first identity follows from
	\[
	(\Pi (s\Pi)^m,1)((s\Pi)^{-m},1)=(\Pi,(p^{-m},p)(-p^{m+1},p^{-m}))=(\Pi,1),
	\]
	the second one is trivial, the last one follows from $((\Pi s)^m \Pi,1)(\Pi^{-1},1)=((\Pi s)^m,1)$.
\end{proof}

\begin{cor}\label{TPi_iso-cInd}
	The operator $T_{\Pi}$ restricts to a $\wt{G}$-equivariant isomorphism
	\[
	T_{\Pi}\colon \cInd^{\wt{G}}_{\wt{I}}(\chi\boxtimes \iota)\xrightarrow{\cong} \cInd^{\wt{G}}_{\wt{I}}(\chi^s[1,0]\boxtimes \iota),
	\]
	for each $\chi\in \Hhat$.
\end{cor}

\begin{proof}
	By definition, $T_{\Pi}$ is a $\wt{G}$-equivariant endomorphism on $\cInd^{\wt{G}}_{\wt{I}_1}(\mathbf{1}\boxtimes \iota)$. The assertion follows from Lemma \ref{identities_Hecke_operators_II} (v) and (i), noting that $ \cInd^{\wt{G}}_{\wt{I}}(\chi\boxtimes \iota)=e_{\chi}\cInd^{\wt{G}}_{\wt{I}_1}(\mathbf{1}\boxtimes \iota)$.
\end{proof}

	\subsection{Decomposing the pro-$p$ Iwahori Hecke algebra} It follows from Lemma \ref{identities_Hecke_operators_I} that we have the decomposition of $\wt{G}$-representations
\begin{equation}\label{decomp_cInd}
	\cInd^{\wt{G}}_{\wt{I}_1}(\mathbf{1}\boxtimes \iota)=\bigoplus_{\chi\in \Hhat} e_{\chi} \cInd^{\wt{G}}_{\wt{I}_1}(\mathbf{1}\boxtimes \iota)=\bigoplus_{\chi\in \Hhat} \cInd^{\wt{G}}_{\wt{I}}(\chi\boxtimes \iota).
\end{equation}

As we will see below, there are non-zero intertwinings between two direct summands if and only if the two respective characters are conjugate in the following sense: The group $\wt{N_G(T)}$ acts on the set of genuine characters of $\wt{H}$ by conjugation. Since $\mu_2$ is central, this action descends to $N_G(T)$. Given $\chi\in \Hhat$ and $n\in N_G(T)$, we let $\chi^{\tilde{n}}\in \Hhat$ be the unique character satisfying $\chi^{\tilde{n}}\boxtimes \iota = (\chi\boxtimes \iota)^{\tilde{n}}$ as genuine characters on $\wt{H}=H\times \mu_2$. This depends only on the image of $n$ in $N_G(T)/(T\cap I)$ and defines an action of the latter group on $\Hhat$. In particular, the subgroup generated by $\Pi$ acts on $\Hhat$.

\begin{definition}\label{definition_orbit}
	Let $\chi\in \Hhat$.
	\begin{enumerate}
		\item[{\rm (i)}] Define $O_{\chi}=\left\{\chi^{\tilde{w}} : w\in N_G(T)/(T\cap I)\right\}$ to be the \textit{orbit of $\chi$}. Denote the set of all orbits by $\mathcal{O}$.
		\item[{\rm (ii)}] The \textit{$\Pi$-orbit of $\chi$} is the orbit of $\chi$ under the action of the subgroup generated by $\Pi$, i.e.\ it is equal to $\chi^{\tilde{\Pi}^{\bZ}}$.
	\end{enumerate}
\end{definition}

	Recall from Definition \ref{twisted_char} that to $\chi$ and each pair $(i,j)\in (\bZ/2\bZ)^2$ we have associated a twisted character $\chi[i,j]$.

\begin{remark}\label{Pi_equiv_Remark}
	For $h=\begin{mat}
		[\lambda] & 0\\ 0 &[\mu]
	\end{mat}\in H$, one has $\tilde{\Pi} \tilde{h}(\tilde{\Pi})^{-1}=(h^s,\mu^{\frac{p-1}{2}})$, i.e.\ for $\chi\in \Hhat$, $\chi^{\tilde{\Pi}}=\chi^s[0,1]$. Thus, the $\Pi$-orbit of $\chi$ is equal to
	\[
	\chi^{\wt{\Pi}^{\bZ}}=\left\{\chi,\chi^s[0,1],\chi[1,1],\chi^s[1,0]\right\}.
	\]
\end{remark}

	\begin{lem}\label{orbit-description}
	Let $\chi\in \Hhat$.
	\begin{enumerate}
		\item[{\rm (i)}] If $(\chi^{-1}\chi^s)^2  \neq 1$, then the map
		\begin{align*}
			\bZ/2\bZ \times (\bZ/2\bZ)^2&\cong O_{\chi}\\
			(a,(i,j)) & \mapsto \chi^{s^{a}}[i,j]
		\end{align*}
		is bijective.
		\item[{\rm (ii)}] If $(\chi^{-1}\chi^s)^2 = 1$, then
		the map
		\begin{align*}
			(\bZ/2\bZ)^2&\cong O_{\chi}\\
			(i,j) & \mapsto \chi[i,j]
		\end{align*}
		is bijective.
	\end{enumerate}
\end{lem}

\begin{proof} The group $N_G(T)/(T\cap I)$ is generated by the elements $\Pi$ and $s$. Remark \ref{Pi_equiv_Remark} gives a description of the $\Pi$-orbit. Since $(\chi\boxtimes \iota)^{\tilde{s}}=\chi^s\boxtimes \iota$, we deduce that the map in (i) is surjective. Since we cannot have $\chi = \chi[i,j]$ for $0 \neq (i,j)\in (\bZ/2\bZ)^2$, the map in (ii) is injective. If additionally $(\chi^{-1}\chi^s)^2 \neq 1$, then we can neither have $\chi=\chi^s[i,j]$ for any non-zero pair $(i,j)$, showing that the map in (i) is also injective. If $(\chi^{-1}\chi^s)^2=1$, we may choose $t\in \bZ/2\bZ$ such that $\chi^{-1}\chi^s\left(\begin{mat} [\lambda] & 0\\0 & [\mu]\end{mat}\right) = (\lambda^{-1}\mu)^{t\frac{p-1}{2}}$ for all $\lambda,\mu\in \bF_p^{\times}$, i.e.\ $\chi^s=\chi[t,t]$, proving that the map in (ii) is surjective.
\end{proof}

	\begin{definition}
	Let $\chi\in \Hhat$.
	\begin{enumerate}
		\item[{\rm (i)}] We say that $\chi$ is \textit{square-regular} if $(\chi^{-1}\chi^s)^2 \neq 1$.
		\item[{\rm (ii)}] We say that $\chi$ is \textit{non-square-regular} if $(\chi^{-1}\chi^s)^2 = 1$, i.e.\ $\chi^s=\chi[t,t]$ for some $t\in \bZ/2\bZ$.
		\item[\rm (iii)] We say that an orbit $O\in \mathcal{O}$ is (non-)square-regular if one, or equivalently every, character in it is so.
	\end{enumerate}
\end{definition}

\begin{definition}
	\begin{enumerate}
		\item[{\rm (i)}] For an orbit $O\in \mathcal{O}$, define the \textit{Hecke algebra of $O$} by
		\[
		\mathcal{H}(O)=\End_{\wt{G}}(\bigoplus_{\chi\in O} \cInd^{\wt{G}}_{\wt{I}}(\chi\boxtimes \iota)).
		\]
		\item[{\rm (ii)}]
		For $\chi_1,\chi_2\in \Hhat$, put
		\[
		\mathcal{H}(\chi_1,\chi_2)=\End_{\wt{G}}(\cInd^{\wt{G}}_{\wt{I}}(\chi_1\boxtimes \iota),\cInd^{\wt{G}}_{\wt{I}}(\chi_2\boxtimes \iota)).
		\]
		If $\chi_1=\chi_2$, we simply write $\mathcal{H}(\chi_1)$.
	\end{enumerate}
\end{definition}

	\begin{prop}
	The decomposition (\ref{decomp_cInd}) induces a direct sum decomposition of $k$-algebras
	\[
	\mathcal{H}=\bigoplus_{O\in \mathcal{O}} \mathcal{H}(O).
	\]
\end{prop}

\begin{proof}
	We have to show that $\Hom_{\wt{G}}(\cInd^{\wt{G}}_{\wt{I}}(\chi_1\boxtimes \iota), \cInd^{\wt{G}}_{\wt{I}}(\chi_2\boxtimes \iota))= 0$ if $\chi_1$ and $\chi_2$ lie in distinct orbits. This can be proved as in \cite[p.\ 9]{Vigneras}: For $n\in N_G(T)$, let $I_n=I\cap n^{-1}In$. Since $I=I_1\rtimes H$, we have $I=I_1 I_n$. Using the decomposition $\wt{G}=\sqcup_{n\in N_G(T)/(T\cap I)} \wt{I} \tilde{n}\wt{I}$, Mackey's formula reads
	\[
	\cInd^{\wt{G}}_{\wt{I}}(\chi_2\boxtimes \iota)|_{\wt{I}}\cong \bigoplus_{n\in N_G(T)/(T\cap I)} \Ind^{\wt{I}}_{\wt{I}_n}((\chi_2\boxtimes \iota)^{\tilde{n}}).
	\]
	Therefore, using Frobenius reciprocity twice, the non-vanishing of the Hom-space of interest is equivalent to $(\chi_1\boxtimes \iota)|_{\wt{I}_n} = (\chi_2\boxtimes \iota)^{\tilde{n}}|_{\wt{I}_n}$ for some $n\in N_G(T)$. Smooth characters are trivial on the pro-$p$ group $I_1$, so the equality $\wt{I}=I_1\wt{I}_n$ completes the proof.
\end{proof}

	We finish this section with two lemmas, the first of which will allow us to reduce statements in the non-square-regular case to a consideration of the trivial character, while the second one is a computational preliminary for the upcoming sections, where we determine the Hecke algebra of an orbit explicitly.

\begin{lem}\label{twist}
	Let $\psi\colon \bQ_p^{\times}\to k^{\times}$ be a smooth character. Then the map
	\begin{align*}
		\cInd^{\wt{G}}_{\wt{I}_1}(\mathbf{1}\boxtimes \iota)\otimes \psi\circ \det &\xrightarrow{\cong} 	\cInd^{\wt{G}}_{\wt{I}_1}(\mathbf{1}\boxtimes \iota)\\
		f &\mapsto [(g,\zeta)\mapsto \psi(\det(g))f(g,\zeta)]
	\end{align*}
	is a $\wt{G}$-equivariant isomorphism. The resulting $k$-algebra automorphism 
	\[
	\operatorname{tw}_{\psi}\colon \cH=\End_{\wt{G}}(	\cInd^{\wt{G}}_{\wt{I}_1}(\mathbf{1}\boxtimes \iota)\otimes \psi\circ \det) \cong \cH,
	\]
	where the equality is induced by the functor $(-)\otimes \psi\circ \det$, identifies $\cH(O_{\chi})$ with $\cH(O_{\chi \psi|_H})$.
\end{lem}

\begin{proof}
	Checking that the map is well-defined $\wt{G}$-equivariant is straight forward. Replacing $\psi$ by its inverse shows that the map is an isomorphism. The rest is clear.
\end{proof}

\begin{lem}\label{Hecke_Lemma}
	Let $r,s\in \bZ$ and let $\zeta\in \mu_2$.
	\begin{enumerate}
		\item[{\rm (i)}] Let $x=\begin{mat}
			a & b\\ c & d
		\end{mat}\in I$. Then $y=\begin{mat}
			a & p^{s-r}b\\ p^{r-s}c & d
		\end{mat}\in I$ if and only if $s-r+v(b)\geq 0$ and $r-s+v(c)\geq 1$. In this case, the equality
		\[
		(x,\zeta)\left(\begin{mat}
			p^r & 0\\ 0 & p^s
		\end{mat},1\right) = \left(\begin{mat}
			p^r & 0\\ 0 & p^s
		\end{mat},1\right)(y,1)
		\]
		holds if and only if
		\[
		\zeta= \begin{cases}
			\omega(\det(x))^{r\frac{p-1}{2}} \text{ if } c \neq 0\\
			\omega\left(\det(x)^{s} d^{r-s}\right)^{\frac{p-1}{2}} \text{ if } c=0.
		\end{cases}
		\]
		
		\item[{\rm (ii)}] Let $x=\begin{mat}
			a & b\\ c & d
		\end{mat}\in I$. Then $y=\begin{mat}
			d & p^{r-s}c\\ p^{s-r}b & a
		\end{mat}\in I$ if and only if $r-s+v(c)\geq 0$ and $s-r+v(b)\geq 1$. In this case, the equality
		\[
		(x,\zeta)\left(\begin{mat}
			0 & p^{r}\\ p^{s} & 0
		\end{mat},1\right)=\left(\begin{mat}
			0 & p^{r}\\ p^{s} & 0
		\end{mat},1\right)(y,1)
		\]
		holds if and only if
		\[
		\zeta = \begin{cases}
			\omega\left(d^{v(b)+v(c)} \det(x)^{s-v(c)}\right)^{\frac{p-1}{2}} \text{ if } c\neq 0, b\neq 0\\
			\omega\left(d^{r-s+v(c)} \det(x)^{s-v(c)}\right)^{\frac{p-1}{2}} \text{ if } c\neq 0, b=0\\
			\omega\left(d^{v(b)} \det(x)^{s}\right)^{\frac{p-1}{2}} \text{ if } c=0, b\neq 0\\
			\omega\left(d^{r-s} \det(x)^{s}\right)^{\frac{p-1}{2}} \text{ if } c=0,b=0.
		\end{cases}
		\]
	\end{enumerate}
\end{lem}

\begin{proof}
	This is straight forward.
\end{proof}

\subsection{The Hecke algebra of a regular orbit} In this section, we fix a square-regular orbit $O$ and a character $\chi\in O$. Define
\[
\cH(\chi\oplus \chi[1,0])=\End_{\wt{G}}(\bigoplus_{t=0,1} \cInd^{\wt{G}}_{\wt{I}}(\chi[t,0]\boxtimes \iota))=\begin{pmatrix}
	\cH(\chi) & \cH(\chi[1,0],\chi)\\
	\cH(\chi,\chi[1,0]) & \cH(\chi[1,0])
\end{pmatrix}.
\]
By Lemma \ref{orbit-description}, Remark \ref{Pi_equiv_Remark} and Corollary \ref{TPi_iso-cInd}, we have an isomorphism of algebras $\cH(O)\cong M_{4\times 4}(\cH(\chi\oplus \chi[1,0]))$ and so:

\begin{cor}\label{Morita_regular}
	The functor 
	\begin{align*}
		\Mod_{\cH(O)}&\xrightarrow{\sim} \Mod_{\cH(\chi\oplus \chi[1,0])}\\
		M&\mapsto M(e_{\chi}+e_{\chi[1,0]})
	\end{align*}
	is an equivalence of categories.
\end{cor}

\noindent In order to understand the simple objects on the left-hand side of the equivalence, it therefore suffices to study the algebra $\cH(\chi\oplus \chi[1,0])$ and its simple modules.\\

For $t_1,t_2\in \bZ/2\bZ$, we view the Hecke module $\mathcal{H}(\chi[t_1,0],\chi[t_2,0])$ as the space
\[
\left\{\varphi\colon \wt{G}\to k : \begin{array}{l}\diamond \hspace{0.1cm} \varphi((x,\zeta_x)g(y,\zeta_y))=(\chi[t_2,0]\boxtimes \iota)((x,\zeta_x)) \varphi(g)(\chi[t_1,0]\boxtimes \iota)((y,\zeta_y))\\ \phantom{\diamond}\text{ for all } g\in \wt{G}, x,y\in I,\zeta_x,\zeta_y\in \mu_2\\
	\diamond \hspace{0.1cm} |\wt{I}\setminus \supp(\varphi)/\wt{I}|<\infty\end{array}\right\}
\]
with left $\mathcal{H}(\chi[t_2,0])$-, resp.\ right $\mathcal{H}[t_1,0])$-action given by the convolution product. As a consequence of Lemma \ref{Hecke_Lemma} and the regularity assumption on $\chi$, we obtain:

\begin{cor}\label{basis_regular_Hecke}
	For $t_1,t_2\in \bZ/2\bZ$, $\mathcal{H}(\chi[t_1,0],\chi[t_2,0])$ has $k$-basis $\{T_{i,t_1-t_2+j}\}_{i,j\in 2\bZ}$, where $T_{i,t_1-t_2+j}$ is uniquely determined by
	\[
	\supp(T_{i,t_1-t_2+j})=\wt{I}\left(\begin{mat}
		p^{i} &0\\0 &p^{t_1-t_2+j}
	\end{mat},1\right)\wt{I}, \hspace{0.2cm} T_{i,t_1-t_2+j}\left(\left(\begin{mat}
		p^{i} &0\\0 &p^{t_1-t_2 + j}
	\end{mat},1\right)\right)=1.
	\]
\end{cor}

\begin{prop}\label{Hecke_algebra_regular_char}
	For each $t\in \bZ/2\bZ$, the map
	\begin{align*}
		k[X,Y,Z^{\pm 1}]/(XY)&\xrightarrow{\cong} \mathcal{H}(\chi[t,0])\\
		Z&\mapsto T_{2,2}\\
		Z^{-1} & \mapsto T_{-2,-2}\\
		X&\mapsto T_{2,0}\\
		Y&\mapsto T_{0,2}
	\end{align*}
	is a well-defined $k$-algebra isomorphism. In particular, the right-hand side is commutative.
\end{prop}

\begin{proof}
	The proposition follows from Corollary \ref{basis_regular_Hecke} and the three points
	\begin{enumerate}
		\item[{\rm (a)}] $T_{2,2}^M T_{i,j} =T_{i+2M,j+2M}=T_{i,j}T_{2,2}^M$ for all $i,j\in 2\bZ$ and $M\in \bZ$;
		\item[{\rm (b)}] $T_{2,0}^m =T_{2m,0}$ and $T_{0,2}^m=T_{0,2m}$ for all $m\in \bZ_{\geq 0}$;
		\item[{\rm (c)}] $T_{2,0}T_{0,2}=0$;
	\end{enumerate}
	which we quickly explain:
	
	\begin{enumerate}
		\item[{\rm (a)}] follows from Lemma \ref{identities_Hecke_operators_II} (iii);
		
		\item[{\rm (b)}] follows from Lemma \ref{identities_Hecke_operators_II} (vi) and (vii);
		
		\item[{\rm (c)}] follows from Lemma \ref{identities_Hecke_operators_II} (vi), (vii) (applied to $m=1$) and Lemma \ref{identities_Hecke_operators_I} (v).
	\end{enumerate}
\end{proof}

	Fix now integers $0\leq t_1,t_2\leq 1$. We want to compute $\mathcal{H}(\chi[t_1,0],\chi[t_2,0])$ as a left $\mathcal{H}(\chi[t_2,0])$- and right $\mathcal{H}(\chi[t_1,0])$-module, respectively. In view of the previous proposition, we may assume that $t_1\neq t_2$.

\begin{prop}\label{Hecke_module_regular_char}
	For $0\leq t_1\neq t_2\leq 1$, the maps of left $\cH(\chi[t_2,0])$-, resp.\ right $\cH(\chi[t_1,0])$-modules,
	\begin{align*}
		\cH(\chi[t_2,0])/(T_{2,0})\oplus \cH(\chi[t_2,0])/(T_{0,2})&\xrightarrow{\cong} \cH(\chi[t_1,0],\chi[t_2,0])\\
		(1,0)&\mapsto T_{0,1}\\
		(0,1)&\mapsto T_{2,1}
	\end{align*}
	and
	\begin{align*}
		\cH(\chi[t_1,0])/(T_{2,0})\oplus \cH(\chi[t_1,0])/(T_{0,2})&\xrightarrow{\cong} \cH(\chi[t_1,0],\chi[t_2,0])\\
		(1,0)&\mapsto T_{0,1}\\
		(0,1)&\mapsto T_{2,1}
	\end{align*}
	are well-defined isomorphisms.
\end{prop}

\begin{proof}
	By Corollary \ref{basis_regular_Hecke}, the space $\cH(\chi[t_1,0],\chi[t_2,0])$ has $k$-basis $\{T_{i,1+j}\}_{i,j\in 2\bZ}$. The proposition now follows from the identities, for $m\in \bZ_{\geq 0}$,
	\begin{enumerate}
		\item[{\rm (a)}] $T_{2,0}T_{0,1}=0=T_{0,1}T_{2,0}$;
		\item[{\rm (b)}] $T_{0,2}T_{2,1}=0=T_{2,1}T_{0,2}$;
		\item[{\rm (c)}] $T_{2,2}T_{i,1+j}=T_{i+2,1+j+2}=T_{i,1+j}T_{2,2}$ for all $i,j\in 2\bZ$;
		\item[{\rm (d)}] $T_{0,2m}T_{0,1}=T_{0,1+2m}=T_{0,1}T_{0,2m}$;
		\item[{\rm (e)}] $T_{2m,0}T_{2,1}=T_{2(m+1),1}=T_{2,1}T_{2m,0}$.
	\end{enumerate}
	To check these, write $T_{0,1+2m}=T_{(\Pi s)^{1+2m}}e_{\chi[t_1,0]}$, $T_{2(m+1),1}=T_{\Pi^2 (s\Pi)^{2m+1}}e_{\chi[t_1,0]}$ and, for $t\in \{t_1,t_2\}$, $T_{2,2}=T_{\Pi^4}e_{\chi[t,0]}$, $T_{2m,0}=T_{(s\Pi)^{2m}}e_{\chi[t,0]}$, $T_{0,2m}=T_{(\Pi s)^{2m}}e_{\chi[t,0]}$. The points (a)-(e) now follow from Lemma \ref{identities_Hecke_operators_II}, and Lemma \ref{identities_Hecke_operators_I} (v) to prove (a) and (b).
\end{proof}

\begin{remark}\label{left_right_agree}
	The points (a)-(e) show that the left action of $\cH(\chi[t_1,0])$ agrees with the right action of $\cH(\chi[t_2,0])$ on $\cH(\chi[t_1,0],\chi[t_2,0])$ under the isomorphism $\cH(\chi[t_1,0])\cong \cH(\chi[t_2,0])$ resulting from Proposition \ref{Hecke_algebra_regular_char}. Moreover, as a consequence of the proposition, we obtain that $\cH(\chi[1,0],\chi)\cong \cH(\chi,\chi[1,0])$ as modules over the commutative ring $\cH(\chi)\cong \cH(\chi[1,0])$.
\end{remark}

Recall that we defined
$\cH(\chi\oplus \chi[1,0])=\begin{mat}
	\cH(\chi) & \cH(\chi[1,0],\chi)\\
	\cH(\chi,\chi[1,0]) & \cH(\chi[1,0])
\end{mat}$. By Proposition \ref{Hecke_algebra_regular_char}, the $k$-algebras on the diagonal may be identified with $k[X,Y,Z^{\pm 1}]/(XY)$.

\begin{cor}\label{center_reg}
	The center $Z(\cH(\chi\oplus \chi[1,0]))$ of $\cH(\chi\oplus \chi[1,0])$ is isomorphic to the $k$-algebra $k[X,Y,Z^{\pm 1}]/(XY)$ via the diagonal embedding.
\end{cor}

\begin{proof}
	Considering elementary diagonal matrices, an element $b\in Z(\cH(\chi\oplus \chi[1,0]))$ needs to be diagonal, say $b=\operatorname{diag}(P_0,P_{1})$, where each $P_t\in k[X,Y,Z^{\pm 1}]/(XY)$ with $Z=T_{2,2},X=T_{2,0},Y=T_{0,2}\in \mathcal{H}(\chi[t,0])$. 
	Letting $\mathbf{T}_{0,1}=\begin{mat}
		0 & T_{0,1}\\0 & 0
	\end{mat}$ and $\mathbf{T}_{2,1}=\begin{mat}
		0 & T_{2,1}\\0 & 0
	\end{mat}$,
	the equality $b\mathbf{T}_{0,1}=\mathbf{T}_{0,1}b$ forces $P_1(0,Y,Z)=P_0(0,Y,Z)$, while the equality $b\mathbf{T}_{2,1}=\mathbf{T}_{2,1}b$ forces $P_1(X,0,Z)=P_0(X,0,Z)$. Since $XY=0$, this implies $P_1=P_0$. It remains to show that such elements indeed lie in the center, which follows from Remark \ref{left_right_agree}.
\end{proof}

	We now compute to which extend the off-diagonal entries contribute to the diagonal.

\begin{lem} \label{diagonal_contribution}
	Let $0\leq t_1\neq t_2 \leq 1$. The image of the composition map
	\[
	\mathcal{H}(\chi[t_2,0],\chi[t_1,0])\otimes_{\mathcal{H}(\chi[t_2,0])} \mathcal{H}(\chi[t_1,0],\chi[t_2,0]) \to \mathcal{H}(\chi[t_1,0])
	\]
	is equal to the ideal $(T_{2,0},T_{0,2})$ generated by $T_{2,0}$ and $T_{0,2}$. More precisely,
	\begin{itemize}
		\item $T_{0,1}\circ T_{2,1}=0=T_{2,1}\circ T_{0,1}$;
		\item $T_{0,1}\circ T_{0,1}=T_{0,2}$;
		\item $T_{2,1}\circ T_{2,1}=T_{4,2}=T_{2,2}T_{2,0}$.
	\end{itemize}
\end{lem}

\begin{proof}
	By Proposition \ref{Hecke_module_regular_char}, we have
	\[
	\mathcal{H}(\chi[t_2,0],\chi[t_1,0])=\mathcal{H}(\chi[t_1,0])T_{0,1} + \mathcal{H}(\chi[t_1,0])T_{2,1}
	\]
	and
	\[
	\mathcal{H}(\chi[t_1,0],\chi[t_2,0])=T_{0,1}\mathcal{H}(\chi[t_1,0])+T_{2,1}\mathcal{H}(\chi[t_1,0]).
	\]
	Since the composition map is $\mathcal{H}(\chi[t_1,0])$-equivariant from both sides, it is enough to compute how the generators compose, i.e.\ it is enough to check the claimed identities. These follow from Lemma \ref{identities_Hecke_operators_II} and Lemma \ref{identities_Hecke_operators_I} (v) by writing the generators $T_{0,1}$ and $T_{2,1}$ as in the proof of Proposition \ref{Hecke_module_regular_char}.
\end{proof}

\subsubsection{Simple modules} In order to determine the simple modules, we first show that the center acts via scalars in the coefficient field $k$ and then study those with a fixed central character.

\begin{lem}\label{simple_fd_reg}
	Any simple right $\cH(\chi\oplus \chi[1,0])$-module admits a central character.
\end{lem}

\begin{proof}
	To ease notation, put $R=\cH(\chi\oplus \chi[1,0])$, and identify the center of $R$ with $k[X,Y,Z^{\pm 1}]/(XY)$ as in Corollary \ref{center_reg}.
	Let $M$ be a simple right $R$-module. The kernel of the central action map $k[X,Y,Z^{\pm 1}]/(XY)\to \End_R(M)$ must be a prime ideal. In particular, it contains $X$ or $Y$ and we may without loss of generality assume that it contains $X$. Thus, the action factors through the localized polynomial ring $(k[Y])[Z^{\pm 1}]$ over the PID $k[Y]$. The prime ideals in this ring are
	\begin{enumerate}
		\item[{\rm (a)}] $(0)$;
		\item[{\rm (b)}] $(g(Z))$ for an irreducible polynomial $g(Z)\in (k[Y])[Z^{\pm 1}]$;
		\item[{\rm (c)}] $(Y-y, Z-z)$ for $y\in k$ and $z\in k^{\times}$.
	\end{enumerate}
	We need to exclude that the kernel of the map $(k[Y])[Z^{\pm 1}]\to \End_R(M)$ falls into (a) or (b). In case (a) and (b), $M$ is naturally a $C=k(Y)$-vector space. Since $T_{2,0}=X=0$ on $M$, the last relation in Lemma \ref{diagonal_contribution} implies that $\begin{mat}
		0 & T_{2,1}\\T_{2,1} & 0
	\end{mat}^2$ kills $M$. In particular, we may choose some $0\neq v\in M$ with $v\begin{mat}
	0 & T_{2,1}\\T_{2,1} & 0
\end{mat}=0$. We may assume that $ve_{\chi[t,0]}=v$ for either $t=0$ or $t=1$. By Lemma \ref{diagonal_contribution}, the $C$-subvector space 
\begin{equation}\label{sub}
C[Z^{\pm 1}]v + C[Z^{\pm 1}]v\begin{mat}
	0 & T_{0,1}\\ T_{0,1} & 0
\end{mat}\subset M.
\end{equation}
is stable under the action of $R$ and must therefore be equal to $M$. In particular, $M$ is finitely generated as a module over $C[Z^{\pm 1}]$ and so the $C[Z^{\pm 1}]$-module structure on $M$ cannot extend to $C(Z)$, which excludes (a). In case (b), replacing $C$ by $k[Y]$ in (\ref{sub}) yields an $R$-stable subspace and we deduce that $M$ is finitely generated over $k[Y]$, which is impossible since $M$ is a $C$-vector space by assumption.
\end{proof}

Fix now a central character $\xi\colon k[X,Y,Z^{\pm 1}]/(XY)\to k$ corresponding to the $k$-rational point $(x,y,z)$ and put
\[
\cH(\chi\oplus \chi[1,0])_{\xi}=\cH(\chi\oplus \chi[1,0])\otimes_{Z(\cH(\chi\oplus \chi[1,0])),\xi} k.
\]

\begin{prop}\label{simple_regular_ss}
	Assume that $x=0=y$. Up to isomorphism, there exist precisely two simple right $\cH(\chi\oplus \chi[1,0])_{\xi}$-modules, namely the $1$-dimensional $\cH(\chi\oplus \chi[1,0])_{\xi}$-algebras
	\[
	\operatorname{pr}_t\colon \cH(\chi\oplus \chi[1,0])_{\xi} \twoheadrightarrow \cH(\chi[t,0])/(T_{2,0},T_{0,2},T_{2,2}-z)\cong k, \hspace{0.1cm} \text{ for } t=0,1,
	\]
	induced by projecting to the diagonal entry in $\mathcal{H}(\chi[t,0])$.
\end{prop}

\begin{proof}
	It follows from Lemma \ref{diagonal_contribution} and the assumption on $x$ and $y$ that the projection maps $\operatorname{pr}_t$ are indeed $k$-algebra homomorphisms. Since the respective target is $1$-dimensional over $k$, they define simple $\cH(\chi\oplus \chi[1,0])_{\xi}$-modules. Moreover, these are pairwise non-isomorphic as the kernels are distinct two-sided maximal ideals.
	
	For showing that these are the only simple right modules, we make use of propositions \ref{Hecke_algebra_regular_char} and \ref{Hecke_module_regular_char} to identify
	\begin{equation}\label{H(O/Pi)_xi}
		\cH(\chi\oplus \chi[1,0])_{\xi}\cong \begin{pmatrix}
			k & k\oplus k\\ k \oplus k & k
		\end{pmatrix}
	\end{equation}
	with multiplication $\begin{pmatrix}
		a & b\\ c & d
	\end{pmatrix} \begin{pmatrix}
		a' & b'\\ c' & d'
	\end{pmatrix}=\begin{pmatrix}
		a a' & ab'+d'b\\
		a'c+dc' & dd'
	\end{pmatrix}$, where $k$ acts on $k\oplus k$ via the usual componentwise scalar multiplication. Let now $I\subset \cH(\chi\oplus \chi[1,0])_{\xi}$ be a right ideal neither contained in $\ker(\operatorname{pr}_0)$ nor in $\ker(\operatorname{pr}_1)$. We claim that $I$ is the whole ring. Indeed, by assumption $I$ contains elements $E_0,E_1$, where the $(t+1)$-th entry on the diagonal in $E_t$ is non-zero. By scaling we may assume that this non-zero entry is equal to $1$. Since $I$ is a right ideal, we may further assume that $E_t$ contains only one non-zero column for $t=0,1$. Hence, $E_0+E_1$ is of the form $\begin{pmatrix}
		1 & b\\c & 1
	\end{pmatrix}$ for some $b,c\in k \oplus k$. Since $\begin{pmatrix}
		1 & b\\c & 1
	\end{pmatrix}\begin{pmatrix}
		1 & -b\\-c & 1
	\end{pmatrix}=\begin{pmatrix}
		1 & 0\\0 & 1
	\end{pmatrix}$, we are done.
\end{proof}

\begin{prop}\label{simple_regular_principal} Assume that $x\neq 0$ or $y\neq 0$. Up to isomorphism, there exists a unique simple right $\cH(\chi\oplus \chi[1,0])_{\xi}$-module, namely the $2$-dimensional space $kv\oplus kw$ uniquely determined by $ve_{\chi}=v$ and $we_{\chi[1,0]}=w$.
\end{prop}

\begin{proof}
	Consider the idempotent $e_{\chi}=\begin{pmatrix}
		1 & 0\\0 & 0
	\end{pmatrix}$ in $\cH(\chi\oplus \chi[1,0])_{\xi}$. Since $x$ or $y$ is non-zero, Lemma \ref{diagonal_contribution} implies that $\cH(\chi\oplus \chi[1,0])_{\xi}e_{\chi}\cH(\chi\oplus \chi[1,0])_{\xi}=\cH(\chi\oplus \chi[1,0])_{\xi}$, i.e.\ $e_{\chi}$ is a full idempotent. Hence, $\cH(\chi\oplus \chi[1,0])_{\xi}$ is Morita equivalent to the field
	\[
	e_{\chi}\cH(\chi\oplus \chi[1,0])_{\xi}e_{\chi}=\cH(\chi)\otimes_{k[X,Y,Z^{\pm 1}]/(XY),\xi} k\cong k
	\]
	and therefore admits unique simple right module. The same works with $e_{\chi}$ replaced by $e_{\chi[1,0]}$. Hence, if $M$ denotes the unique simple right $\cH(\chi\oplus \chi[1,0])_{\xi}$-module, we have $M=Me_{\chi}\oplus Me_{\chi[1,0]}$ with each summand being $1$-dimensional.
\end{proof}

Under the isomorphism $Z(\cH(O))\cong Z(\cH(\chi\oplus \chi[1,0]))\cong k[X,Y,Z^{\pm 1}]/(XY)$ resulting from the Morita equivalence Corollary \ref{Morita_regular}, the elements $X,Y,Z$ correspond to $T_{(s\Pi)^2},T_{(\Pi s)^2}, T_{\Pi^4}$, respectively.

	\begin{definition}\label{def_simple_reg}
	\begin{enumerate}
		\item[{\rm (i)}] We say that a simple right $\cH(O)$-module is \textit{supersingular} if it is killed by the central elements $T_{(s\Pi)^2}$ and $T_{(\Pi s)^2}$. Otherwise, we say that the module is \textit{principal}.
		\item[{\rm (ii)}] For $\chi\in O$ and $z\in k^{\times}$, denote by $\operatorname{SS}(\chi,z)$ the unique supersingular $\cH(O)$-module with $\operatorname{SS}(\chi,z)e_{\chi}\neq 0$ and the central element $T_{\Pi^{4}}$ acting by the scalar $z$.
		\item[{\rm (iii)}] For $\chi\in O$ and $(x,y,z)\in k^{3}$ with $xy=0\neq z$ and $x\neq 0$ or $y\neq 0$, denote by $\operatorname{PS}(\chi,x,y,z)$ the unique principal $\cH(O)$-module with $\operatorname{PS}(\chi,x,y,z)e_{\chi}\neq 0$ such that $(T_{(s\Pi)^2},T_{(\Pi s)^2},T_{\Pi^4})$ act on $\operatorname{PS}(\chi,x,y,z)e_{\chi}$ via the scalars $(x,y,z)$.
	\end{enumerate}
\end{definition}

\begin{remark}\label{inter}
	The only intertwinings between the simple right $\cH(O)$-modules are $\operatorname{SS}(\chi,z)\cong \operatorname{SS}(\chi^s[1,0],z)$ and
	\[
	\operatorname{PS}(\chi[1,0],x,y,z)\cong \operatorname{PS}(\chi,x,y,z)\cong \operatorname{PS}(\chi^s[1,0],y,x,z)
	\]
	for $\chi\in O$. For the last isomorphism note that $T_{\Pi} (T_sT_{\Pi})^2 T_{\Pi}^{-1}=(T_{\Pi}T_s)^2$ and $T_{\Pi} (T_{\Pi}T_s)^2T_{\Pi}^{-1}=(T_sT_{\Pi})^2$, which follows from Lemma \ref{identities_Hecke_operators_II} (ii), (iv).
\end{remark}

\subsection{The Hecke algebra of a non-regular orbit}\label{Hecke_non_reg} In this section, we fix a non-square-regular orbit $O$ and a character $\chi\in O$. Possibly replacing $\chi$ by $\chi^s[0,1]$, we may assume that $\chi=\chi^s$. Twisting appropriately as in Lemma \ref{twist}, we may assume that $\chi=\mathbf{1}$ is the trivial character. By Lemma \ref{orbit-description} and Remark \ref{Pi_equiv_Remark}, the orbit $O$ of agrees with the $\Pi$-orbit of $\mathbf{1}$ and thus, by Corollary \ref{TPi_iso-cInd}, we have an algebra isomorphism $\cH(O)\cong M_{4\times 4}(\cH(\mathbf{1}))$ and so:

\begin{cor}\label{Morita_non-regular}
	The functor 
	\begin{align*}
		\Mod_{\cH(O)}&\xrightarrow{\sim} \Mod_{\cH(\mathbf{1})}\\
		M&\mapsto Me_{\mathbf{1}}
	\end{align*}
	is an equivalence of categories.
\end{cor}

	Viewing $\cH(\mathbf{1})$ as the $k$-algebra of genuine bi-$I$-invariant functions $\wt{G}\to k$ endowed with the convolution product, Lemma \ref{Hecke_Lemma} yields:

\begin{cor}\label{basis_H(1)}
	The algebra $\cH(\mathbf{1})$ has $k$-basis $\{T_{i,j},S_{i,j}\}_{i,j\in 2\bZ}$, where $T_{i,j}$ and $S_{i,j}$ are uniquely determined by 
	\[
	\supp(T_{i,j})=\wt{I}\left(\begin{mat}
		p^{i} & 0\\ 0 & p^{j}
	\end{mat},1\right)\wt{I}, \hspace{0.2cm} T_{i,j}\left(\left(\begin{mat}
		p^{i} & 0\\ 0 & p^{j}
	\end{mat},1\right)\right)=1
	\]
	and
	\[
	\supp(S_{i,j})=\wt{I}\left(\begin{mat}
		0&p^{i}\\ p^{j} & 0
	\end{mat},1\right)\wt{I}, \hspace{0.2cm} S_{i,j}\left(\left(\begin{mat}
		0 & p^{i}\\  p^{j} & 0
	\end{mat},1\right)\right)=1,
	\]
	respectively.
\end{cor}

\begin{prop}\label{H(1)}
	The $k$-algebra $\cH(\mathbf{1})$ is generated by $T_{2,2},T_{-2,-2},S_{0,0},S_{0,2}$, and the following is a full set of relations: $T_{2,2}$ is central and invertible with inverse $T_{-2,-2}$, $S_{0,0}^2=-S_{0,0}$ and $S_{0,2}^2=0$.
\end{prop}

\begin{proof}
	Writing $T_{2,2}=T_{\Pi^{4}}e_{\mathbf{1}}$, $T_{-2,-2}=T_{\Pi^{-4}}e_{\mathbf{1}}$, $S_{0,0}=T_se_{\mathbf{1}}$ and $S_{0,2}=T_{\Pi s \Pi} e_{\mathbf{1}}$, the relations follow from Lemma \ref{identities_Hecke_operators_II}, and Lemma \ref{identities_Hecke_operators_I} (v) to prove $S_{0,2}^2=0$. Furthermore, for all $m\in \bZ_{\geq 1}$,
	\begin{enumerate}
		\item[{\rm (a)}] $(S_{0,0}S_{0,2})^m=T_{2m,0}$,  $(S_{0,2}S_{0,0})^m=T_{0,2m}$;
		\item[{\rm (b)}] $T_{2m,0}S_{0,0}=S_{2m,0}$, \hspace{0.3cm} $S_{0,2}T_{2m,0}=S_{0,2(m+1)}$;
		\item[{\rm (c)}] $T_{2,2} T_{i,j}=T_{i+2,j+2}$, \hspace{0.2cm} $T_{2,2}S_{i,j}=S_{i+2,j+2}$ for all $i,j\in 2\bZ$.
	\end{enumerate}
	Indeed, (a) follows from (vi) and (vii), (b) follows from  (viii) and (ix), while (c) follows from (iii) of Lemma \ref{identities_Hecke_operators_II}.
	Corollary \ref{basis_H(1)} now finishes the proof.
\end{proof}

\begin{cor}\label{Z(R)}
	The center of $\cH(\mathbf{1})$ is a localized polynomial ring in two variables, namely the $k$-algebra morphism
	\begin{align*}
		k[\mathfrak{Z},Z^{\pm 1}]&\xrightarrow{\cong} Z(\cH(\mathbf{1}))\\
		\mathfrak{Z}&\mapsto S_{0,0}S_{0,2}+S_{0,2}S_{0,0}+S_{0,2}\\
		Z&\mapsto T_{2,2}
	\end{align*}
	is a well-defined isomorphism.		
\end{cor}

\begin{proof}
	The explicit description of the algebra $\cH(\mathbf{1})$ in the previous proposition implies that the center is equal to the set of elements, which are up to an integral power of $T_{2,2}$ of the form
	\[
	bT_{2,2}^M + \sum_{i\geq 1} a_i T_{2,2}^{N_i} (S_{0,0}S_{0,2})^{i}+\sum_{i\geq 1} a_i T_{2,2}^{N_i}(S_{0,2}S_{0,0})^{i}+\sum_{i\geq 0} a_{i+1}T_{2,2}^{N_{i+1}}(S_{0,2}S_{0,0})^{i}S_{0,2} 
	\]
	for $b,a_i\in k$, $M\in \bZ$ and $N_i\in \bZ_{\geq 0}$.
	An induction shows moreover that
	\[
	\mathfrak{Z}^{n}=(S_{0,0}S_{0,2})^n + (S_{0,2}S_{0,0})^n + (S_{0,2}S_{0,0})^{n-1}S_{0,2}
	\]
	for all $n\in \bZ_{\geq 1}$, finishing the proof.
\end{proof}

\subsubsection{Simple modules} 

\begin{lem}\label{simple_fd_non-reg}
	Any simple right $\cH(\mathbf{1})$-module is finite dimensional.
\end{lem}

\begin{proof}
	Let $M$ be a simple right $\cH(\mathbf{1})$-module. If $S_{0,0}+1$ kills this module, then $M$ descends to a module over the commutative finite type $k$-algebra $\cH(\mathbf{1})/(S_{0,0}+1)$ and we are done by Hilbert's Nullstellensatz. Hence, we may assume that there exists some non-zero $v\in M(S_{0,0}+1)$. In particular, $vS_{0,0}=0$. Writing $Z(\cH(\mathbf{1}))=k[\mathfrak{Z},Z^{\pm 1}]$ as in Corollary \ref{Z(R)}, one can now argue as in the square-reguler case, cf.\ proof of Lemma \ref{simple_fd_reg} replacing $v\begin{mat}
		0 & T_{0,1}\\ T_{0,1} & 0
	\end{mat}$ by $vS_{0,2}$.
\end{proof}

\begin{prop}\label{simple_non-regular} The simple right $\cH(\mathbf{1})$-modules are up to isomorphism as follows:
	\begin{enumerate}
		\item[{\rm (SS)}] For each $(\lambda,z)\in \{0,-1\} \times k^{\times}$, the character $\cH(\mathbf{1})\to k$, $T_{2,2}\mapsto z, S_{0,2}\mapsto 0, S_{0,0}\mapsto \lambda$.
		\item[{\rm (PS)}] For each $(\lambda,z)\in k^{\times}\times k^{\times}$, the $2$-dimensional space $kv\oplus kw$ with basis $v,w$ and action determined by $vT_{2,2}=zv, \hspace{0.1cm} vS_{0,2}=w, \hspace{0.1cm} wS_{0,0}=\lambda v.$
	\end{enumerate}
	There are no isomorphisms between objects associated to different parameters.
\end{prop}

\begin{proof}
	By the previous lemma, each simple right $\cH(\mathbf{1})$-module $M$ is finite dimensional. The same poof used by Vign\'{e}ras in \cite[Theorem A.2]{Vigneras} now works here: If $\dim_k M > 1$, then one shows that $S_{0,2}S_{0,0}$ cannot be nilpotent on $M$, so one may choose an eigenvector $v\in M$ with non-zero eigenvalue $\lambda\in k^{\times}$ and one has $M=kv\oplus kvS_{0,2}$.
\end{proof}

\begin{remark}\label{central_char_non-regular}
	The center $Z(\cH(\mathbf{1}))\cong k[\mathfrak{Z},Z^{\pm 1}]$, see Lemma \ref{Z(R)}, acts on the module of type (SS) (resp.\ (PS)) corresponding to the parameter $(\lambda,z)$ via the central character corresponding to the $k$-rational point $(0,z)$ (resp.\ $(\lambda,z)$).
\end{remark}

Using the Morita equivalence Corollary \ref{Morita_non-regular}, we identify the center of $\cH(O)$ with $Z(\cH(\mathbf{1}))\cong k[\mathfrak{Z},Z^{\pm 1}]$.

\begin{definition}\label{def_simple_nreg}
	\begin{enumerate}
		\item[{\rm (i)}] We say that a simple right $\cH(O)$-module is \textit{supersingular} if the center $k[\mathfrak{Z},Z^{\pm 1}]$ acts on it via a $k$-rational point of the form $(0,z)$ for $z\in k^{\times}$. Otherwise, we say that the module is \textit{principal}.
		\item[{\rm (ii)}] For $(\lambda,z)\in \{0,-1\}\times k^{\times}$, denote by $\operatorname{SS}(\lambda,z)$ the unique supersingular $\cH(O)$-module such that $\operatorname{SS}(\lambda,z)e_{\mathbf{1}}$ is equal to the character on $\cH(\mathbf{1})$ corresponding to $(\lambda,z)$ as in Proposition \ref{simple_non-regular}.
		\item[{\rm (iii)}] For $(\lambda,z)\in k^{\times}\times k^{\times}$, denote by $\operatorname{PS}(\lambda,z)$ the unique right $\cH(O)$-module such that $\operatorname{PS}(\lambda,z)e_{\mathbf{1}}$ is equal to the $2$-dimensional $\cH(\mathbf{1})$-module corresponding to $(\lambda,z)$ as in Proposition \ref{simple_non-regular}.
	\end{enumerate}
\end{definition}

\subsection{Extensions of simple modules}\label{Extensions}
Using \cite[§6]{Breuil_Paskunas} as a guideline, we compute the extensions of the simple modules determined in the previous sections in the category of $\cH$-modules on which $T_{\Pi^4}$ acts via a fixed scalar. Applying Lemma \ref{twist} to an appropriate unramified character, we may and do assume that $T_{\Pi^4}$ acts trivially. We put
\[
\cH_{p^2=1}=\cH/(T_{\Pi^4}-1)
\]
and use a similar notation for the smaller algebras $\cH(O)$, $\cH(\chi)$, etc.\ There can only be non-trivial extensions of two simple modules if they descend to the Hecke algebra of the same orbit. We therefore fix an orbit $O$ and we may and do without loss of generality assume that $O=O_{\mathbf{1}}$ in the non-square-regular case. Define the element
\[
\cH\ni \mathscr{Z}_O=\begin{cases}
	(T_{\Pi}T_s)^2+(T_sT_{\Pi})^2 \text{ if $O$ is square-regular}\\
	\sum_{i=0}^3 T_{\Pi}^{-i}\left((T_{\Pi}T_s)^2+(T_sT_{\Pi})^2 + T_{\Pi}T_sT_{\Pi}\right)e_{\mathbf{1}}T_{\Pi}^{i} \text{ if $O=O_{\mathbf{1}}$.}
\end{cases}
\]
which, when projected to $\cH(O)$, is central by Corollary \ref{center_reg}, resp.\ \ref{Z(R)}.

We start with two lemmas, the first of which we record for later reference.

\begin{lem}\label{T_s_for_later} We have equalities of right $\cH_{p^2=1}$-modules:
	\begin{enumerate}
		\item[{\rm (i)}] If $\chi\neq \chi^s$, then $T_s e_{\chi}\mathcal{H}_{p^2=1}=\ker\left(T_s\colon e_{\chi^s}\mathcal{H}_{p^2=1} \twoheadrightarrow e_\chi T_s\cH_{p^2=1} \right)$;
		\item[{\rm (ii)}] $T_se_{\mathbf{1}}\cH_{p^2=1}=\ker\left(T_s+1\colon e_{\mathbf{1}}\cH_{p^2=1}\twoheadrightarrow e_{\mathbf{1}}(T_s+1)\cH_{p^2=1}\right)$;
		\item[{\rm (iii)}] $(T_s+1)e_{\mathbf{1}}\cH_{p^2=1}=\ker\left(T_s\colon e_{\mathbf{1}}\cH_{p^2=1} \twoheadrightarrow e_{\mathbf{1}}T_s\cH_{p^2=1}\right)$;
	\end{enumerate}
	where the map $T_s$, resp.\ $T_s+1$, is given by acting from the left.
\end{lem}

\begin{proof}
	It suffices to prove the equality after applying the idempotent cutting out $\cH(O)_{p^2=1}$. We distinguish between the square-regular and non-square-regular case.
	
	Assume that $O$ is square-regular. By the Morita equivalence of Corollary \ref{Morita_regular}, it is enough to prove the statement after applying $e:=e_{\chi}+e_{\chi[1,0]}$ from the right. Applying the invertible element $T_{\Pi}$ from the left, the desired equality becomes equivalent to
	\[
	T_{\Pi}T_s e_{\chi} \cH_{p^2=1}e=\ker(T_sT_{\Pi}^{-1}\colon T_{\Pi} e_{\chi^s}\mathcal{H}_{p^2=1}e\to e_\chi T_s\cH_{p^2=1}e\subset \cH(\chi\oplus \chi[1,0])_{p^2=1}).
	\]
	Using Lemma \ref{identities_Hecke_operators_II} (v), (vi) and (vii), this is equivalent to
	\begin{align*}
		&\begin{mat}
			0 & T_{0,1}\\T_{0,1} & 0
		\end{mat} \begin{mat}
			\cH(\chi)_{p^2=1}&\cH(\chi[1,0],\chi)_{p^2=1}\\
			0 & 0
		\end{mat}\\
		&=\ker\left(\begin{mat}
			0 & T_{0,-1}\\T_{0,-1} & 0
		\end{mat}\colon \begin{mat}
			0 & 0\\ \cH(\chi,\chi[1,0])_{p^2=1} & \cH(\chi[1,0])_{p^2=1}
		\end{mat}\to \cH(\chi\oplus \chi[1,0])_{p^2=1}\right)\\
		&=\ker\left(\begin{mat}
			0 & T_{2,1}\\T_{2,1} & 0
		\end{mat}\colon \begin{mat}
			0 & 0\\ \cH(\chi,\chi[1,0])_{p^2=1} & \cH(\chi[1,0])_{p^2=1}
		\end{mat}\to \cH(\chi\oplus \chi[1,0])_{p^2=1}\right).
	\end{align*}
	For the last equality we used that the central element $T_{2,2}$ is invertible.
	The result now follows from Proposition \ref{Hecke_algebra_regular_char}, Proposition \ref{Hecke_module_regular_char} and Lemma \ref{diagonal_contribution}.
	
	Assume now that $O=O_{\mathbf{1}}$. By the Morita equivalence \ref{Morita_non-regular}, it is enough to prove the equalities after applying the idempotent $e_{\mathbf{1}}$. In this case, (ii) and (iii) translate into
	\[
	S_{0,0}\cH(\mathbf{1})_{p^2=1}=\ker\left(S_{0,0}+1\colon \cH(\mathbf{1})_{p^2=1}\twoheadrightarrow (S_{0,0}+1)\cH(\mathbf{1})_{p^2=1}\right)
	\]
	and
	\[
	(S_{0,0}+1)\cH(\mathbf{1})_{p^2=1}=\ker\left(S_{0,0}\colon \cH(\mathbf{1})_{p^2=1}\twoheadrightarrow S_{0,0}\cH(\mathbf{1})_{p^2=1}\right).
	\]
	For (i), note that by applying $T_{\Pi^2}$, we may replace $\chi$ by $\chi[1,1]$ and so assume that $\chi=\mathbf{1}[1,0]$. Thus, $e_{\chi}T_{\Pi}=T_{\Pi} e_{\mathbf{1}}$, see Lemma \ref{identities_Hecke_operators_II} (v). Acting by $T_{\Pi}$ from the left, we find that (iii) is equivalent to
	\begin{align*}
		T_{\Pi}T_sT_{\Pi} \cH(\mathbf{1})_{p^2=1}&=\ker\left(T_s T_{\Pi}^{-1}\colon \cH(\mathbf{1})_{p^2=1}\to e_{\mathbf{1}[1,0]}T_s\cH_{p^2=1}e_{\mathbf{1}}\right)\\
		&=\ker\left(T_{\Pi}^{-1}T_sT_{\Pi}^{-1}\colon \cH(\mathbf{1})_{p^2=1}\twoheadrightarrow T_{\Pi}^{-1}T_sT_{\Pi}^{-1}\cH(\mathbf{1})_{p^2=1}\right),
	\end{align*}
	where in the last equality we used the injectivity of $T_{\Pi}^{-1}$. This just means that
	\begin{align*}
		S_{0,2} \cH(\mathbf{1})_{p^2=1}&=\ker\left(T_{-2,-2}S_{0,2}\colon \cH(\mathbf{1})_{p^2=1}\twoheadrightarrow S_{0,2}\cH(\mathbf{1})_{p^2=1}\right)\\
		&=\ker\left(S_{0,2}\colon \cH(\mathbf{1})_{p^2=1}\twoheadrightarrow S_{0,2}\cH(\mathbf{1})_{p^2=1}\right)
	\end{align*}
	The assertions now follow from the description of $\cH(\mathbf{1})$ in Proposition \ref{H(1)}.	
\end{proof}

By similar arguments, one proves the following lemma.

\begin{lem}\label{Z-lambda}
	For all $\lambda\in k$ and all $\chi\in O$, the endomorphism
	\[
	\mathscr{Z}_{O}-\lambda\colon e_{\chi}\cH_{p^2=1}\to e_{\chi}\cH_{p^2=1}
	\]
	is injective.
\end{lem}

\begin{definition}\label{M(chi,lambda)_def}
	For $\lambda\in k$ and a character $\chi\in O$, define the $\cH_{p^2=1}$-modules $M(\chi,\lambda)$ by the short exact sequence
	\[
	0\to T_se_{\chi}\cH_{p^2=1}\xrightarrow{\mathscr{Z}_{O}-\lambda} T_se_{\chi}\cH_{p^2=1}\to M(\chi,\lambda)\to 0.
	\]
	For $O=O_{\mathbf{1}}$, we also define $M(\lambda)$ by
	\[
	0\to (T_s+1)e_{\mathbf{1}}\cH_{p^2=1}\xrightarrow{\mathscr{Z}_{O}-\lambda} (T_s+1)e_{\mathbf{1}}\cH_{p^2=1}\to M(\lambda)\to 0.
	\]
\end{definition}

\begin{remark}\label{twist_TPi^2}
	Using the relations $e_{\chi}T_{\Pi}^2=T_{\Pi}^2e_{\chi[1,1]}$, $T_sT_{\Pi}^2=(-1)^{\frac{p-1}{2}}T_{\Pi}^2T_s$ and the invertibility of $T_{\Pi}$ (see Lemma \ref{identities_Hecke_operators_II}), we obtain an isomorphism $T_{\Pi}^2\colon T_se_{\chi}\cH_{p^2=1}\cong T_se_{\chi[1,1]}\cH_{p^2=1}$ of right $\cH_{p^2=1}$-modules inducing an isomorphism on the quotients $M(\chi,\lambda)\cong M(\chi[1,1],\lambda)$.
\end{remark}

\begin{prop}\label{M(chi,lambda)_regular}
	Assume that $O$ is square-regular. Let $\chi\in O$ and $\lambda\in k$.
	\begin{enumerate}
		\item[{\rm (i)}] If $\lambda\neq 0$, then $M(\chi,\lambda)\cong \operatorname{PS}(\chi,0,\lambda,1)$.
		\item[{\rm (ii)}] If $\lambda=0$, then there is a non-split short exact sequence
		\[
		0\to \operatorname{SS}(\chi[1,0],1)\to M(\chi,0)\to \operatorname{SS}(\chi,1)\to 0
		\]
		of $\cH_{p^2=1}$-modules.
	\end{enumerate}
\end{prop}

\begin{proof} Put $e=e_{\chi}+e_{\chi[1,0]}$. The Morita equivalence \ref{Morita_regular} reduces to proving the statement after applying $e$. Acting by $T_{\Pi}$ from the left defines an isomorphism between $M(\chi,\lambda)$ and the cokernel of $\mathscr{Z}_O-\lambda$ acting on $T_{\Pi}T_se_{\chi}\cH_{p^2=1}$, which we denote by  $T_{\Pi}M(\chi,\lambda)$. It follows from propositions \ref{Hecke_module_regular_char} and \ref{Hecke_algebra_regular_char} that the images of the elements $T_{\Pi}T_s e_{\chi}=\begin{mat}
		0 & 0\\T_{0,1} & 0
	\end{mat}$, $T_{\Pi}T_se_{\chi}T_{\Pi}T_s=T_{(\Pi s)^2}e_{\chi[1,0]}=\begin{mat}
		0 & 0\\0 & T_{0,2}
	\end{mat}$ form a basis of the cokernel of the map
	\[
	\begin{mat}
		T_{2,0}+T_{0,2} -\lambda & 0\\ 0 & T_{2,0}+T_{0,2}-\lambda
	\end{mat} \colon \begin{mat}
		0 & 0\\ T_{0,1}\cH(\chi)_{p^2=1} & (T_{0,2})
	\end{mat} \to \begin{mat}
		0 & 0\\ T_{0,1}\cH(\chi)_{p^2=1} & (T_{0,2})
	\end{mat},
	\]
	where $(T_{0,2})$ is the ideal generated by $T_{0,2}$ in $\cH(\chi[1,0])_{p^2=1}$. Moreover, the central elements $(T_{2,0},T_{0,2},T_{2,2})$ act on the cokernel via the scalars $(0,\lambda,1)$. 
	
	Part (i) follows from the uniqueness in Proposition \ref{simple_regular_principal}. 
	
	For (ii), note that the $1$-dimensional subspace of $T_{\Pi}M(\chi,0)e$ spanned by the image of $\begin{mat}
		0 & 0\\ 0 & T_{0,2}
	\end{mat}$ is $\cH(\chi\oplus \chi[1,0])$-stable and must therefore be isomorphic to $\operatorname{SS}(\chi[1,0],0)e$. The cokernel of the resulting inclusion $\operatorname{SS}(\chi[1,0])e\hookrightarrow T_{\Pi}M(\chi,0)e$ is $1$-dimensional and $e_{\chi}$ acts trivially on it, so we obtain an exact sequence of the desired form, by Proposition \ref{simple_regular_ss}. It is non-split for if $\alpha,\beta\in k$ such that $\alpha \begin{mat}
		0 & 0\\ T_{0,1} & 0
	\end{mat}+\beta \begin{mat}
		0 & 0\\0 & T_{0,2}
	\end{mat}$ is fixed by $e_{\chi}$, then $\beta=0$, but the $1$-dimensional subspace spanned by $\begin{mat}
		0 & 0\\ T_{0,1} & 0
	\end{mat}$ is not $\cH(\chi\oplus \chi[1,0])$-stable: $\begin{mat}
		0 & 0\\ T_{0,1} & 0
	\end{mat} \begin{mat}0 & T_{0,1}\\0 & 0 \end{mat}=\begin{mat}
		0 & 0\\0 & T_{0,2}
	\end{mat}$, by Lemma \ref{diagonal_contribution}.
\end{proof}

\begin{prop}\label{M(chi,lambda)_non_regular} Assume that $O=O_{\mathbf{1}}$. Let $\chi\in O$ and $\lambda\in k$.
	\begin{enumerate}
		\item[{\rm (i)}] If $\lambda\neq 0$, then $M(\chi,\lambda)\cong \operatorname{PS}(\lambda,1) \cong M(\lambda)$.
		\item[{\rm (ii)}] If $\lambda=0$, then $M(\chi,0)=\operatorname{SS}(0,1)\oplus \operatorname{SS}(-1,1)$ for $\chi\neq \chi^{s}$, and for $\chi=\chi^{s}$ there are non-split exact sequences
		\[
		0\to \operatorname{SS}(0,1)\to M(\chi,0)\to \operatorname{SS}(-1,1)\to 0
		\]
		and
		\[
		0\to \operatorname{SS}(-1,1)\to M(0)\to \operatorname{SS}(0,1)\to 0.
		\]
	\end{enumerate}
\end{prop}

\begin{proof}
	It is enough to prove the statements after applying $e_{\mathbf{1}}$.
	If $\chi\neq \chi^{s}$, then, arguing as in the proof of Lemma \ref{T_s_for_later}, $M(\chi,\lambda)$ is isomorphic to the cokernel of the endomorphism on $S_{0,2}\cH(\mathbf{1})_{p^2=1}$ defined by the central element $\mathfrak{Z}-\lambda \in k[\mathfrak{Z},Z^{\pm 1}]\cong Z(\cH(\mathbf{1}))$ (see Corollary \ref{Z(R)}). Using the explicit description of $\cH(\mathbf{1})$ in Proposition \ref{H(1)}, one checks that a basis for this cokernel as a $k$-vector space is given by the images $v$ and $w$ of the two elements $S_{0,2}$ and $S_{0,2}S_{0,0}$, respectively. By construction, the center $k[\mathfrak{Z},Z^{\pm 1}]$ acts on it via the rational point $(\lambda,z)$, which implies (i) by the uniqueness in Proposition \ref{simple_non-regular}, see also Remark \ref{central_char_non-regular}. If $\lambda=0$, then write the cokernel as
	\[
	k(v+w)\oplus kw.
	\]
	Since $wS_{0,0}=-w$, $wS_{0,2}=0$, i.e.\ $kw=\operatorname{SS}(-1,1)$, and $(v+w)S_{0,0}=0$, $(v+w)S_{0,2}=0$, i.e.\ $k(v+w)=\operatorname{SS}(0,1)$, this settles the case $\chi\neq \chi^{s}$.
	
	If $\chi=\chi^{s}$, we may assume that $\chi=\mathbf{1}$ by Remark \ref{twist_TPi^2}. In this case, a basis of $M(\chi,\lambda)$ (resp.\ $M(\lambda)$) is given by the images $v$ and $w$ of $S_{0,0}$ and $S_{0,0}S_{0,2}$ (resp.\ $S_{0,0}+1$ and $(S_{0,0}+1)S_{0,2}$). Again by uniqueness of the principal module with central character corresponding to the rational point $(\lambda,z)$, we deduce (i). Assume that $\lambda=0$, then for $M(\chi,\lambda)$, $vS_{0,0}=-v$, $vS_{0,2}=w$ and $wS_{0,0}=0$, $wS_{0,2}=0$, so we obtain a non-split short exact sequence
	\[
	0\to kw=\operatorname{SS}(0,1)e_{\mathbf{1}}\to M(\chi,\lambda)e_{\mathbf{1}}\to \operatorname{SS}(-1,1)e_{\mathbf{1}}\to 0.
	\]
	For $M(0)$ on the other hand, we have $vS_{0,0}=0$, $vS_{0,2}=w$ and $wS_{0,0}=-w$, $wS_{0,2}=0$, so we obtain a non-split short exact sequence
	\[
	0\to kw=\operatorname{SS}(-1,1)e_{\mathbf{1}}\to M(0)e_{\mathbf{1}}\to \operatorname{SS}(0,1)e_{\mathbf{1}}\to 0,
	\]
	finishing the proof in case $\chi=\chi^s$.
\end{proof}

If a non-split extension between simple modules exists, then their central characters have to coincide. In particular, there are no non-split extensions neither of supersingular and principal modules nor of different principal modules.

\subsubsection{Extensions of supersingular modules} In order to compute the extensions of supersingular modules, we will distinguish between the square-regular case and the non-square-regular case.\\

\noindent \textit{Square-regular case.} Assume that $O$ is square-regular and write
\[
\cH(\chi\oplus \chi[1,0])=\begin{pmatrix}
	\cH(\chi) & \cH(\chi[1,0],\chi)\\
	\cH(\chi,\chi[1,0]) & \cH(\chi[1,0])
\end{pmatrix}
\]
and
identify
\[
\cH(\chi\oplus \chi[1,0])_{\xi}\cong \begin{pmatrix}k & k\oplus k\\ k \oplus k & \ol{\bF}_p
\end{pmatrix}
\]
as in (\ref{H(O/Pi)_xi}), where $\xi$ is the central character corresponding to the singular rational  point $(0,0,1)$, which we may view as a central character of $\cH(O)$ via the Morita equivalence \ref{Morita_regular}.

\begin{definition}
	For $\lambda,\mu \in k$, define $E_{\lambda,\mu}$ to be the unique $\cH(O)$-module so that $E_{\lambda,\mu}(e_{\chi}+e_{\chi[1,0]})$ is the $2$-dimensional $\cH(\chi\oplus \chi[1,0])_{\xi}$-module $kv\oplus kw$ with
	\[
	v\begin{mat}
		a & (b_1,b_2)\\ (c_1,c_2) & d
	\end{mat}=dv, \hspace{0.2cm} w\begin{mat}
		a & (b_1,b_2)\\ (c_1,c_2) & d
	\end{mat}=aw+(\lambda b_1+\mu b_2)v
	\]
	for all $\begin{mat}
		a & (b_1,b_2)\\ (c_1,c_2) & d
	\end{mat}\in \cH(\chi\oplus \chi[1,0])_{\xi}$.
\end{definition}

\begin{prop}
	\begin{enumerate}
		\item[{\rm (i)}] The map
		\begin{align*}
			k^2 & \xrightarrow{\cong} \Ext^1_{\cH(O)}(\operatorname{SS}(\chi,1),\operatorname{SS}(\chi[1,0],1))\\
			(\lambda,\mu) & \mapsto E_{\lambda,\mu}
		\end{align*}
		is a well-defined $k$-linear isomorphism.
		\item[{\rm (ii)}] The space $\Ext^1_{\cH(O)_{p^2=1}}(\operatorname{SS}(\chi,1),\operatorname{SS}(\chi,1))=0$ is trivial.
	\end{enumerate}
\end{prop}

\begin{proof}
	(i) Since the two supersingular modules are non-isomorphic, any non-trivial extension admits a central character, so the Ext-space in question is equal to
	\[
	\Ext^1_{\cH(O)_{\xi}}(\operatorname{SS}(\chi,1),\operatorname{SS}(\chi[1,0],1))\cong \Ext^1_{\cH(\chi\oplus \chi[1,0])_{\xi}}(\operatorname{SS}(\chi,1)e,\operatorname{SS}(\chi[1,0],1)e),
	\]
	where $e=e_{\chi}+e_{\chi[1,0]}$ and the isomorphism uses the Morita equivalence \ref{Morita_regular}. Write $E_{\lambda,\mu}e=k v\oplus k w$ as in its definition. Let $\bar{w}$ denote the image of $w$ in $E_{\lambda,\mu}e/kv$, so that we have a short exact sequence
	\[
	0\to kv \to E_{\lambda,\mu}e\to k\bar{w}\to 0
	\]
	It is now clear from the definition of $E_{\lambda,\mu}$ that the kernel (resp.\ cokernel) is just $\operatorname{SS}(\chi[1,0],1)e$ (resp.\ $\operatorname{SS}(\chi,1)e$) as $\cH(\chi\oplus \chi[1,0])_{\xi}$-modules, see also Proposition \ref{simple_regular_ss}. 
	
	One verifies that the assignment is $k$-linear. If the extension $E_{\lambda,\mu}e$ splits, then there exist $\alpha,\beta\in k$ such that $\alpha v +\beta w$ is fixed by the idempotent $e_{\chi}$, which forces $\alpha = 0$ and so $\lambda,\mu=0$. This proves injectivity.
	
	Conversely, let $0\to kv\to E \to k\bar{w}\to 0$ be an extension of the two supersingular modules. Then the operator $e_{\chi}$ on $E$ is diagonalizable with distinct eigenvalues, so we may choose an eigenvector $w$ with eigenvalue $1$. Since $\begin{mat}
		0 & (1,0)\\ 0& 0
	\end{mat}$ and $\begin{mat}
		0 & (0,1)\\0 & 0
	\end{mat}$ act trivially on $k\bar{w}$, we can find $\lambda,\mu\in k$ so that they act on $w$ via $\lambda v$ and $\mu v$, respectively. In other words, $E=E_{\lambda,\mu}$.
	
	(ii) We need to show that there are no non-trivial self-extensions of $\operatorname{SS}(\chi,1)e$ in the category of $\cH(\chi\oplus \chi[1,0])_{p^2=1}$-modules. Let $E$ be such an extension. The element $(e_{\chi}-1)^2=1-e_{\chi}$ kills $E$ and so for each $f\in E$,
	\[
	f\begin{mat}
		a & b\\ c & d
	\end{mat}=fe_{\chi} \begin{mat}
		a & b\\ c & d
	\end{mat}=fe_{\chi} \begin{mat}
		a & b\\ c & d
	\end{mat}e_{\chi}=f\begin{mat}
		a & 0\\ 0 & 0
	\end{mat} \text{ for all $\begin{mat}
			a & b\\c & d
		\end{mat}\in \cH(\chi\oplus \chi[1,0])_{p^2=1}$.}
	\]
	In particular, elements of the form $\begin{mat}
		0 & \ast\\0 & 0
	\end{mat}$ or $\begin{mat}
		0 & 0\\\ast & 0
	\end{mat}$ have to kill $E$. By Lemma \ref{diagonal_contribution}, this implies that also the elements of the form $\begin{mat}
		a & 0\\0 & 0
	\end{mat}$ for $a\in (T_{2,0},T_{0,2})$ kill $E$. Since $\cH(\chi)_{p^2=1}=k \oplus (T_{2,0},T_{0,2})$ by Proposition \ref{Hecke_algebra_regular_char}, we are done.
\end{proof}

\begin{remark}
	One may check that $M(\chi,0)\cong E_{1,0}$ and $M(\chi^{s}[1,0],0)\cong E_{0,1}$.
\end{remark}

\noindent \textit{Non-square-regular case.} Assume that $O=O_{\mathbf{1}}$.

\begin{prop}
	Let $M_1,M_2$ be supersingular $\cH(O)_{p^2=1}$-modules. Then
	\[
	\dim_{k}\Ext^1_{\cH(O)_{p^2=1}}(M_1,M_2)=1.
	\]
\end{prop}

\begin{proof}
	It is enough to prove this in the category of $\cH(\mathbf{1})_{p^2=1}$-modules after applying $e_{\mathbf{1}}$. Write $M_2=\operatorname{SS}(\lambda,1)$ for $\lambda\in \{0,-1\}$. Define an extension
	\[
	0\to M_1e_{\mathbf{1}}\to E\to M_2e_{\mathbf{1}}\to 0
	\]
	by putting $E=kv\oplus kpw$ with $kv=M_1e_{\mathbf{1}}$ and $wS_{0,0}=\lambda w$ and $wS_{0,2}=v$. This is a well-defined extension which is non-split (as $S_{0,2}$ acts non-trivially on $E$). Conversely, given any other such extension, say $F$, we may diagonalize the action of $S_{0,0}$ since $S_{0,0}(S_{0,0}+1)=0$. Choose a basis of eigenvectors $v,w$ with $v\in M_1e_{\mathbf{1}}$. Then necessarily $wS_{0,0}=\lambda w$. Since $S_{0,2}$ acts trivially on $M_2e_{\mathbf{1}}$, there exists some $\mu\in k$ with $wS_{0,2}=\mu v$. In conclusion, the extension $F$ is equal to the $\mu$-scalar multiple of the extension $E$.
\end{proof}

\subsubsection{Extensions of principal modules} 
In the square-regular case, we may write a principal module as $M(\chi,\lambda)$ for a suitable $0\neq \lambda\in k$ and some $\chi\in O$, see Proposition \ref{M(chi,lambda)_regular} and the intertwinings in Remark \ref{inter}. In case $O=O_{\mathbf{1}}$, we may write each principal module of $\cH(O)_{p^2=1}$ as $M(\lambda)$ for a suitable $\lambda\in k^{\times}$, see Proposition \ref{M(chi,lambda)_non_regular}. The following assertion can now either be proved as in \cite[Corollary 6.6]{Breuil_Paskunas}, resp.\ \cite[Corollary 6.7]{Breuil_Paskunas}, or is also a consequence of our later Ext-computations of principal series representations, see Remark \ref{Iwahori_ext}.

\begin{prop}\label{self_ext_pro_hecke_reg} Let $\lambda \in k^{\times}$.
	\begin{enumerate}
		\item[{\rm (i)}] If $O=O_{\chi}$ is square-regular, then the space $\Ext^1_{\cH(O)_{p^2=1}}(M(\chi,\lambda),M(\chi,\lambda))$ is $1$-dimensional with basis
		\[
		0\to M(\chi,\lambda)\xrightarrow{\mathscr{Z}_O-\lambda} T_s e_{\chi}\cH_{p^2=1} /(\mathscr{Z}_O-\lambda)^2\to M(\chi,\lambda)\to 0.
		\]
		\item[{\rm (ii)}] If $O=O_{\mathbf{1}}$, then the space $\Ext^1_{\cH(O)_{p^2=1}}(M(\lambda),M(\lambda))$ is $1$-dimensional with basis given by the extension
		\[
		0\to M(\lambda)\xrightarrow{\mathscr{Z}_{O}-\lambda} (T_s+1)e_{\mathbf{1}}\cH_{p^2=1}/(\mathscr{Z}_{O}-\lambda)^2\to M(\lambda)\to 0.
		\]
	\end{enumerate}
\end{prop}

\section{Genuine representation theory}\label{genuine_rep_theory}

In this part of the paper, we explain a complete classification result for the smooth irreducible genuine $k$-representations of $\wt{G}$ as introduced in Section \ref{gen_reps_and_notations}. We do this by working out the strategy outlined in the introduction. For $\pi\in\Mod^{\operatorname{sm}}_{\wt{G},\iota}(k)$, the $I_1$-invariants
	\begin{equation}\label{I_1-inv_can_iso}
	\pi^{I_1}\cong \Hom_{I_1}(\mathbf{1},\pi|_{I_1})\cong \Hom_{\wt{I}_1}(\mathbf{1}\boxtimes \iota, \pi|_{\wt{I}_1})\cong \Hom_{\wt{G}}(\cInd^{\wt{G}}_{\wt{I}_1}(\mathbf{1}\boxtimes \iota),\pi)
\end{equation}
carry a natural right action of the pro-$p$ Iwahori Hecke algebra $\cH$. We show that the adjoint pair $(\mathcal{T},\mathcal{I})$ defined by
\begin{align*}
	\mathcal{I}\colon \Mod^{\operatorname{sm}}_{\wt{G},\iota}(k)&\to  \Mod_{\mathcal{H}}\\
	\pi & \mapsto \mathcal{I}(\pi)=\pi^{I_1},
\end{align*}
and
\begin{align*}
	\hspace{1.85cm}\mathcal{T}\colon \Mod_{\mathcal{H}}&\to \Mod^{\operatorname{sm}}_{\wt{G},\iota}(k)\\
	M&\mapsto \mathcal{T}(M)=M\otimes_{\cH} \cInd^{\wt{G}}_{\wt{I}_1}(\mathbf{1}\boxtimes \iota),
\end{align*}
induces a bijection
\[
\left\{\begin{array}{c}
	\text{Simple right}\\
	\text{$\cH$-modules}
\end{array}\right\} \cong \left\{\begin{array}{c}
	\text{Smooth irreducible genuine}\\
	\text{$k$-representations of $\wt{G}$}
\end{array}\right\},
\]
where both sides are considered up to isomorphism. The results of the previous sections provide a classification of the left-hand side, and in particular show that there are only two kinds of simple modules: principal and supersingular modules. We will prove the desired bijection by considering the two kinds of modules separately. Before we start, let us extend the list of notations given by sections \ref{conv_notations} and \ref{gen_reps_and_notations}.

\subsection{Preliminaries}\label{prelim_and_not}

We define subgroups of $B$ and $T$ by
\begin{alignat*}{7}
	&B_2=\begin{pmatrix}
		A &\bQ_p\\0 & A
	\end{pmatrix} &&\supset T_2&&=\begin{pmatrix}
		A & 0\\0 & A
	\end{pmatrix}, &U=\begin{pmatrix}
		1 & \bQ_p\\0 & 1
	\end{pmatrix}\\
	&B_S=\begin{pmatrix}
		\bQ_p^{\times} & \bQ_p\\0 & S
	\end{pmatrix}&&\supset T_S && = \begin{pmatrix}
	\bQ_p^{\times} & 0\\0 & S
\end{pmatrix}&& &&\\
&P=\begin{pmatrix}
		\bQ_p^{\times} & \bQ_p\\0 & 1
	\end{pmatrix} &&\supset P^+&&=\begin{pmatrix}
		\bZ_p\setminus \{0\} & \bZ_p\\0 & 1
	\end{pmatrix}\supset  \hspace{0.1cm} &\Gamma=\begin{pmatrix}
		\bZ_p^{\times} & 0\\0 & 1
	\end{pmatrix}
\end{alignat*}

We identify $k\llbracket U\cap I\rrbracket =k\llbracket X\rrbracket$, the power series ring in the variable $X=\begin{mat}
	1 & 1\\0 & 1
\end{mat}-1$. Following \cite{emerton_coherent_rings}, we let $F=\begin{mat}
	p & 0\\0&1
\end{mat}$ act on $U\cap I$ via conjugation: $F\begin{mat}
	1 & b\\0 & 1
\end{mat}F^{-1}=\begin{mat}
	1 & pb\\0 &1
\end{mat}$ for all $b\in \bZ_p$, inducing the flat $k$-algebra endomorphism on $k\llbracket X\rrbracket$ given by $X\mapsto X^p$, and we let $k\llbracket X\rrbracket [F]$ be the non-commutative polynomial ring in $F$: $Ff(X)=F(f(X))F=f(X^p)F$ for all $f(X)\in k\llbracket X\rrbracket$. We also define $k\llbracket X\rrbracket [F,\Gamma]$ to be the twisted group ring so that the elements of $\Gamma$ commute with $F$ and $\gamma X=((X+1)^{\gamma}-1) \gamma$ for $\gamma\in \Gamma$.

\smallskip

If $M$ is a $k\llbracket X\rrbracket$-module, we denote the subspace of $X$-torsion elements (i.e.\ those killed by $X$) in $M$  by $M[X]$. 

\smallskip

The category of smooth $P^+$-representations is naturally a full subcategory of the category of left $k\llbracket X\rrbracket [F,\Gamma]$-modules.

\smallskip

By Lemma \ref{split_subgroups}, we have $\wt{B}_S=B_S\times \mu_2\supset \wt{P}=P\times \mu_2$ and $\wt{B}_2=B_2\times \mu_2$ as groups, which we will from now on use without further mention. In particular, we will simply view $B_S,P$ and $B_2$ as subgroups of the corresponding covering.\\

Without further comment, we will make use of the following lemma.
\begin{lem}
	Let $\pi \in \Mod^{\sm}_{\wt{G},\iota}(k)$ and $n\in N_G(T)$. Under the isomorphism (\ref{I_1-inv_can_iso}), the action of the Hecke operator $T_n$ on $\pi^{I_1}$ is given by
	\[
	vT_n=\sum_{x\in I_1/(I_1\cap (n^{-1}I_1n))} x^{-1} (\tilde{n})^{-1}v, \text{ for all $v\in \pi^{I_1}$.}
	\]
\end{lem}

\begin{proof}
	This follows from the definition, see also \cite[Corollary 2.6]{coeffsystems}: Letting $\varphi\in \cInd^{\wt{G}}_{\wt{I}_1}(\mathbf{1}\boxtimes \iota)$ be the unique function with support $\wt{I}_1$ and value $1$ at the identity, the isomorphism $\pi^{I_1}\cong \Hom_{\wt{G}}(\cInd^{\wt{G}}_{\wt{I}_1}(\mathbf{1}\boxtimes \iota),\pi)$ is given by $f(\varphi)\mapsfrom f$. Since $\wt{I}_1\tilde{n}I_1=\sqcup_{x\in I_1/(I_1\cap n^{-1}I_1n)} \wt{I}_1\tilde{n}x$, we have $T_n(\varphi)=\sum_{x\in I_1/(I_1\cap n^{-1}I_1n)} x^{-1}(\tilde{n})^{-1}\varphi$. Applying the $\wt{G}$-equivariant map $f$ gives the assertion.
\end{proof}

As in \cite[Corollary 2.0.11]{coeffsystems}, Lemma \ref{identities_Hecke_operators_I} has the following consequence.
\begin{cor}\label{cut}
	Given $\pi \in \Mod^{\operatorname{sm}}_{\wt{G},\iota}(k)$, the $I_1$-invariants $\pi^{I_1}\in \Mod^{\sm}_{\wt{I},\iota}(k)$ decomposes as
	\[
	\pi^{I_1}=\bigoplus_{\chi\in \Hhat} \pi^{I_1}e_{\chi},
	\]
	and $\wt{I}$ acts on $\pi^{I_1}e_{\chi}$ via the character $\chi\boxtimes \iota$.
\end{cor}

In the study of the Hecke algebra of a non-square-regular orbit in Section \ref{Hecke_non_reg}, we used the $k$-algebra automorphism $\operatorname{tw}\colon \cH\cong \cH$ of Lemma \ref{twist} induced by twisting. We will also do this on the representation theoretic side and we now check that this is harmless.

\begin{lem}\label{twist_lemma}
	Let $\psi\colon \bQ_p^{\times}\to k^{\times}$ be a smooth character. Let $\operatorname{tw}_{\psi}\colon \cH\cong \cH$ be the automorphism induced by twisting by $\psi\circ {\det}$. For each $M\in \Mod_{\cH}$, there is a $\wt{G}$-equivariant isomorphism
	\[
	\cT(M)\otimes \psi\circ {\det} \cong \cT(\operatorname{tw}_{\psi}^{\ast}(M)),
	\]
	which is natural in $M$.
\end{lem}

\begin{proof}
	Recall that $\operatorname{tw}_{\psi}$ is given by $h\mapsto j (h\otimes \psi\circ \det) j^{-1}$ for $h\in \cH$ , where $j\colon \cInd^{\wt{G}}_{\wt{I}_1}(\mathbf{1}\boxtimes \iota)\otimes \psi\circ \det \cong \cInd^{\wt{G}}_{\wt{I}_1}(\mathbf{1}\boxtimes \iota)$ is the canonical isomorphism.
	
	Using the adjoint pairs $((-)\otimes \psi\circ \det,(-)\otimes \psi^{-1}\circ \det)$, $(\cT,\cI)$ and $(\operatorname{tw}_{\psi}^{\ast},\operatorname{tw}_{\psi,\ast})$, it suffices, by Yoneda's Lemma, to show that for each genuine $\wt{G}$-representation $\pi$,
	\[
	\cI(\pi\otimes \psi^{-1}\circ \det)\cong\operatorname{tw}_{\psi,\ast}\cI(\pi)
	\]
	as $\cH$-modules, natural in $\pi$. The left-hand side is equal to
	\begin{align*}
		&\Hom_{\wt{G}}(\cInd^{\wt{G}}_{\wt{I}_1}(\mathbf{1}\boxtimes \iota),\pi \otimes \psi^{-1}\circ \det)\\
		&=\Hom_{\wt{G}}(\cInd^{\wt{G}}_{\wt{I}_1}(\mathbf{1}\boxtimes \iota)\otimes \psi\circ \det,\pi)\\
		&\cong \Hom_{\wt{G}}(\cInd^{\wt{G}}_{\wt{I}_1}(\mathbf{1}\boxtimes \iota),\pi)
	\end{align*}
	where the $\cH$-action in the second line is given by precomposition with $h\otimes (\psi\circ \det)$ for $h\in \cH$ and the isomorphism is induced by $j$, so that the induced $\cH$-action in the last line is given by precomposition with $j (h\otimes \psi \circ \det)j^{-1}$ for $h\in \cH$, which proves the lemma.
\end{proof}

Speaking of twisting, the elements of $\wt{Z}$ define interesting $\wt{G}$-equivariant isomorphisms. Besides the already introduced smooth character $\omega\colon \bQ_p^{\times}\twoheadrightarrow [\bF_p^{\times}]\cong \bF_p^{\times}\subset k^{\times}$, we denote the unramified character mapping $p$ to a fixed element $y\in k^{\times}$ by $\mu_{y}$.

\begin{lem}\label{Ztilde_action}
	For all $\left(\begin{mat}
		x & 0\\0 & x
	\end{mat},\xi\right)\in \wt{Z}$ and $(g,\zeta)\in \wt{G}$, 
	\[
	\left(\begin{mat}
		x & 0\\0 & x
	\end{mat},\xi\right) (g,\zeta) \left(\begin{mat}
		x & 0\\0 & x
	\end{mat},\xi\right)^{-1}=\left(g,\zeta \left(\mu_{-1}^{v(x)} \omega^{v(x)} \mu_{\omega(x)}\right)^{\frac{p-1}{2}}(\det(g))\right).
	\]
	In particular, for $\pi\in \Mod^{\sm}_{\wt{G},\iota}(k)$, the map
	\[
	\left(\begin{mat}
		x & 0\\0 & x
	\end{mat},\xi\right) \colon \pi\xrightarrow{\cong} \pi \otimes \left(\mu_{-1}^{v(x)} \omega^{v(x)} \mu_{\omega(x)}\right)^{\frac{p-1}{2}}\circ {\det}
	\]
	is a $\wt{G}$-equivariant isomorphism.
\end{lem}

\begin{proof}
	This is a straight forward computation using the definition of the cocycle (\ref{cocycle}) defining $\wt{G}$.
\end{proof}

We finish the section with a quick observation.

\begin{lem}\label{no_fd}
	There are no finite dimensional objects in $\Mod^{\sm}_{\wt{G},\iota}(k)$.
\end{lem}

\begin{proof}
	Let $\pi$ be a finite dimensional and smooth $\wt{G}$-representation, which we may assume to be irreducible. Its kernel is an open subgroup of $\wt{G}$. In particular, it contains a congruence subgroup $K_m$ for some $m\geq 1$. Projecting to $G$ and conjugating by $F^n$ (resp.\ $(sFs)^n$) for all $n\geq 0$, shows that the image of the kernel in $G$ contains $U$ as well as $sUs$ and thus $\operatorname{SL}_2(\bQ_p)$. We conclude that  each element of $\wt{\operatorname{SL}_2(\bQ_p)}$ acts on $\pi$ by $\pm 1$ (hence by scalars). It follows that the restriction of $\pi$ to the subgroup $P\cap T\cong \bQ_p^{\times}$ is still irreducible and must therefore be $1$-dimensional. However, the commutator subgroup of $\wt{G}$ contains $[\wt{T},\wt{T}]=\mu_2$ (cf.\ Lemma \ref{centers} (ii)), contradicting genuinity.
\end{proof}

\subsection{Weights}
We remind the reader of the well-known classification of smooth irreducible $k$-representations of $K$, so-called \textit{weights}.

\begin{definition}\label{sym_def}
	For $r\in \bZ_{\geq 0}$, define $\Sym^{r}(k^2)$ to be the space of homogeneous degree-$r$ polynomials in $k[x,y]$ with $\GL_2(\bF_p)$-action defined by
	\[
	\begin{mat}
		a & b\\ c & d
	\end{mat}P(x,y)=P(ax+cy,bx+dy), \text{ for all $\begin{mat}
			a & b\\c & d
		\end{mat}\in \GL_2(\bF_p)$, $P(x,y)\in \Sym^{r}(k^2)$.}
	\]
	We will view it as representation of $K$ via inflation along the reduction map.
\end{definition}

\begin{prop}\label{class_weights}
	The map
	\begin{align*}
		\left\{
		(a,b)\in \bZ^2 : \begin{array}{l}
			0\leq a-b\leq p-1\\
			0\leq b < p-1
		\end{array}\right\} &\to \left\{\begin{array}{c}\text{Smooth irreducible}\\ \text{$k$-representations of $K$}\end{array}\right\}\\
		(a,b)&\mapsto \Sym^{a-b}(k^2)\otimes {\det}^b
	\end{align*}
	is well-defined and bijective, where the right-hand side is considered up to isomorphism.
\end{prop}

\begin{proof}
	See for example \cite[Proposition 14]{herzig}.
\end{proof}

	\begin{prop}\label{JH_weight}
	Given $(a,b)\in \bZ^2$ with $0\leq a-b\leq p-1$ and $0\leq b < p-1$, there is a short exact sequence of $K$-representations,
	\[
	0\to \Sym^{p-1-(a-b)}(k^2)\otimes {\det}^{a}\to \Ind^{K}_I(\omega^{a}\otimes \omega^{b})\to \Sym^{a-b}(k^2)\otimes {\det}^b\to 0.
	\]
\end{prop}

\begin{proof}
	This follows from Frobenius reciprocity and counting dimensions, see for example \cite[p.\ 23]{herzig}.
\end{proof}

\begin{definition}\label{ss_weights}
	Given a character $\psi\in \Hhat$, define the semi-simple $K$-representation $W(\psi)$ by
	\[
	W(\psi)=\bigoplus_{\sigma: \sigma^{I_1}=\psi} \sigma,
	\]
	where the direct sum is over all smooth irreducible $K$-representations $\sigma$ with $\sigma^{I_1}=\psi$.
\end{definition}

\begin{remark}
	If $\psi\neq \psi^{s}$, then $W(\psi)$ is irreducible. Otherwise, $W(\psi)$ consists of the two constituents of $\Ind^K_I(\psi)$.
\end{remark}

\begin{definition}
	The maximal semi-simple subrepresentation of a smooth $K$-representation $\pi$ will be denoted by $\soc_K(\pi)$ and is called the \textit{$K$-socle of $\pi$}.
\end{definition}

\subsection{Principal series representations}

\begin{definition}
	A \textit{genuine principal series representation} of $\wt{G}$ is a representation of $\wt{G}$ isomorphic to the smooth parabolic induction $\Ind^{\wt{G}}_{\wt{B}}(\tau)$ of some smooth irreducible genuine $k$-representation $\tau$ of $\wt{T}$ viewed as a representation of $\wt{B}$ via inflation.
\end{definition}

In order to understand these objects, an investigation of the $\wt{T}$-representations we are inducing is necessary, which is the content of the next section.
	
\subsubsection{Smooth irreducible genuine $\wt{T}$-representations} 
The classification result we are about to prove closely follows \cite[Chapter 5]{McNamara}, which deals with smooth genuine complex representations in a more general setting, see especially Corollary 5.2 in \textit{loc.\ cit.} It says that a smooth irreducible genuine $\wt{T}$-representation is uniquely determined by its central character. Before stating the result, we need two lemmas.

\begin{lem}\label{torus_irred_fd}
	Any smooth irreducible genuine $k$-representation of $\wt{T}$ is finite dimensional.
\end{lem}

\begin{proof}
	The proof is inspired by \cite[Lemma 2.4]{labesse_langlands_1979}. By Lemma \ref{centers}, $Z(\wt{T})=T^{\sq}\times \mu_2$. Filtering the inclusion $Z(\wt{T})\subset \wt{T}$ by subgroups of successive index two, and observing that, if $G_1\subset G_2$ is an inclusion of groups of index two, then the restriction of an irreducible $G_2$-representation to $G_1$ is semi-simple of length at most two, the lemma reduces to the statement that a smooth irreducible genuine $Z(\wt{T})$-representation, or equivalently a smooth irreducible $T^{\sq}$-representation, is finite dimensional. Modulo the pro-$p$ part, $T^{\sq}$ is a finitely generated abelian group, so Hilbert's Nullstellensatz implies that a smooth irreducible $k$-representation of it is in fact $1$-dimensional.
\end{proof}

\begin{lem}\label{conj_ij}
	Let $i,j\in \bZ$. For all $\begin{mat}
		a & b\\ 0 & d
	\end{mat}\in B_2$, 
	\[
	\left(\begin{mat}
		p^{i} & 0\\0 & p^{j}
	\end{mat},1\right)\begin{mat}
		a & b\\0 & d
	\end{mat}	\left(\begin{mat}
		p^{i} & 0\\0 & p^{j}
	\end{mat},1\right)^{-1}=\left(\begin{mat}
		a & p^{i-j}b\\0 & d
	\end{mat}, \omega(d)^{i\frac{p-1}{2}} \omega(a)^{j\frac{p-1}{2}}\right).
	\]
\end{lem}

\begin{proof}
	By Example \ref{law_torus}, $\left(\begin{mat}
		p^{i} & 0\\0 & p^{j}
	\end{mat},1\right)^{-1}=\left(\begin{mat}
		p^{-i} & 0\\0 & p^{-j}
	\end{mat},(-1)^{ij\frac{p-1}{2}}\right)$. The lemma is now a simple computation based on the definition of the cocycle (\ref{cocycle}) defining $\wt{G}$.
\end{proof}

\begin{notation}\label{notation_[i,j]}
	For $i,j\in \bZ$ and a smooth representation $\pi$ of a subgroup of $\wt{B}$, denote by
	\[
	 \pi[j,i]=\pi\otimes (\omega^{j\frac{p-1}{2}}\otimes \omega^{i\frac{p-1}{2}})
	 \]
	 the twist of $\pi$ by the character $\omega^{j\frac{p-1}{2}}\otimes \omega^{i\frac{p-1}{2}}\colon \wt{B}\twoheadrightarrow B\twoheadrightarrow T\to \mu_2\subset k^{\times}$ restricted to the subgroup of consideration.
\end{notation}

\begin{def-prop}\label{class_T}
	The map
	\begin{align*}
		\left\{\begin{array}{c}\text{Smooth characters}\\ T^{\sq}\to k^{\times}\end{array}\right\} &\xrightarrow{\cong} \left\{\begin{array}{c}
			\text{Smooth irreducible genuine}\\
			\text{$k$-representations of $\wt{T}$}
		\end{array}\right\}\\
		\chi\colon T^{\sq}\to k^{\times} & \mapsto \tau_{\chi}:=\Ind^{\wt{T}}_{\wt{T}_2}\left(\chi'\boxtimes \iota\right)
	\end{align*}
	is a well-defined bijection, where $\chi'$ is any extension of $\chi$ to $T_2$ and the right-hand side is considered up to isomorphism.
\end{def-prop}

\begin{proof}
	Note that, since $k$ is divisible and so injective in the category of abelian groups, any (smooth) character of $T^{\sq}$ extends to a (smooth) character of $T_2$. Let now $\chi$ be a smooth character on $T_2$. By Lemma \ref{conj_ij} we have
	\begin{equation}\label{Ind|_{T_2}}
		\Ind^{\wt{T}}_{\wt{T}_2}(\chi\boxtimes \iota)|_{T_2}=\bigoplus_{i,j\in \bZ/2\bZ} \chi[j,i],
	\end{equation}
	which is the direct sum over all distinct characters extending $\chi|_{T^{\sq}}$. This proves irreducibility as well as the independence of the chosen extension.
	
	Conversely, by Lemma \ref{torus_irred_fd}, any smooth irreducible genuine $k$-representation $\tau$ of $\wt{T}$ is finite dimensional and thus admits a central character, i.e.\ there exists a smooth character $\chi\colon T^2\to k^{\times}$ with $\chi\subset \tau|_{T^{\sq}}$. Using Frobenius reciprocity and that $\Ind^{T_2}_{T^{\sq}}(\chi)$ is the direct sum of the characters extending $\chi$ (since $T_2/T^{\sq}$ is finite of order prime to $p$), we deduce that there exists a character $\chi'$ of $T_2$ extending $\chi$ with $\chi'\subset \tau|_{T_2}$. Using Frobenius reciprocity once again, we obtain a non-trivial $\wt{T}$-equivariant morphism $\Ind^{\wt{T}}_{\wt{T}_2}(\chi'\boxtimes \iota)\to \tau$, which has to be an isomorphism as both sides are irreducible.
\end{proof}

\begin{remark}
	\begin{enumerate}
		\item[{\rm (i)}] In the previous proposition, the subgroup $\wt{T}_2$ can be replaced by any other maximal abelian subgroup of $\wt{T}$, for example by $\wt{T}_S$.
		\item[{\rm (ii)}] A genuine principal series of $\wt{G}$ is of the form $\Ind^{\wt{G}}_{\wt{B}_2}(\chi\boxtimes \iota)$ for a smooth character $\chi\colon T_2\to k^{\times}$ viewed as a character of $B_2$ via inflation.
	\end{enumerate}
\end{remark}

	\begin{definition}\label{def_std_basis}
	For a smooth character $\chi\colon T_2\to k^{\times}$, define the \textit{standard basis} $\{f_{i,j}\}_{i,j=0,1}$ of $\tau_{\chi}$ by letting $f_{i,j}$ be the unique function with
	\[
	\supp(f_{i,j})=\wt{T}_2\left(\begin{mat}
		p^{i} & 0\\0 & p^{j}
	\end{mat},1\right) \text{ and } f_{i,j}\left(\left(\begin{mat}
		p^{i} & 0\\0 & p^{j}
	\end{mat},1\right)\right)=1.
	\]
\end{definition}

\begin{lem}\label{action_standard_basis}
	For all $i,j=0,1$, the following hold true.
	\begin{enumerate}
		\item[{\rm (i)}] $f_{i,j}$ is an $H$-eigenvector with eigencharacter $\chi|_H[j,i]$;
		\item[{\rm (ii)}] $\left(\begin{mat}
			1 & 0\\ 0 & p^{-1}
		\end{mat},1\right) f_{i,0} = (-1)^{i\frac{p-1}{2}} f_{i,1}$;
		\item[{\rm (iii)}] $\left(\begin{mat}
			1 & 0\\ 0 & p^{-1}
		\end{mat},1\right) f_{i,1} = (-1)^{i\frac{p-1}{2}}\chi\left(\begin{mat}
			1 & 0\\0 & p^{-2}
		\end{mat}\right) f_{i,0}$;
		\item[{\rm (iv)}] $\left(\begin{mat}
			p^{-1} & 0\\ 0 & 1
		\end{mat},1\right)f_{0,j} = f_{1,j}$;
		\item[{\rm (v)}] $\left(\begin{mat}
			p^{-1} & 0\\ 0 & 1
		\end{mat},1\right)f_{1,j} = \chi\left(\begin{mat}
			p^{-2} & 0\\0 & 1
		\end{mat}\right)f_{0,j}$.
	\end{enumerate}
\end{lem}

\begin{proof}
	Part (i) follows from Lemma \ref{conj_ij}. Part (ii) follows from the equation $\left(\begin{mat}
		p^{i} & 0\\ 0 & p
	\end{mat}\right) \left(\begin{mat}
		1 & 0\\ 0 & p^{-1}
	\end{mat},1\right)=\left(\begin{mat}
		p^{i} & 0\\ 0 & 1
	\end{mat}, (-1)^{i\frac{p-1}{2}}\right)$, see Example \ref{law_torus}. Part (iii) follows from (ii) since the element $\begin{mat} 1 & 0\\ 0 & p^{-2}
	\end{mat}$ is central and $\tau_{\chi}$ has central character $\chi$. The remaining two parts are proved analogously.
\end{proof}

\subsubsection{Irreducibility of principal series representations}

Let us fix a smooth character $\chi\colon T_2\to k^{\times}$ and denote by
\[
\tilde{\pi}(\chi)=\Ind^{\wt{G}}_{\wt{B}}(\tau_{\chi})
\]
the corresponding principal series. The goal of this section it to prove:

\begin{prop}\label{PS_irred}
	The genuine principal series $\tilde{\pi}(\chi)$ is irreducible as a representation of $\wt{G}$. Moreover, $\tilde{\pi}(\chi)\cong \tilde{\pi}(\psi)$ for another smooth character $\psi\colon T_2\to k^{\times}$ if and only if $\chi|_{T^{\sq}}=\psi|_{T^{\sq}}$.
\end{prop}

\noindent The if-direction of the proposition trivially holds true as it does already for $\tau_{\chi}$. In particular, to prove irreducibility, we may replace $\chi$ by one of its twists $\chi[i,j]$. In the non-square-regular case, we may and do therefore assume that $\chi|_H=\chi|_H^s$. 

\begin{definition}
	For $i,j=0,1$, define elements $F_{i,j}^0,F_{i,j}^1\in \tilde{\pi}(\chi)^{I_1}$ uniquely determined by 
	\[
	\supp(F_{i,j}^0)=\wt{B}I_1 \text{ and } F_{i,j}^{0}\left(\begin{mat}
		1 & 0\\0 & 1
	\end{mat},1\right)=f_{i,j}
	\]
	and
	\[
	\supp(F_{i,j}^1)=\wt{B}\tilde{s}I_1 \text{ and } F_{i,j}^1(\tilde{s})=f_{i,j}.
	\]
\end{definition}

\noindent The smooth character $\chi|_{T^{\sq}}$ being trivial on the pro-$p$ group $T\cap I_1$ implies that these elements are well-defined. Since $\{f_{i,j}\}_{i,j=0,1}$ is a $k$-basis of $\tau_{\chi}$, the decomposition $\wt{G}=\wt{B}I_1 \sqcup \wt{B}\tilde{s}I_1$ immediately implies that $\{F_{i,j}^0,F_{i,j}^1\}_{i,j=0,1}$ is a $k$-basis of $\tilde{\pi}(\chi)^{I_1}$.

\begin{lem}\label{action_Fij}
	For all $i,j=0,1$, the following hold true.
	\begin{enumerate}
		\item[{\rm (i)}] $F_{i,j}^0e_{\chi|_H[j,i]} = F_{i,j}^0$ and $F_{i,j}^1 e_{\chi|_H[j,i]^s}=F_{i,j}^1$;
		\item[{\rm (ii)}] $F_{i,0}^0 T_{\Pi} = (-1)^{i\frac{p-1}{2}}F_{i,1}^1$ and $F_{i,1}^0T_{\Pi} = (-1)^{i\frac{p-1}{2}}\chi\left(\begin{mat}
			1 & 0\\0 & p^{-2}
		\end{mat}\right) F_{i,0}^1$;
		\item[{\rm (iii)}] $F^1_{0,j}T_{\Pi} = F_{1,j}^0$ and $F^1_{1,j}T_{\Pi} = \chi\left(\begin{mat}
			p^{-2} & 0\\0 & 1
		\end{mat}\right)F_{0,j}^0$;
		\item[{\rm (iv)}] $F^0_{i,j}T_s = F_{i,j}^1$.
	\end{enumerate}
\end{lem}

\begin{proof}
	Parts (i) to (iii) follow from Lemma \ref{action_standard_basis}. For part (iv) recall that in the non-square-regular case we agreed to choose $\chi$ so that $\chi|_H=\chi|_H^s$. Part (i) together with Lemma \ref{identities_Hecke_operators_I} (iv) implies that $F_{i,j}^0T_s$ is a linear combination of $F_{i,j}^1$ and $F_{j,i}^0$ (in the square-regular case, even just a scalar multiple of $F_{i,j}^1$). It therefore suffices to evaluate at $\tilde{s}$ and $1$. Now, $I_1sI_1=\sqcup_{\lambda\in \bF_p} I_1 s\begin{mat}
		1 & [\lambda]\\0 & 1
	\end{mat}$, and so
	\[
	F_{i,j}^0T_s=\sum_{\lambda\in \bF_p} 
	\begin{mat}
		1 & [\lambda]\\0 & 1
	\end{mat}\tilde{s}F_{i,j}^0.
	\]
	Since $\supp(F_{i,j}^0)=\wt{B}I_1$, each summand evaluated at $1$ gives $0$, while only the summand corresponding to $\lambda=0$ contributes with value $1$ if one evaluates at $\tilde{s}$. This concludes the proof.
\end{proof}

	\begin{proof}[Proof of Proposition \ref{PS_irred}]
	It is enough to prove that $\tilde{\pi}(\chi)^{I_1}$ is a simple right $\cH$-module. Indeed, this follows from $\tilde{\pi}(\chi)^{I_1}$ generating $\tilde{\pi}(\chi)$ as a $\wt{G}$-representation, which in turn is a consequence of the fact\footnote{E.g.\ see \cite[p. 24 (a)]{Vigneras}; a shorter proof: the Iwahori decomposition for the congruence subgroup $K_m$, for $m\geq 1$, implies that $BK_m=BI_1 \begin{mat}
			p^{m-1} & 0\\0 & 1
		\end{mat}$.} that each compact open subgroup of $\wt{B}\setminus \wt{G}\cong B\setminus G$ is a finite disjoint union of subsets of the form $B\setminus BI_1g$ for suitable $g\in G$. In order to prove the irreducibility of $\tilde{\pi}(\chi)^{I_1}$ as an $\cH$-module, first note that $\tilde{\pi}(\chi)^{I_1}e_{O_{\chi|_H}}=\tilde{\pi}(\chi)^{I_1}$, where $e_{O_{\chi|_H}}\in \cH$ is the idempotent cutting out $\cH(O_{\chi|_H})$. We distinguish two cases.
	
		If $\chi|_H$ is square-regular, it suffices to prove that $\tilde{\pi}(\chi)^{I_1}(e_{\chi|_H}+e_{\chi|_H[1,0]})$ is simple as a module over $\cH(\chi|_H\oplus \chi|_H[1,0])$, see Corollary \ref{Morita_regular}. By Lemma \ref{action_Fij}, we have
		\[
		\tilde{\pi}(\chi)^{I_1}(e_{\chi|_H}+e_{\chi|_H[1,0]})=k F_{0,0}^0 + k F_{0,1}^0,
		\]
		and the Hecke operators $(T_{(s\Pi)^2},T_{(\Pi s)^2},T_{\Pi^4})$ act on this space via the scalars	$\left(\chi\left(\begin{mat}
			p^{-2} & 0\\0 & 1
		\end{mat}\right),0,\chi\left(\begin{mat}
			p^{-2} & 0\\ 0 & p^{-2}
		\end{mat}\right)\right)$.
		The uniqueness in Proposition \ref{simple_regular_principal} now gives 
		\begin{equation}\label{ps_reg}
			\tilde{\pi}(\chi)^{I_1}\cong \operatorname{PS}\left(\chi|_H,	\left(\chi\left(\begin{mat}
				p^{-2} & 0\\0 & 1
			\end{mat}\right),0,\chi\left(\begin{mat}
				p^{-2} & 0\\ 0 & p^{-2}
			\end{mat}\right)\right)\right).
		\end{equation}
		
		Assume now that $\chi|_H=\chi|_H^s$. By twisting, we may even assume that $\chi|_H=\mathbf{1}$. By Lemma \ref{action_Fij}, we have
		\[
		\tilde{\pi}(\chi)^{I_1}e_{\mathbf{1}}=kF_{0,0}^0 + k F_{0,0}^1
		\]
		and the center $k[\mathfrak{Z},Z^{\pm 1}]$ of $\cH(\mathbf{1})$ (Corollary \ref{Z(R)}) acts via the character corresponding to the $k$-rational point $\left(\chi\left(\begin{mat}
			p^{-2} & 0\\0 & 1
		\end{mat}\right),\chi\left(\begin{mat}
			p^{-2} & 0\\ 0 & p^{-2}
		\end{mat}\right)\right)$. Thus, by Proposition \ref{simple_non-regular} and the remark thereafter,
		\begin{equation}\label{ps_nreg}
			\tilde{\pi}(\chi)^{I_1}=\operatorname{PS}\left(\chi\left(\begin{mat}
				p^{-2} & 0\\0 & 1
			\end{mat}\right),\chi\left(\begin{mat}
				p^{-2} & 0\\ 0 & p^{-2}
			\end{mat}\right)\right).
		\end{equation}
	
	\noindent This concludes the proof of the irreducibility of $\tilde{\pi}(\chi)$. 
	Concerning the intertwinings: in the square-regular case, the isomorphism (\ref{ps_reg}) together with Remark \ref{inter} implies that $\tilde{\pi}(\chi)\cong \tilde{\pi}(\psi)$ for another smooth character $\psi$ on $T_2$ only if $\psi|_H=\chi|_H[i,j]$ for some $i,j\in \bZ/2\bZ$,  $\chi\left(\begin{mat}
		p^{-2} & 0\\0 & 1
	\end{mat}\right)=\psi\left(\begin{mat}
		p^{-2} & 0\\0 & 1
	\end{mat}\right)$ and $\chi\left(\begin{mat}
		p^{-2} & 0\\ 0 & p^{-2}
	\end{mat}\right)=\psi\left(\begin{mat}
		p^{-2} & 0\\ 0 & p^{-2}
	\end{mat}\right)$, i.e.\ only if $\chi|_{T^{\sq}}=\psi|_{T^{\sq}}$. 
	
	In the non-square-regular case, an isomorphism $\tilde{\pi}(\chi)^{I_1}\cong \tilde{\pi}(\psi)^{I_1}$ of $\cH$-modules implies that $O_{\chi|_H}=O_{\psi|_H}$, i.e.\ $\psi|_H=\chi|_H[i,j]$ for some $i,j\in \bZ/2\bZ$. From the isomorphism (\ref{ps_nreg}) we can read off the values of $\chi$ and $\psi$ at $\begin{mat}
		p^{-2} & 0\\0 & 1
	\end{mat}$ and $\begin{mat}
		p^{-2} & 0\\ 0 & p^{-2}
	\end{mat}$, so in total $\chi|_{T^{\sq}}=\psi|_{T^{\sq}}$.
\end{proof}

\subsubsection{From principal Hecke modules to principal series representations} The proof of the previous proposition shows that the assignment $\pi \mapsto \cI(\pi)$ defines a bijection between the isomorphism classes of principal $\cH$-modules and the isomorphism classes of (irreducible) genuine principal series representations of $\wt{G}$. For showing that the inverse is given by the functor $\cT$, we need to check that the canonical map
\begin{equation}\label{can_map_PS}
	\cT(\cI(\tilde{\pi}(\chi)))\to \tilde{\pi}(\chi)
\end{equation}
is an isomorphism. To ease notation, put $M=\cI(\tilde{\pi}(\chi))$. After an unramified twist, we may and do also assume that $\chi|_{H}=\mathbf{1}$ in the non-square-regular case, cf.\ Lemma \ref{twist_lemma}. The map (\ref{can_map_PS}) is non-trivial (by adjunction), so the source must be non-zero. In particular, the unit $M\to \cI\cT(M)$ is an injection, via which we will view $M=\tilde{\pi}(\chi)^{I_1}$ as a subspace of $\cT(M)$.\\

Denote by 
\[
\tilde{\pi}(\chi)^{\operatorname{bc}}=\left\{f\in \tilde{\pi}(\chi) : \supp(f)\subset \wt{B}\tilde{s}\wt{B}\right\}
\]
the $\wt{B}$-stable subspace consisting of those functions whose support is contained in the big cell $\wt{B}\tilde{s}\wt{B}$. We thus have an exact sequence of $\wt{B}$-representations
\[
0\to \tilde{\pi}(\chi)^{\operatorname{bc}}\to \tilde{\pi}(\chi)\to \tau_{\chi}\to 0
\]
induced by evaluation at the identity. In order to prove the canonical map (\ref{can_map_PS}) to be an isomorphism, we will construct an analogue of this sequence for the object $\cT(M)$ and show that the canonical map induces an isomorphism on the outer terms of the short exact sequences.\\

For a smooth $k$-representation $\tau$ of $\wt{T}$, denote by $\mathscr{C}^{\infty}_c(\bQ_p,\tau)$ the space of locally constant compactly supported $\tau$-valued functions on $\bQ_p$. We endow it with a $\wt{B}$-action via the formula
\begin{equation}\label{B_action}
	\left[\left(\begin{mat}
		a & b\\0 & d
	\end{mat},\zeta\right).L\right](x)=\tau\left(\left(\begin{mat}
		d & 0\\0 & a
	\end{mat},\zeta (a,d^{-1})\right)\right)L\left(\frac{b+xd}{a}\right),
\end{equation}
for all $\begin{mat}
	a & b\\0 & d
\end{mat}\in B$, $\zeta\in \mu_2$ and $L\in \mathscr{C}^{\infty}_c(\bQ_p,\tau)$.

If $\tau=\mathbf{1}$ is the trivial representation of $T$ viewed as a $\wt{T}$-representation via inflation, we simply write $\mathscr{C}^{\infty}_c(\bQ_p,k)$.

\begin{lem}\label{bigcell_iso}
	We have the chain of $\wt{B}$-equivariant isomorphisms
	\[
	\tilde{\pi}(\chi)^{\operatorname{bc}}\cong \mathscr{C}^{\infty}_c(\bQ_p,\tau_{\chi}) \cong \mathscr{C}^{\infty}_c(\bQ_p,k)\otimes \tau_{\chi^s},
	\]
	where the first isomorphism is given by $f\mapsto \left[x\mapsto f\left(\begin{mat}
		0 & 1\\1 & x
	\end{mat},1\right)\right]$ and the second one is given by $[x\mapsto L(x)v]\mapsfrom L\otimes v$ implicitly identifying $\tau_{\chi^s}=\tau_{\chi}^{\tilde{s}}$.
\end{lem}

\begin{proof}
	The equality $\wt{B}\tilde{s}\wt{B}=\wt{B}\tilde{s}U$ implies the first isomorphism. Here, note that the function $L$ attached to $f$ is locally constant since so is $f$; it has compact support since $f$ is locally constant around the identity element of $\wt{G}$. The last isomorphism is clear. One checks that these are indeed $\wt{B}$-equivariant.
\end{proof}

Define $M_0=\bigoplus_{i,j=0,1} k F_{i,j}^1 \subset \tilde{\pi}(\chi)^{\operatorname{bc}}\cap M$.

\begin{cor}\label{bc_generated}
	The $\wt{B}$-representation $\tilde{\pi}(\chi)^{\operatorname{bc}}$ is generated by $M_0$.
\end{cor}

\begin{proof}
	Under the isomorphism $\tilde{\pi}(\chi)^{\operatorname{bc}}\cong \mathscr{C}^{\infty}_c(\bQ_p,k)\otimes \tau_{\chi^s}$ of the previous lemma, the function $F_{i,j}^1$ corresponds to the element $\mathbf{1}_{\bZ_p}\otimes f_{i,j}$, where $\mathbf{1}_{\bZ_p}$ is the indicator function of $\bZ_p$. It suffices to show that $\mathscr{C}^{\infty}_c(\bQ_p,k)$ is generated as a $B_2$-representation by this indicator function, which is evident from the definition of the action (\ref{B_action}).
\end{proof}

Recall that $X=\begin{mat}
	1 & 1\\ 0 & 1
\end{mat}-1\in k\llbracket U\cap I\rrbracket\cong k\llbracket X\rrbracket$ and $F=\begin{mat}
	p & 0\\0 & 1
\end{mat}$. 

\begin{definition}
	Define $\cT(M)^{\operatorname{bc}}\subset \cT(M)$ to be the $\wt{B}$-subrepresentation generated by $M_0$.
	
	We also let $\cT(M)^{\operatorname{bc}}_0=k\llbracket X\rrbracket[F]M_0\subset \cT(M)^{\operatorname{bc}}$ be the $P^+$-subrepresentation generated by $M_0$; this is $\wt{Z}$-stable.
\end{definition}

\noindent We will prove that $\cT(M)^{\operatorname{bc}}=\cT(M)^{\operatorname{bc}}_0[F^{-1}]$ is irreducible as a $\wt{B}$-representation. In order to do this, we first show that $\cT(M)^{\operatorname{bc}}_0$ is irreducible as a $P^+\wt{Z}$-representation, for which we determine its $X$-torsion.

\begin{lem}\label{recursion_ps}
	For all $i,j=0,1$, up to a non-zero scalar,
	\[
	\sum_{\lambda\in \bF_p} (X+1)^{[\lambda]}F F_{i,j}^1 = F_{i,j+1}^1.
	\]
\end{lem}

\begin{proof}
	On the subspace $\tilde{\pi}(\chi)^{I_1}$, the operator $\sum_{\lambda\in \bF_p} (X+1)^{[\lambda]}F$ is simply the Hecke operator $T_{\Pi^{-1}s}=T_{\Pi^{-1}}T_s$, so the assertion follows from Lemma \ref{action_Fij}.
\end{proof}

\begin{cor}
	For all $i,j=0,1$, we have the following equalities of $X$-torsion subspaces inside the space $\cT(M)^{\operatorname{bc}}_0$.
	\begin{enumerate}
		\item[{\rm (i)}] $\left(k\llbracket X\rrbracket[F^2]F_{i,j}^1\right)[X]=k F_{i,j}^1$;
		\item[{\rm (ii)}] $\left(k\llbracket X\rrbracket[F]F_{i,j}^1\right)[X]=k F_{i,j}^1\oplus k F_{i,j+1}^1$;
		\item[{\rm (iii)}] $\cT(M)^{\operatorname{bc}}_0[X]=M_0$.
	\end{enumerate}
\end{cor}

\begin{proof}
	Applying Lemma \ref{recursion_ps} twice yields, up to a non-zero scalar,
	\begin{equation}\label{it_2}
		F_{i,j}^1 = \sum_{\lambda,\mu\in \bF_p} (X+1)^{[\lambda]+p[\mu]}F^2 F_{i,j}^1.
	\end{equation}
	Iterating this $n$-times shows that $k F_{i,j}^1\subset \left(k\llbracket X\rrbracket F^{2n}	F_{i,j}^1\right)[X] \subset k\llbracket X\rrbracket F^{2n}	F_{i,j}^1$. The latter is a module over $k\llbracket X\rrbracket/F^{2n}XF^{-2n}=k\llbracket X\rrbracket/X^{p^{2n}}$ generated by a single element and thus has $1$-dimensional $X$-torsion, so the first inclusion has to be an equality. Passing to the colimit over all $n\in \mathbf{N}$, which is filtered by (\ref{it_2}), finishes the proof of part (i).
	
	For part (ii), note that Lemma \ref{recursion_ps} implies that $k\llbracket X\rrbracket[F]F_{i,j}^1= k\llbracket X\rrbracket[F^2]F_{i,j}^1 + k\llbracket X\rrbracket[F^2]F_{i,j+1}^1$. Applying part (i) now gives (ii).
	
	Part (iii) can be deduced from (ii) just like (ii) itself was deduced from (i).
\end{proof}	

\begin{prop}\label{bc_irred}
	The $\wt{B}$-subrepresentation $\cT(M)^{\operatorname{bc}}\subset \cT(M)$ is irreducible.
\end{prop}

\begin{proof}
	We first claim that $\cT(M)^{\operatorname{bc}}_0$ is irreducible as a $P^+\wt{Z}$-representation. Indeed, by the previous corollary, $\cT(M)^{\operatorname{bc}}_0[X]=M_0$ and $H$ acts on the basis elements $\{F_{i,j}^1\}_{i,j=0,1}$ via pairwise distinct characters, cf.\ Lemma \ref{action_Fij} (i). Thus, by smoothness, any non-zero $P^+\wt{Z}$-subrepresentation of $\cT(M)^{\operatorname{bc}}_0$ must contain some $F_{i,j}^1$. Using Lemma \ref{recursion_ps} and the action of $\tilde{\Pi}^2$, it must then contain $M_0$ and must therefore be equal to $\cT(M)^{\operatorname{bc}}_0$. The proposition now follows from the claim by writing $\cT(M)^{\operatorname{bc}}=\cT(M)^{\operatorname{bc}}_0[F^{-1}]=\varinjlim_{n\geq 1} F^{-n}\cT(M)^{\operatorname{bc}}_0$.
\end{proof}

\begin{lem}\label{T(M)/bigcell}
	The $\wt{B}$-representation $\cT(M)/\cT(M)^{\operatorname{bc}}$ is equal to the image of $M_0^{\perp}:=\bigoplus_{i,j=0,1} k F_{i,j}^0\subset M\hookrightarrow \cT(M)$
	under the projection map.
\end{lem}

\begin{proof}
	Let $Q$ denote the image of $M_0^{\perp}$ in the quotient $\cT(M)/\cT(M)^{\operatorname{bc}}$. By definition, $\cT(M)$ is generated as a $\wt{G}$-representation by $M$. Writing $\wt{G}=\wt{B}I_1\sqcup \wt{B}I_1 \tilde{\Pi}$ and noting that $\tilde{\Pi}M=MT_{\Pi}^{-1}=M$, it follows from the inclusion $M\subset \cT(M)^{I_1}$ (in fact, $I_1$-stability of $M$ would be enough) that $\cT(M)$ is generated as a $\wt{B}$-representation by $M=M_0\oplus M_0^{\perp}$. As $M_0\subset \cT(M)^{\operatorname{bc}}$, we deduce that $\cT(M)/\cT(M)^{\operatorname{bc}}$ is generated as a $\wt{B}$-representation by $Q$. It therefore suffices to prove that $Q$ is $\wt{B}$-stable.
	
	It is certainly stable under the action of $B\cap I$ and $\tilde{\Pi}^{2\bZ}$, so it remains to show that $Q$ is $F^{\pm 1}$-stable. The subspace $M_0^{\perp}$ is stable under the Hecke operator $T_{s\Pi}=T_sT_{\Pi}$ by Lemma \ref{action_Fij}. The operator $T_{s\Pi}$ is given by $\sum_{\lambda\in \bF_p} \left(\begin{mat}
		1&0 \\ p[\lambda] & 1
	\end{mat},1\right)F^{-1}$, and we claim that for all $m\in M_0^{\perp}$, the element
	\[
	F^{-1}m-mT_sT_{\Pi} =\sum_{\lambda\in \bF_p^{\times}} \begin{mat}
		1&0 \\ p[\lambda] & 1
	\end{mat}F^{-1}m \in \cT(M)^{\operatorname{bc}}
	\]
	lies in $\cT(M)^{\operatorname{bc}}$. Now for all $\lambda\in \bF_p^{\times}$, $\begin{mat}
		1&0 \\ p[\lambda] & 1
	\end{mat}F^{-1}\in Bs(I_1\cap U)=B\Pi (I_1\cap U)$. Since $M_0^{\perp}$ is $(I_1\cap U)$-stable and $\cT(M)^{\operatorname{bc}}$ is $\wt{B}$-stable, it is enough to observe that $\tilde{\Pi}M_0^{\perp}=M_0\subset \cT(M)^{\operatorname{bc}}$, which follows from Lemma \ref{action_Fij}.
\end{proof}
	
\begin{cor}
	The canonical map $\cT(M)=\cT(\cI(\tilde{\pi}(\chi)))\to \tilde{\pi}(\chi)$ is an isomorphism.
\end{cor}

\begin{proof}
	By construction, the canonical map of interest fits into a map of short exact sequences
	\begin{equation}\label{diagram_PS}
	\xymatrix{
		0\ar[r] & \cT(M)^{\operatorname{bc}} \ar[d] \ar[r] & \cT(M) \ar[d] \ar[r] & \cT(M)/\cT(M)^{\operatorname{bc}}\ar[d] \ar[r] & 0\\
		0 \ar[r] & \tilde{\pi}(\chi)^{\operatorname{bc}}\ar[r] & \tilde{\pi}(\chi)\ar[r] & \tau_{\chi} \ar[r] & 0.
	}
	\end{equation}
	By Corollary \ref{bc_generated} and Proposition \ref{bc_irred}, the left vertical arrow is an isomorphism. The right vertical map is an isomorphism by the irreducibility of its $4$-dimensional target and by Lemma \ref{T(M)/bigcell}. In conclusion, also the middle arrow must be an isomorphism. 
\end{proof}

We summarize the results of this section in the following theorem.

\begin{thm}\label{thm_ps}
	The adjoint pair $(\cT,\cI)$ induces a bijection
	\[
	\left\{\begin{array}{c}
		\text{Principal right}\\
		\text{$\cH$-modules}
	\end{array}\right\}\cong \left\{\begin{array}{c}
		\text{Genuine principal series}\\ 
		\text{representations of $\wt{G}$}
	\end{array}\right\},
	\]
	where both sides are considered up to isomorphism.
\end{thm}

\begin{cor}
	Let $\pi$ be a genuine principal series of $\wt{G}$. Then each element in $\pi/X\pi$ is killed by a polynomial in $k[F]$.
\end{cor}

\begin{proof}
	Write $\pi=\cT(M)$ for a principal $\cH$-module $M$. It is enough to show that the statement is true for the kernel and cokernel in the top row of diagram (\ref{diagram_PS}). For the cokernel this is clear as it is finite dimensional. For the kernel, write $\cT(M)^{\operatorname{bc}}=\cT(M)^{\operatorname{bc}}_0[F^{-1}]$ and use Lemma \ref{recursion_ps}, which implies that $\cT(M)^{\operatorname{bc}}_0=X\cT(M)^{\operatorname{bc}}_0$.
\end{proof}

\subsubsection{Extensions of principal series representations}
In this section, we compute the extensions of two genuine principal series representations. Quite surprisingly, this can be deduced from the already existing results for principal series representations of the group $G=\GL_2(\bQ_p)$. In what follows, the Ext-groups appearing will be computed in the category of smooth $k$-representations of $\wt{G}$-, resp.\ $G$.

We will use the same notation for a genuine principal series as in the last section, i.e.\ given a smooth character $\chi\colon T_2\to k^{\times}$, we put $\tilde{\pi}(\chi)=\Ind^{\wt{G}}_{\wt{B}_2}(\chi\boxtimes \iota)$, which up to isomorphism only depends on the restriction of $\chi$ to $T^{\sq}$.
	
Given a smooth character $\chi\colon T\to k^{\times}$, we put
\[
\pi(\chi)=\Ind^G_B(\chi) \text{ and } \tilde{\pi}(\chi)=\tilde{\pi}(\chi|_{T_2})=\Ind^{\wt{G}}_{\wt{B}_2}(\chi|_{T_2}\boxtimes \iota).
\]
	
\begin{prop}[{\cite[Proposition 4.3.15]{ord_parts_II}}]\label{extension_PS_G}
	Let $\chi,\chi'\colon T\to k^{\times}$ be smooth characters.
	\begin{enumerate}
		\item[{\rm (i)}] The space
		\[
		\Ext^1_G(\pi(\chi),\pi(\chi'))=0
		\]
		vanishes
		unless $\chi'=\chi$ or $\chi'=\chi^s \otimes (\omega \otimes\omega^{-1})$.
		\item[{\rm (ii)}] The map
		\[
		\Ext^1_T(\chi,\chi)\xrightarrow{\cong} 		\Ext^1_G(\pi(\chi),\pi(\chi)),
		\]
		induced by the exact functor $\Ind^G_B(-)$, is an isomorphism.
		
		\item[{\rm (iii)}] If $\chi^s\otimes (\omega\otimes \omega^{-1})\neq \chi$, then
		\[
		\dim_k	\Ext^1_G(\pi(\chi),\pi(\chi^s\otimes (\omega\otimes \omega^{-1})))=1.
		\]
	\end{enumerate}
\end{prop}

As in the genuine case, we let $\pi(\chi)^{\operatorname{bc}}\subset \pi(\chi)$ denote the $B$-invariant subspace consisting of those functions whose support is contained in $BsB$ and we identify it $B$-equivariantly with $\mathscr{C}^{\infty}_c(\bQ_p,k)\otimes \chi^s$. The key ingredient for this section is the $\wt{B}$-equivariant isomorphism 
\begin{equation}\label{connection}
	\tilde{\pi}(\chi)^{\operatorname{bc}}\cong \Ind^{\wt{B}}_{\wt{B}_2}(\pi(\chi)^{\operatorname{bc}}|_{B_2}\boxtimes \iota)
\end{equation}
resulting from Lemma \ref{bigcell_iso}.

Evaluation at the identity yields a short exact sequence
\begin{equation}\label{ses_for_B}
	0\to \pi(\chi)^{\operatorname{bc}}\to \pi(\chi)\to \chi\to 0
\end{equation}
of $B$-representations. 

\begin{lem}\label{Ext_bc_G} Let $\chi,\chi'\colon T\to k^{\times}$ be smooth characters. 
	\begin{itemize}
		\item[{\rm (i)}] Assume that $\chi'\neq \chi$. The short exact sequence (\ref{ses_for_B}) induces a natural isomorphism
		\[
		\Ext^1_G(\pi(\chi),\pi(\chi'))\cong \Ext^1_B(\pi(\chi)^{\operatorname{bc}},\chi').
		\]
		\item[{\rm (ii)}] Assume that $\chi'=\chi$ or $\chi' \neq \chi^s\otimes (\omega\otimes \omega^{-1})$. The canonical map
		\[
		\Ext^1_B(\pi(\chi)^{\operatorname{bc}},\chi')\hookrightarrow \Ext^2_B(\chi,\chi'),
		\]
		induced by the sequence (\ref{ses_for_B}),
		is injective.
	\end{itemize}
\end{lem}

\begin{proof}
	By Shapiro's Lemma we have $\Ext^1_G(\pi(\chi),\pi(\chi'))\cong \Ext^1_G(\pi(\chi)|_B,\chi'))$.
	The short exact sequence (\ref{ses_for_B}) induces an exact sequence
	\begin{equation}\label{les}
		0\to \Ext^1_B(\chi,\chi')\to \Ext^1_B(\pi(\chi)|_B,\chi')\to \Ext^1_B(\pi(\chi)^{\operatorname{bc}},\chi')\to \Ext^2_B(\chi,\chi').
	\end{equation}
	The zero on the left-hand side holds because $\Hom_B(\pi(\chi)^{\operatorname{bc}},\chi')=0$ (as both are irreducible, but non-isomorphic). By \cite[Lemma 4.3.7]{ord_parts_II}, we have $\Ext^{i}_B(\chi,\chi')=\Ext^{i}_T(\chi,\chi')$ for $0\leq i\leq 2$.
	
	Assume that $\chi'\neq \chi$. By \cite[Lemma 4.3.10]{ord_parts_II}, we then have $\Ext^{i}_T(\chi,\chi')=0$ for all $i\geq 0$, which gives the assertion in (i).
	
	Assume now that $\chi'=\chi$ or $\chi' \neq \chi^s\otimes (\omega\otimes \omega^{-1})$. By Proposition \ref{extension_PS_G}, the map $\Ext^1_B(\chi,\chi')\to \Ext^1_B(\pi(\chi)|_B,\chi')$ is surjective proving (ii). Indeed, if $\chi'=\chi$, then it is even isomorphism; if $ \chi\neq \chi' \neq \chi^s\otimes (\omega\otimes \omega^{-1})$, then the target is trivial.
\end{proof}

Recall from Notation \ref{notation_[i,j]} that for $i,j\in \bZ/2\bZ$ and a smooth representation of a subgroup of $\wt{B}$, we put $\pi[i,j]=\pi\otimes (\omega^{i\frac{p-1}{2}}\otimes \omega^{j\frac{p-1}{2}})$. For the purpose of this section, we also define
\[
\pi(i,j)=\pi\otimes (\mu_{-1}^{i}\otimes \mu_{-1}^j),
\]
where, as introduced earlier in Secion \ref{prelim_and_not}, $\mu_{-1}$ is the smooth unramified character mapping $p$ to $-1$.

\begin{lem}\label{good_extension}
	Let $\chi\colon T^{\sq}\to k^{\times}$ be a smooth character such that $\chi=\chi^s\otimes (\omega\otimes \omega^{-1})$ (as characters on $T^{\sq}$). Then there exists a smooth character $\chi_e\colon T\to k^{\times}$ such that $\chi_e|_{T^{\sq}}=\chi$ and $\chi_e=\chi_e^s \otimes (\omega\otimes \omega^{-1})$ (as characters on $T$).
\end{lem}

\begin{proof}
	Let $\psi\colon T\to k^{\times}$ be any smooth character extending $\chi$, i.e.\ $\psi|_{T^{\sq}}=\chi$. The assumption on $\chi$ means that $\psi[i,j](a,b)=\psi^{s}\otimes (\omega\otimes \omega^{-1})$ for some (unique) quadruple $(i,j,a,b)$ of elements in $\bZ/2\bZ$. One quickly verifies that necessarily $i=j$ and $a=b$ as elements of $\bZ/2\bZ$. Then $\chi_e:=\psi[0,j](0,b)$ does the job.
\end{proof}

\begin{prop}\label{extension_PS}
	Let $\chi,\chi'\colon T\to k^{\times}$ be smooth characters.
	\begin{enumerate}
		\item[{\rm (i)}] The space
		\[
		\Ext^1_{\wt{G}}(\tilde{\pi}(\chi),\tilde{\pi}(\chi'))=0
		\]
		vanishes
		unless $\chi'|_{T^{\sq}}=\chi|_{T^{\sq}}$ or $\chi'|_{T^{\sq}}=\chi^s \otimes (\omega \otimes\omega^{-1})|_{T^{\sq}}$.
		
		\item[{\rm (ii)}] The map
		\[
		\Ext^1_{T_2}(\chi|_{T_2},\chi|_{T_2})\xrightarrow{\cong} 		\Ext^1_{\wt{G}}(\tilde{\pi}(\chi),\tilde{\pi}(\chi)),
		\]
		induced by the exact functor $\Ind^{\wt{G}}_{\wt{B}_2}(-\boxtimes \iota)$, is an isomorphism.
		
		\item[{\rm (iii)}] If $\chi^s\otimes (\omega\otimes \omega^{-1})|_{T^{\sq}}\neq \chi|_{T^{\sq}}$, then
		\[
		\dim_{k}	\Ext^1_{\wt{G}}(\tilde{\pi}(\chi),\tilde{\pi}(\chi^s\otimes (\omega\otimes \omega^{-1})))=1.
		\]
	\end{enumerate}
\end{prop}

\begin{proof} Assume first that $\chi'|_{T^{\sq}}\neq \chi|_{T^{\sq}}$, i.e.\ $\chi'\neq \chi[i,j](a,b)$ for any $i,j,a,b\in \bZ/2\bZ$. Writing $\tau_{\chi}|_{B_2}=\oplus_{i,j\in \bZ/2\bZ} \chi[i,j]$, we deduce that
	\[
	\Ext^1_{\wt{G}}(\tilde{\pi}(\chi),\tilde{\pi}(\chi'))\cong \Ext^1_{B_2}(\tilde{\pi}(\chi)^{\operatorname{bc}}|_{B_2},\chi')
	\]
	just as in the proof of Lemma \ref{Ext_bc_G} (i). Restricting the isomorphism (\ref{connection}) to $B_2$ and then moving the resultings twists (obtained from Mackey's decomposition and Lemma \ref{conj_ij}) over to $\chi'$, Frobenius reciprocity implies that the above Ext-space is isomorphic to
	\[
	\bigoplus_{i,j\in \bZ/2\bZ} \Ext^1_{B}(\pi(\chi)^{\operatorname{bc}},\Ind^{B}_{B_2}(\chi'[i,j]))=\bigoplus_{i,j,a,b\in \bZ/2\bZ} \Ext^1_{B}(\pi(\chi)^{\operatorname{bc}},\chi'[i,j](a,b)).
	\]	
	By our assumption and Lemma \ref{Ext_bc_G} (i), this is isomorphic to
	\[
	\bigoplus_{i,j,a,b\in \bZ/2\bZ} \Ext^1_G(\pi(\chi),\pi(\chi'[i,j](a,b))),
	\]
	which by Proposition \ref{extension_PS_G} is non-zero (and then $1$-dimensional) if and only if $\chi'[i,j](a,b)=\chi^s\otimes (\omega\otimes \omega^{-1})$ for some (necessarily unique) quadruple $(i,j,a,b)$ of elements in $\bZ/2\bZ$, or equivalently if and only if $\chi'|_{T^{\sq}}=\chi^s\otimes(\omega\otimes \omega^{-1})|_{T^{\sq}}$. This proves (i) and (iii).
	
	To prove (ii), note that the statement only depends on the restriction of $\chi$ to $T^{\sq}$. If $\chi=\chi^s \otimes (\omega\otimes \omega^{-1})$ as characters of $T^{\sq}$, Lemma \ref{good_extension} allows us to assume that this equality holds on the full group $T$, which we will do. Using the evident analogue of the long exact sequence (\ref{les}), we reduce to showing the injectivity of the map
	\begin{equation}\label{injective?}
		\Ext^1_{B_2}(\tilde{\pi}(\chi)^{\operatorname{bc}}|_{B_2},\chi)\to \Ext^2_{B_2}(\tau_{\chi}|_{T_2},\chi).
	\end{equation}
	As before, the left-hand side is equal to $\oplus_{i,j,a,b\in \bZ/2\bZ} 	\Ext^1_B(\pi(\chi)^{\operatorname{bc}},\chi[i,j](a,b))$. The right-hand side is similarly equal to $\oplus_{i,j,a,b\in \bZ/2\bZ} \Ext^2_{B}(\chi,\chi[i,j](a,b))$ and, by \cite[Lemma 4.3.7, Lemma 4.3.10]{ord_parts_II}, this equals $\Ext^2_{B}(\chi,\chi)$. By Proposition \ref{extension_PS_G}, the extension space $\Ext^1_B(\pi(\chi)^{\operatorname{bc}},\chi[i,j](a,b))$ is trivial unless\footnote{Here we are using that we chose $\chi$ so that $\chi=\chi^s\otimes (\omega\otimes \omega^{-1})$ if this happens to be true on $T^{\sq}$.} $i=j=a=b=0$ in $\bZ/2\bZ$. Thus, the injectivity of (\ref{injective?}) is equivalent to the injectivity of the connecting homomorphism
	\[
	\Ext^1_B(\pi(\chi)^{\operatorname{bc}},\chi)\to \Ext^2_B(\chi,\chi).
	\]
	This is the content of Lemma \ref{Ext_bc_G} (ii).
\end{proof}

\begin{remark}\label{Iwahori_ext}
	Twisting induces an isomorphism $\Ext^1_{T_2}(\chi|_{T_2},\chi|_{T_2})\cong \Ext^1_{T_2}(\mathbf{1},\mathbf{1})$, which in turn can be identified with $\Hom^{\operatorname{cts}}(T_2,\ol{\bF}_p)$. The latter contains the subspace of unramified homomorphisms, i.e.\ those trivial on $T_2\cap K$, and the resulting inclusion induces a commutative diagram
	\[
	\xymatrix{
		\Hom^{\operatorname{cts}}(T_2,k) \ar[r]^{\cong} & \Ext^1_{T_2}(\chi|_{T_2},\chi|_{T_2})\ar[r]^{\cong} & \Ext^1_{\wt{G}}(\tilde{\pi}(\chi),\tilde{\pi}(\chi))\\
		\Hom^{\operatorname{unr}}(T_2,k) \ar@{^{(}->}[u]\ar[rr]_{\cong} && \Ext^1_{\cH}(\cI(\tilde{\pi}(\chi)),\cI(\tilde{\pi}(\chi))). \ar@{^{(}->}[u]_{\cT}
	}
	\]
	Looking only at those continuous homomorphisms killing the scalar matrix $p^2$ yields the isomorphism $\Hom_{p^2=0}^{\operatorname{unr}}(T_2,k) \cong \Ext^1_{\cH_{p^2=1}}(\cI(\tilde{\pi}(\chi)),\cI(\tilde{\pi}(\chi)))$. In particular, the latter Ext-space is $1$-dimensional, cf.\ Proposition \ref{self_ext_pro_hecke_reg}.
\end{remark}

\begin{remark}
	Consider the diagram induced by evaluation at the identity,
	\[
	\xymatrix{
		0\ar[r] & \tilde{\pi}(\chi)^{\operatorname{bc}}\ar[r] & \tilde{\pi}(\chi) \ar[r] & \tau_{\chi}\ar[r] & 0\\
		0\ar[r] & \pi(\chi)^{\operatorname{bc}}\ar[r] \ar[u]& \pi(\chi)\ar[r] & \chi\ar[r] \ar[u] & 0,
	}
	\]
	where the left vertical $B_2$-equivariant map is given by sending $f\in \pi(\chi)^{\operatorname{bc}}$ to $\tilde{f}\in \tilde{\pi}(\chi)^{\operatorname{bc}}$ uniquely determined by $\tilde{f}\left(\left(\begin{mat}
		0 & 1\\1 & x
	\end{mat},1\right)\right)=f\left(\begin{mat}
		0 & 1\\1 & x
	\end{mat}\right)f_{0,0}$ for all $x\in \bQ_p$, and the right vertical $B_2$-equivariant map is given by mapping $1\in \chi$ to the standard basis vector $f_{0,0}\in \tau_{\chi}$. By the isomorphism (\ref{connection}) and the very definition of $\tau_{\chi}$, the $\wt{B}$-equivariant maps obtained via Frobenius reciprocity from the outer vertical arrows are isomorphisms. It is tempting to believe that this is also true for $\tilde{\pi}(\chi)$ in a compatible way, i.e.\ such that the top sequence in the above diagram is obtained from the lower one by applying the functor $\Ind^{\wt{B}}_{\wt{B}_2}((-)\boxtimes \iota)$. However, one may check that this diagram cannot be filled in, which also seems to explain why the application of Lemma \ref{good_extension} in the proof of the previous proposition is crucial.
\end{remark}

\subsection{Supersingular representations}
\begin{definition}
	A \textit{genuine supersingular representation} of $\wt{G}$ is a smooth irreducible genuine $k$-representation of $\wt{G}$, which is not isomorphic to a subquotient of a principal series. 
\end{definition}

\begin{remark}\label{ss=notps}
	By Proposition \ref{PS_irred}, every genuine principal series is irreducible, so removing ``a subquotient of'' gives an equivalent definition.
\end{remark}

The goal of this section is to prove that the adjoint pair $(\cT,\cI)$ induces a bijection between the set of isomorphism classes of supersingular $\cH$-modules and the set of isomorphism classes of \textit{admissible} genuine supersingular representations. We will later show that any smooth irreducible representation of $\wt{G}$ is already admissible, see Theorem \ref{fg+adm=fl}. For showing that $\cT(M)$ is irreducible, for a supersingular Hecke module $M$, we will use a weight-cycling argument, that has already been used for supersingular representations of $G$, see \cite[§4.1]{Paskunas_ext_ss} or \cite[§10]{herzig}.

\subsubsection{Basic lemma}\label{basic_lemma_section}
In a previous draft of this work we worked out the weight-cycling in the specific case of interest (in dimension four). However, in a lecture \cite{herzig_notes} of Florian Herzig at the Spring School on the mod-$p$ Langlands Correspondence, which took place in Essen (online) in 2021, he formulated it more generally (in dimension $n\geq 1$) and called the resulting statement ``Basic Lemma''. In this subsection, we simply reproduce his result.\\
	
\textbf{Setup.} Let $\pi$ be a smooth $k$-representation of the mirabolic subgroup $P$. Suppose that we are given $k$-linearly independent $\Gamma$-eigenvectors $v_1,\ldots,v_n\in \pi[X]$ killed by $X$, satisfying
\begin{equation}\label{cycle_equation}
	X^{s_i}F(v_i)=c_i v_{i+1} \text{ for some $s_i\in \bZ_{\geq 0}$ and $c_i\in k^{\times}$}
\end{equation}
for all $1\leq i\leq n$.
Here and in what follows we understand the subscript $i$ to be the unique element in $\{1,\ldots,n\}$ congruent to $i$ modulo $n$. We define
\[
M=\bigoplus_{i=1}^n kv_i
\]
and let
\[
\pi_M=k\llbracket X\rrbracket[F]M \subset \pi.
\]
be the $k\llbracket X\rrbracket[F]$-submodule generated by $M$ in $\pi$. Finally, for $1\leq i\leq n$, we put
\[
\pi_i=k\llbracket X\rrbracket[F^n]v_i\subset \pi_M.
\]

We wish to prove the following result.

\begin{lem}[{Basic Lemma}]\label{basic_lemma}
	\begin{enumerate}
		\item[{\rm (i)}] For each $1\leq i\leq n$, the subspace $\pi_i\subset \pi$ is $\Gamma$-stable.
		\item[{\rm (ii)}] The subspace $\pi_M\subset \pi$ is $P^+$-stable.
		\item[{\rm (iii)}] For each $1\leq i\leq n$, $\pi_i[X]=kv_i$.
		\item[{\rm (iv)}] As $k\llbracket X\rrbracket[F^n,\Gamma]$-modules, $\pi_M=\bigoplus_{i=1}^n \pi_i$.
		\item[{\rm (v)}] If $s_i=0$ for all $1\leq i\leq n$, then $\pi_i=kv_i$.
		\item[{\rm (vi)}] If $s_j>0$ for some $1\leq j\leq n$, then $\pi_i\cong \mathscr{C}^{\infty}(U\cap I,k)$ as $U\cap I$-representations and the inclusion $kv_i\subset \pi_i$ is an injective envelope in the category of smooth $U\cap I$-representations, for all $1\leq i\leq n$.
		\item[{\rm (vii)}] Assume that the following condition is satisfied: If for some $1\leq i\neq j\leq n$, the vectors $v_i,v_j$ have the same $\Gamma$-eigencharacter, then $s_{i}\neq s_j$. Then $\pi_M$ is an irreducible $P^+$-representation.
	\end{enumerate}
\end{lem}
	
This is a more detailed version of the Basic Lemma in Herzig's lecture alluded to above. The last part on the irreducibility was for instance not contained in Herzig's version. However, the additional details are essentially just a byproduct of the proof of the original version. Before jumping into the proof, we make some preparations.
	
\begin{definition}
	For $1\leq i\leq n$, define 
	\[
	e(i)=p^{n-1}s_i+p^{n-2}s_{i+1} + \ldots + p^2 s_{i+n-3}+ps_{i+n-2}+s_{i+n-1}
	\]
	and, if $m\in \bZ_{\geq 1}$, put
	\[
	e(i)_m=\sum_{j=0}^{m-1} p^{jn}e(i).
	\]
	We let
	\[
	c=\prod_{i=1}^n c_i
	\]
	be the product of the scalars appearing in (\ref{cycle_equation}).
\end{definition}

\begin{lem}\label{prep_lem} Let $1\leq i\leq n$. The following hold true:
	\begin{enumerate}
		\item[{\rm (i)}] $X^{p^{n-1}s_i+p^{n-2}s_{i+1} + \ldots + p^2 s_{i+n-3}+ps_{i+n-2}}F^{n}(v_i)=\left(\prod_{j\neq i-1} c_j\right) F(v_{i-1})$;
		\item[{\rm (ii)}] for all $m\in \bZ_{\geq 1}$, $X^{e(i)_m}F^{nm}(v_i)=c^mv_i$.
	\end{enumerate}
\end{lem}

\begin{proof}
	Part (i) follows from the relation $FX=X^pF$ and inductively applying the equations (\ref{cycle_equation}) starting with $X^{s_i}F(v_i)=c_i v_{i+1}$. The case of $m=1$ in part (ii) now follows from this by applying $X^{s_{i-1}}$ to (i) and using (\ref{cycle_equation}). Inductively applying the resulting relation gives the assertion for an arbitrary $m\in \bZ_{\geq 1}$.
\end{proof}

\begin{proof}[Proof of the Basic Lemma \ref{basic_lemma}]
	Part (i) and (ii) follow from the relations $\gamma F=F\gamma$ and $\gamma X=((X+1)^{\gamma}-1)\gamma$ for all $\gamma\in \Gamma$. By Lemma \ref{prep_lem} (i), $F\pi_{i-1}\subset \pi_i$ and thus $\pi_M=\sum_i \pi_i$. In particular, part (iv) will follow from part (iii) as the vectors $v_1,\ldots,v_n$ are linearly independent. 
	
	Fix $1\leq i\leq n$. For $m\in \bZ_{\geq 1}$, let $\pi_{i,m}=k\llbracket X\rrbracket F^{nm}(v_i)$ be the $k\llbracket X\rrbracket$-submodule of $\pi_i$ generated $F^{nm}(v_i)$. Applying $F^{nm}$ to Lemma \ref{prep_lem} (ii) (with $m=1$), shows that $kv_i\subset \pi_{i,m}\subset \pi_{i,m+1}$. By Lemma \ref{prep_lem} (ii), the $k\llbracket X\rrbracket$-linear map
	\begin{align}\label{iso_pi_{i,m}}
		k\llbracket X\rrbracket/(X^{e(i)_{m}+1}) & \xrightarrow{\cong} \pi_{i,m}\\
		f(X)& \mapsto c^{-m}f(X)F^{nm}(v_i)\nonumber
	\end{align}
	is a well-defined isomorphism. Comparing $X$-torsion shows that $\pi_{i,m}[X]$ is $1$-dimensional, necessarily equal to $kv_i$. This proves (iii). The isomorphisms (\ref{iso_pi_{i,m}}) fit into a commutative diagram
	\begin{equation}\label{com_square}
		\begin{gathered}
			\xymatrix@C+2pc{
				\pi_{i,m} \ar@{^{(}->}[r] & \pi_{i,m+1}\\
				k\llbracket X\rrbracket/(X^{e(i)_m+1}) \ar[u]^{\cong} \ar@{^{(}->}[r]^{\cdot X^{p^{nm}e(i)}} & k\llbracket X\rrbracket /(X^{e(i+1)_{m+1}+1}) \ar[u]_{\cong}.
			}	
		\end{gathered}
	\end{equation}
	For $d\in \bZ_{\geq 1}$, we identify the $k$-linear dual $\left(k\llbracket X\rrbracket/(X^{d+1})\right)^{\vee}$ with $k\llbracket X\rrbracket/(X^{d+1})$ via the $k\llbracket X\rrbracket$-linear map corresponding to the projection $\operatorname{pr}_d\in \left(k\llbracket X\rrbracket/(X^{d+1})\right)^{\vee}$ defined by $\operatorname{pr}_d(X^{j})=\delta_{jd}$, the Kronecker delta, for $0\leq j\leq d$. Under these identifications, the Pontryagin dual of the lower horizontal map is simply the canonical projection. Passing to the inverse limit, we obtain
	\[
	\pi_i^{\vee}=(\varinjlim_{m\geq 1}\pi_{i,m})^{\vee} =\varprojlim_{m\geq 1} \pi_{i,m}^{\vee}\cong \varprojlim_{m\geq 1} k\llbracket X\rrbracket/(X^{e(i)_m+1}).
	\]
	If $s_j=0$ for all $1\leq j\leq n$, then $e(i)_m=0$, so we obtain $\pi_i=kv_i$, which proves (v). If however $s_j$ is non-zero for some $1\leq j\leq n$, then $e(i)_m$ tends to infinity as $m\to \infty$. In this case we therefore obtain $\pi_i^{\vee}\cong k\llbracket X\rrbracket$ as $k\llbracket X\rrbracket$-modules, or equivalently $\pi_i\cong \mathscr{C}^{\infty}(U\cap I,k)$ as smooth $U\cap I$-representations. In particular, $\pi_{i}$ is injective in the category of smooth $U\cap I$-modules, and since $\pi_i[X]=kv_i$, we deduce that the inclusion $kv_i\subset \pi_i$ is an injective envelope proving part (vi).
	
	Assume now that the condition formulated in (vii) holds.
	Let $\tau\subset \pi_M$ be a non-zero $P^+$-subrepresentation (or equivalently a $k\llbracket X\rrbracket[F,\Gamma]$-submodule). Since $\pi_M[X]=\oplus_{i=1}^n kv_i$ is a $\Gamma$-stable subspace, there exists a subset $\{i_1,\ldots,i_r\}\subset \{1,\ldots,n\}$, for some $r\leq n$, such that
	\[
	\sum_{j=1}^r a_{i_j} v_{i_j}\in \tau \text{ for some $a_{i_j}\in k^{\times}$,}
	\]
	where the $v_{i_j}$, for $1\leq j\leq r$, have the same $\Gamma$-eigencharacter. After reordering we may assume that $s_{i_{j+1}}\geq s_{i_j}$ for all $1\leq j\leq r-1$. Applying $X^{1+s_{i_{r-1}}}F(-)$ will kill $v_{i_j}$ for $1\leq j\leq r-1$. Here we are using equation (\ref{cycle_equation}) and the fact the $Xv_{j}=0$ for all $1\leq j\leq n$. Thus, $a_{i_r} X^{1+s_{i_{r-1}}}F(v_{i_r})\in \tau$. Since $s_{i_{r-1}}\neq s_{i_r}$ (and so $s_{i_{r-1}}<s_{i_r}$) by our assumption,  applying $X^{s_{i_r}-s_{i_{r-1}}-1}$ and using equation (\ref{cycle_equation}) once again, we see that $v_{i_r + 1}\in \tau$. Another application of the same equation now shows that $\tau$ contains $v_j$ for all $1\leq j\leq n$ and must therefore be equal to $\pi_M$.
\end{proof}

\begin{cor}\label{basic_lemma_application}
	Assume that the following condition is satisfied: If for some $1\leq i\neq j\leq n$, the vectors $v_i,v_j$ have the same $\Gamma$-eigencharacter, then $s_{i}\neq s_j$. Then the subspace
	\[
	\pi_M[F^{-1}]=\varinjlim_{m\geq 1} F^{-m}\pi_M \subset \pi
	\]
	is $P$-stable and -irreducible.
\end{cor}

\begin{proof}
	By the Basic Lemma \ref{basic_lemma}, the subspace $\pi_M\subset \pi$ is $P^+$-stable and -irreducible. The equality $P=F^{-\mathbf{N}}P^+$ now implies the assertion.
\end{proof}

\subsubsection{From supersingular Hecke modules to supersingular representations} 
Let $M$ be a supersingular $\cH(O)$-module for some orbit $O$. Applying the Basic Lemma, we now prove that the $\cT(M)$ is irreducible as a $P$-representation (hence as a $\wt{G}$-representation), and that the canonical map $M\to \cI(\cT(M))$ is an isomorphism. Using Lemma \ref{twist_lemma}, we may and do assume that, in the non-square-regular case, $\mathbf{1}\in O$. We fix a character $\chi\in O$, with $\chi=\mathbf{1}$ in the non-square-regular case. By Definition \ref{def_simple_reg}, resp.\ \ref{def_simple_nreg}, we can write
\[
M=\begin{cases}
	\operatorname{SS}(\chi,z) \text{ if $\chi$ is square-regular}\\
	\operatorname{SS}(\lambda,z) \text{ if $\chi=\mathbf{1}$}
\end{cases}
\]
for some $z\in k^{\times}$ and $\lambda\in \{0,-1\}$. By Proposition \ref{simple_regular_ss}, resp.\ \ref{simple_non-regular}, we have
\begin{equation}\label{M_direct_sum}
	M=\bigoplus_{\psi\in \chi^{\tilde{\Pi}^{\bZ}}} Me_{\psi} \text{ and } \dim_{\ol{\bF}_p}Me_{\psi}=1,
\end{equation}
where $\chi^{\tilde{\Pi}^{\bZ}}$ denotes the $\Pi$-orbit of $\chi$ as introduced in Definition \ref{definition_orbit}.

It is explained in Remark \ref{T(ss)_non_zero} below that $\cT(M)$ is non-zero, so that the canonical map $M\hookrightarrow \cI\cT(M)$ is an injection, via which we will view elements of $M$ as elements of $\cT(M)$.

\begin{lem}\label{T_s=scalar}
	For each $\psi\in \chi^{\tilde{\Pi}^{\bZ}}$, the Hecke operator $T_s$ acts on $Me_{\psi}$ via a scalar. More precisely:
	\begin{enumerate}
		\item[{\rm (i)}] If $O$ is square-regular, then $T_s=0$ on $M$.
		\item[{\rm (ii)}] If $O$ is non-square-regular, then $T_s$ acts on $Me_{\psi}$ via $\lambda$ (resp. $0$) if $\psi=\psi^s$ (resp.\ $\psi\neq \psi^s$).
	\end{enumerate}
\end{lem}

\begin{proof}
	The decomposition (\ref{M_direct_sum}) together with the relation $e_{\psi}T_s=T_se_{\psi^s}$ of Lemma \ref{identities_Hecke_operators_I} (iv) already implies (i). For (ii), note that by applying the element $T_{\Pi}^2$ and using Lemma \ref{identities_Hecke_operators_II} (iv) and (v), it suffices to treat the characters $\psi=\mathbf{1},\mathbf{1}[0,1]$. Using part (v) once again and applying the invertible element $T_{\Pi}$ from the right, the vanishing of $Me_{\mathbf{1}[0,1]}T_s=Me_{\mathbf{1}}T_{\Pi}T_s$ is equivalent to the vanishing of $Me_{\mathbf{1}}T_{\Pi}T_sT_{\Pi}$, which by parts (vi) and (ix) of the cited lemma means that $T_{\Pi s\Pi}$, or equivalently $S_{0,2}\in \cH(\mathbf{1})$, kills $Me_{\mathbf{1}}$. This holds true by the very definition of $M$, which also says that $T_s$, or equivalently $S_{0,0}\in \cH(\mathbf{1})$, acts on $Me_{\mathbf{1}}$ by the scalar $\lambda$.
\end{proof}

Recall from Definition \ref{ss_weights} that to a character $\psi\in \Hhat$ we attached the semi-simple smooth $K$-representation $W(\psi)$, that is irreducible if and only if $\psi\neq \psi^s$.

\begin{cor}\label{generates_weight}
	For each $\psi\in \chi^{\tilde{\Pi}^{\bZ}}$, the $\wt{K}$-subrepresentation $\sigma(\psi):=\left\langle \wt{K}Me_{\psi}\right\rangle\subset \cT(M)$ generated by $Me_{\psi}$ is irreducible. More precisely:	
	\begin{enumerate}
		\item[{\rm (i)}] If $O$ is square-regular, then $\sigma(\psi)=W(\psi)\boxtimes \iota$.
		\item[{\rm (ii)}] If $O$ is non-square-regular and $\psi \neq \psi^{s}$, then $\sigma(\psi)=W(\psi)\boxtimes \iota$, while
		\[
		\sigma(\mathbf{1})=\begin{cases}
			\mathbf{1}\boxtimes \iota \text{ if $\lambda=0$}\\
			\Sym^{p-1}(k^2)\boxtimes \iota \text{ if $\lambda = -1$}
		\end{cases} \text{ and \phantom{m}} \sigma(\mathbf{1}[1,1])=\sigma(\mathbf{1})\otimes {\det}^{\frac{p-1}{2}}.
		\] 
	\end{enumerate}
\end{cor}

\begin{proof}
	Frobenius reciprocity yields a $\wt{K}$-equivariant surjection $\Ind^{\wt{K}}_{\wt{I}}(\psi\boxtimes \iota)\twoheadrightarrow \sigma(\psi)$ mapping the $I_1$-invariant function $\varphi$ with $\supp(\varphi)=\wt{I}$ and $\varphi(1)=1$ to a fixed element in $Me_{\psi}$. By Lemma \ref{T_s=scalar}, the operator $T_s$ acts on $Me_{\psi}$ by a scalar. However, it does not act by a scalar on $\varphi$, and so the surjection cannot be an isomorphism, i.e.\ the kernel is non-trivial.
	
	Assume that $\psi\neq \psi^s$. By Proposition \ref{JH_weight}, we have a short exact sequence
	\[
	0\to W(\psi^s)\boxtimes \iota\to \Ind^{\wt{K}}_{\wt{I}}(\psi\boxtimes \iota)\to W(\psi)\boxtimes \iota\to 0
	\]
	of $\wt{K}$-representation with the outer terms being irreducible. Since this sequence is non-split, we deduce that $\sigma(\psi)\cong W(\psi)\boxtimes \iota$.
	
	Assume now that $\psi=\psi^s$. By Lemma \ref{Ztilde_action}, the element $\tilde{\Pi}^2$ defines a $\wt{K}$-equivariant isomorphism $\cT(M)\cong \cT(M)\otimes \omega^{\frac{p-1}{2}} \circ {\det}$ restricting to an isomorphism $\sigma(\mathbf{1}[1,1])\cong \sigma(\mathbf{1})\otimes {\det}^{\frac{p-1}{2}}$, so it is enough to treat the case $\psi=\mathbf{1}$. By Proposition \ref{JH_weight} once again, we have $\Ind^{\wt{K}}_{\wt{I}}(\mathbf{1}\boxtimes \iota)=\mathbf{1}\boxtimes \iota \oplus \Sym^{p-1}(k^2)\boxtimes \iota$. The $I_1$-invariants of the second summand are $1$-dimensional generated by $\varphi T_s$. Hence, the kernel of $\Ind^{\wt{K}}_{\wt{I}}(\mathbf{1}\boxtimes \iota)\twoheadrightarrow \sigma(\psi)$ meets $\Sym^{p-1}(k^2)\boxtimes \iota$ trivially if and only if $\varphi T_s$ does not lie in the kernel, i.e.\ if and only if $T_s$ does not kill $Me_{\mathbf{1}}$, which by Lemma \ref{T_s=scalar} just means that $\lambda = -1$.
\end{proof}

By Proposition \ref{class_weights}, we may write, for each $\psi\in \chi^{\tilde{\Pi}^{\bZ}}$,
\[
\sigma(\psi)=\Sym^{r_{\psi}}(k^2)\otimes {\det}^{b_{\psi}}
\]
for a unique $0\leq r_{\psi}\leq p-1$ and $0\leq b_{\psi}\leq p-2$.

\begin{remark}\label{weight_remark}
	By Lemma \ref{Ztilde_action}, the element $\tilde{\Pi}^2$ defines an isomorphism $\sigma(\psi[1,1])\cong \sigma(\psi)\otimes \omega^{\frac{p-1}{2}}\circ {\det}$, i.e.\ $r_{\psi[1,1]}=r_{\psi}$ and $b_{\psi[1,1]}\equiv b_{\psi}+\frac{p-1}{2} \bmod p-1$. It follows more generally from Lemma \ref{generates_weight} that a single weight $\sigma(\psi)$ determines all the other ones, namely
	\[
	b_{\psi^s[1,0]}\equiv b_{\psi}+r_{\psi} \bmod p-1 \text{ and } r_{\psi^s[1,0]}=\begin{cases}
		\frac{p-1}{2}-r_{\psi} \text{ if $0\leq r_{\psi}\leq \frac{p-1}{2}$}\\
		\frac{3(p-1)}{2}-r_{\psi} \text{ if $\frac{p-1}{2}\leq r_{\psi}\leq p-1$.}
	\end{cases}
	\]
\end{remark}

\begin{cor}\label{recursion_formula}
	For each $\psi\in \chi^{\tilde{\Pi}^{\bZ}}$ and $v_{\psi}\in Me_{\psi}$,
	\[
	X^{r_{\psi^s[1,0]}}F(v_{\psi})=(-1)^{b_{\psi^s[1,0]}} r_{\psi^s[1,0]}! v_{\psi}T_{\Pi}^{-1}.
	\]
\end{cor}

\begin{proof}
	We may assume that $v_{\psi}\neq 0$. Since $F=\tilde{s}\tilde{\Pi}$ and $\tilde{\Pi}$ acts on $M$ via the Hecke operator $T_{\Pi}^{-1}$, we need to show that
	\begin{equation}\label{explicit_weight}
		X^{r_{\psi^s[1,0]}}\tilde{s} (v_{\psi}T_{\Pi}^{-1}) = (-1)^{b_{\psi^s[1,0]}} r_{\psi^s[1,0]}! v_{\psi}T_{\Pi}^{-1}.
	\end{equation}
	The relation $e_{\psi}T_{\Pi}^{-1}=T_{\Pi}^{-1}e_{\psi^s[1,0]}$ of Lemma \ref{identities_Hecke_operators_II} (i) implies that $0\neq v_{\psi}T_{\Pi}^{-1}\in Me_{\psi^{s}[1,0]}\subset \sigma(\psi^s[1,0])^{I_1}$. By Lemma \ref{generates_weight}, $\sigma(\psi^s[1,0])$ is a weight. Therefore, equation (\ref{explicit_weight}) is an explicit computation in $\sigma(\psi^s[1,0])=\Sym^{r_{\psi^s[1,0]}}(k^2)\otimes {\det}^{b_{\psi^s[1,0]}}$, that is for example carried out in \cite[Lemma 3.2]{Paskunas_ext_ss}.
\end{proof}

We now bring ourselves into the following setup: Fix $0\neq v_1\in Me_{\chi}$ and put $v_i=v_1T_{\Pi}^{-i+1}$ for $1\leq i\leq 4$. We then have 
\[
M=\underbrace{kv_1}_{Me_{\chi}} \oplus \underbrace{kv_2}_{Me_{\chi^s[1,0]}} \oplus \underbrace{k v_3}_{Me_{\chi[1,1]}} \oplus \underbrace{k v_4}_{Me_{\chi^s[0,1]}}
\]
We further set
\[
s_1=r_{\chi^s[1,0]}, \hspace{0.2cm}s_2=r_{\chi[1,1]},\hspace{0.2cm} s_3=r_{\chi^s[0,1]},\hspace{0.2cm} s_4=r_{\chi}
\]
and
\begin{align*}
	c_1=(-1)^{b_{\chi^s[1,0]}} r_{\chi^s[1,0]}!,\hspace{0.3cm} & 	c_2=(-1)^{b_{\chi[1,1]}}r_{\chi[1,1]}!,\\
	c_3=(-1)^{b_{\chi^s[0,1]}}r_{\chi^s[0,1]}!,\hspace{0.3cm} & c_4=(-1)^{\frac{p-1}{2}}z^{-1} (-1)^{b_{\chi}} r_{\chi}!.
\end{align*}

\noindent By Corollary \ref{recursion_formula}, we then have
\[
X^{s_i}F(v_i)=c_i v_{i+1} \text{ for all $1\leq i\leq 4$.}
\]
Here, note that, by Lemma \ref{identities_Hecke_operators_II} (iv), $T_{\Pi}^{-4}=(-1)^{\frac{p-1}{2}}T_{\Pi^{-4}}$,  which acts on $M$ via $(-1)^{\frac{p-1}{2}}z^{-1}$. By the Basic Lemma \ref{basic_lemma} (vii),
\[
\pi_M=k\llbracket X\rrbracket[F]M,
\]
is an irreducible $P^+$-subrepresentation of $\cT(M)$: in the square-regular case, the $\Gamma$-eigencharacters of the vectors $v_1,\ldots,v_4$ are pairwise distinct, while in case $\mathbf{1}\in O$, the $\Gamma$-eigencharacters of $Me_{\mathbf{1}}$ and $Me_{\mathbf{1}[0,1]}$ (resp.\ $Me_{\mathbf{1}[1,1]}$ and $Me_{\mathbf{1}[1,0]}$) coincide, but the corresponding $s_i$'s differ by $\frac{p-1}{2}$ (see Remark \ref{weight_remark}), so we can indeed apply the Basic Lemma.

\begin{prop}\label{supersingular_irred}
	The equality $\cT(M)=\pi_M[F^{-1}]$ holds true. In particular, the restriction of $\cT(M)$ to $P$ is irreducible.
\end{prop}

\begin{proof}
	We first claim that $\pi_M$ contains the weights $\sigma(\psi)$ for all $\psi\in \chi^{\tilde{\Pi}^{\bZ}}$. The Iwahori decomposition and the inclusion $(I\cap sUs)F\subset F(I\cap sUs)$ imply that $\pi_M$ is $I$-stable. Now $\sigma(\psi)$, being a weight, is generated as a $K$-representation by its $I_1$-invariants $Me_{\psi}$, which are contained in $\pi_M$ by definition. Using the decomposition $K=I\sqcup IsI$, it suffices to prove that $\tilde{s}Me_{\psi}\in \pi_M$. Writing $\tilde{s}=F\tilde{\Pi}^{-1}$, we have
	\[
	\tilde{s}Me_{\psi}=F(Me_{\psi}T_{\Pi})=F(Me_{\psi^s[0,1]})\subset \pi_M.
	\]
	As already used a few times, the element $\tilde{\Pi}^2$ acting on $\cT(M)$ restricts to an isomorphism $\sigma(\psi[1,1])\cong \sigma(\psi)\otimes {\det}^{\frac{p-1}{2}}$, which implies that $M$ is stable under the action of $\wt{Z}$. Putting this together with the previous claim, we obtain
	\begin{equation}\label{equation_for_ss}
		\cT(M)\supset \pi_M[F^{-1}] =\left\langle \wt{B}.\sum_{\psi} \sigma(\psi)\right\rangle
		=\left\langle \wt{B}K.\sum_{\psi} \sigma(\psi)\right\rangle\supset \left\langle \wt{G}.M\right\rangle =\cT(M),
	\end{equation}
	i.e.\ $\pi_M[F^{-1}]=\cT(M)$. As explained prior to this proposition, the condition of Corollary \ref{basic_lemma_application} is satisfied, whence $\pi_M[F^{-1}]=\cT(M)$ is irreducible as a $P$-representation.
\end{proof}

\subsubsection{...and back} It remains to prove that the unit $M\to \cI(\cT(M))$ is an isomorphism. We will use an analogue of the short exact sequence in \cite[Theorem 6.3]{Paskunas_ext_ss}, the proof of which we will closely follow.

For $1\leq i\leq 4$, let
\[
\pi_i=k\llbracket X\rrbracket[F^4]v_i
\]
as in the setup of Section \ref{basic_lemma}. The Iwahori decomposition and the inclusion $(I\cap sUs)F\subset F(I\cap sUs)$ imply that $\pi_i$ is $\wt{I}$-stable, and $Z_1:=I_1\cap Z$ acts trivially on it. By part (iii) and (vi) of the Basic Lemma \ref{basic_lemma}, $0\neq \pi_i^{I_1}\subset \pi_i^{I\cap U}=kv_i$ (hence $\pi_i^{I_1}=kv_i$) and the inclusion $kv_i\subset \pi_i$ is an injective envelope in the category of smooth $I\cap U$-representations. Thus, $(\pi_i/k v_i)^{I\cap U}=H^1(I\cap U,k v_i)\cong \Hom_{\operatorname{cts}}(\bZ_p,k)$ is $1$-dimensional; in particular, $(\pi_i/k v_i)^{I_1}=(\pi_i/k v_i)^{I\cap U}$.

	Consider the short exact sequence
\[
0\to kv_i\to \pi_i\to \pi_i/kv_i\to 0
\]
of $\wt{I}/Z_1$-representations. The injection becomes an equality on $I_1$-invariants, so we obtain an inclusion
\[
(\pi_i/kv_i)^{I_1}\hookrightarrow H^1(I_1/Z_1,kv_i)\subset H^1(I_1/Z_1,M)\cong \Hom(I_1/Z_1,M),
\]
whose $1$-dimensional image we denote by $\Delta_i$. Using the decomposition $\pi_M=\oplus_{i=1}^4 \pi_i$, we take the direct sum of the above inclusions and call the resulting image $\Delta$:
\[
\Delta=\bigoplus_{i=1}^4 \Delta_i =\operatorname{image}\left((\pi_M/M)^{I_1}\hookrightarrow H^1(I_1/Z_1,M)\cong \Hom(I_1/Z_1,M)\right).
\]
The element $\tilde{\Pi}$ normalizes $I_1$ and its subgroup $Z_1$. Its action on $M$ thus induces an action on $H^1(I_1/Z_1,M)\cong \Hom(I_1/Z_1,M)$:
\[
[\tilde{\Pi}.f](x)=\tilde{\Pi}f(\tilde{\Pi}^{-1}x\tilde{\Pi}) \text{ for all $x\in I_1$ and $f\in \Hom(I_1/Z_1,M)$.} 
\]
In particular, $\tilde{\Pi}\Delta_i\subset H^1(I_1/Z_1,kv_{i+1})$. Note that $\tilde{\Pi}^2\Delta_i=\Delta_{i+2}$ since $\tilde{\Pi}^2\pi_i=\pi_{i+2}$.

\begin{lem}\label{Delta_cap=0}
	The intersection $\Delta\cap \tilde{\Pi}\Delta = 0$ is trivial.
\end{lem} 

\begin{proof}
	We need to show that for each $1\leq i\leq 4$, $\tilde{\Pi}\Delta_i\cap \Delta_{i+1}=0$. Acting by $\tilde{\Pi}$, we may replace $i$ by $i+1$ and thus assume that the $H$-eigencharacter $\chi_i$ of $v_i$ is regular, i.e.\ $\chi_i\neq \chi_i^{s}$. By \cite[Proposition 5.2]{Paskunas_ext_ss}, a $k$-basis of $H^1(I_1/Z_1,\ol{\bF}_pv_j)\cong \Hom(I_1/Z_1,kv_j)$, for $1\leq j\leq 4$,  is given by the homomorphisms $\kappa^{u}_j,\kappa^{\ell}_j$ defined by
	\[
	\kappa^{u}_j\left(\begin{mat}
		a & b\\ c & d
	\end{mat}\right)=\omega(b)v_j \text{ and } \kappa^{\ell}_j\left(\begin{mat}
		a & b\\ c & d
	\end{mat}\right)=\omega(p^{-1}c)v_j \text{ for all $\begin{mat}
			a & b\\ c & d
		\end{mat}\in I_1$.}
	\]
	Since the eigencharacter of $v_i$ is regular, we have $\Delta_i=k\kappa^{u}_i$, by Proposition 6.2 (i) in \textit{loc.\ cit.} The cited proposition also implies that $\Delta_{i+1}=k(\kappa^{u}_{i+1}+b\kappa^{\ell}_{i+1})$ for some $b\in k$. Noting that $\tilde{\Pi}\kappa^{u}_i$ is a non-zero scalar multiple of $\kappa^{\ell}_{i+1}$, we obtain
	\[
	\tilde{\Pi}\Delta_i \cap \Delta_{i+1}=k \kappa^{\ell}_{i+1} \cap k(\kappa^{u}_{i+1}+b\kappa^{\ell}_{i+1})=0,
	\]
	which finishes the proof.
\end{proof}

\begin{prop}\label{theorem_ses}
	The $\langle \wt{I},\tilde{\Pi}\rangle$-equivariant sequence
	\[
	0\to M\xrightarrow{\operatorname{diag}} \pi_M\oplus \tilde{\Pi}\pi_M\xrightarrow{(v,w)\mapsto v-w} \cT(M)\to 0
	\]
	is exact.
\end{prop}

\begin{proof}
	With the previous preparations, the same proof as in \cite[Theorem 6.3]{Paskunas_ext_ss} works.
\end{proof}

\begin{cor}\label{ss:IT(M)=M}
	The unit $M\to \cI(\cT(M))$ is an isomorphism.
\end{cor}

\begin{proof}
	Since $M$ is irreducible, the unit is an inclusion, so it is enough to prove that $\cI(\cT(M))$ has the same dimension as $M$. The equality $M=\pi_M^{I_1}$ reduces us to proving that the sequence in Proposition \ref{theorem_ses} remains exact after taking $I_1$-invariants, i.e.\ we need that the map $H^1(I_1/Z_1,M)\to H^1(I_1/Z_1,\pi_M\oplus \tilde{\Pi}\pi_M)$ is injective. This amounts to showing that $W\cap \tilde{\Pi}W=0$, where $W$ is the kernel of $H^1(I_1/Z_1,M)\to H^1(I_1,/Z_1,\pi_M)$. Since $\pi_M^{I_1}=M$, the kernel $W$ is equal to the image of the injection $(\pi_M/M)^{I_1}\hookrightarrow H^1(I_1/Z_1,M)$, i.e.\ $W=\Delta$. The corollary now follows from Lemma \ref{Delta_cap=0}.
\end{proof}

\begin{cor}\label{ss_socle}
	The $K$-socle of $\cT(M)$ is equal to $\bigoplus_{\psi\in \chi^{\tilde{\Pi}^{\bZ}}} \sigma(\psi)|_K$ with $\sigma(\psi)$ as in Corollary \ref{generates_weight}.
\end{cor}

\begin{proof}
	By Corollary \ref{generates_weight}, $\cT(M)$ contains the pairwise distinct weights $\sigma(\psi)$ for $\psi\in \chi^{\tilde{\Pi}^{\bZ}}$, so we obtain an inclusion $\bigoplus_{\psi\in \chi^{\tilde{\Pi}^{\bZ}}} \sigma(\psi)|_K \subset \soc_K(\cT(M))$.   By Corollary \ref{ss:IT(M)=M} this is an equality on $I_1$-invariants, which implies the assertion.
\end{proof}

\begin{thm}\label{thm_ss}
	The adjoint pair $(\cT,\cI)$ induces a bijection
	\[
	\left\{\begin{array}{c}
		\text{Supersingular right}\\
		\text{$\cH$-modules}
	\end{array}\right\} \cong \left\{\begin{array}{c}
		\text{Admissible genuine supersingular}\\ 
		\text{representations of $\wt{G}$},
	\end{array}\right\}
	\]
	where both sides are considered up to isomorphism.
\end{thm}

\begin{proof}
	Let $M$ be a supersingular $\cH$-module. By Proposition \ref{supersingular_irred}, the $\wt{G}$-representation $\cT(M)$ is irreducible. If it was not supersingular, then it would be isomorphic to a principal series, see Remark \ref{ss=notps}. By Theorem \ref{thm_ps} and Corollary \ref{ss:IT(M)=M}, this is impossible and we deduce that $\cT(M)$ is supersingular and that the assignment $M\mapsto \cT(M)$ is injective with inverse $\cI$. Let now $\pi$ be a genuine supersingular $\wt{G}$-representation. Then $\cI(\pi)$ is a non-zero finite dimensional $\cH$-module and hence contains some simple module $M$. By the classification results for simple Hecke modules, $M$ is either principal or supersingular. In any case, the attached $\wt{G}$-representation $\cT(M)$ is irreducible (as we just explained in the supersingular case and by Theorem \ref{thm_ps} in the principal case). By adjunction we obtain a non-trivial $\wt{G}$-equivariant map $\cT(M)\to \pi$, which has to be an isomorphism as both sides are irreducible. Once again using Theorem \ref{thm_ps}, we see that $M$ cannot be a principal module and therefore needs to be supersingular. This proves surjectivitiy.
\end{proof}

	\begin{cor}\label{ss_restricted_P}
	Let $\pi_1,\pi_2$ be genuine supersingular $\wt{G}$-representations. Then the restriction functor from $\wt{G}$ to $P$ induces an equality
	\[
	\Hom_{\wt{G}}(\pi_1,\pi_2)=\Hom_P(\pi_1|_P,\pi_2|_P).
	\]
	In other words, $\pi_1\cong \pi_2$ as $\wt{G}$-representations if and only if $\pi_1|_P\cong \pi_2|_P$ as $P$-representations.
\end{cor}

\begin{proof}
	By Theorem \ref{thm_ss}, we may write $\pi_i=\cT(M_i)$ for a supersingular $\cH(O_i)$-module $M_i$, for some orbit $O_i$, for $i=1,2$. By Proposition \ref{supersingular_irred}, $\pi_i$ is then irreducible as a $P$-representation. Assume now that $f\colon \pi_1|_P\cong \pi_2|_P$ is an isomorphism of $P$-representations. Since $\pi_2=\pi_{M_2}[F^{-1}]$, we can choose $N\gg 0$ such that $F^{4N} f(M_1)\subset \pi_{M_2}$. Using Lemma \ref{prep_lem} (ii), we then obtain $f(M_1)\subset \pi_{M_2}$. Comparing $X$-torsion shows that $f(M_1)\subset M_2$, which is an equality by symmetry, i.e.\ $f$ restricts to a $\Gamma$-equivariant isomorphism $f\colon M_1\cong M_2$.
	
	We now check that $O_1=O_2$. The space $M_i$ is the direct sum of $1$-dimensional $\Gamma$-stable subspaces. If $O_i$ is square-regular, then the $\Gamma$-characters in $M_i$ appear with multiplicity one. If $O_i$ is non-square-regular, then the $\Gamma$-characters in $M_i$ appear with multiplicity two. This already implies that $O_1$ is square-regular if and only if $O_2$ is square-regular. For $i=1,2$, let us now fix a character $\chi_i\in O_i$, which we choose to satisfy $\chi_i=\chi_i^s$ in the non-square-regular case. Let $O_i^2$ denote the set of characters on $H$ obtained by squaring those in $O_i$. Then $O_i^2$ determines $O_i$. Concretely, $O_i^2=\{\chi_i^2,(\chi_i^2)^s\}$, which has cardinality two if and only if $O_i$ is square-regular. We can recover $O_i^2$ from $\{\chi_i^2|_{\Gamma},(\chi_i^2)^s|_{\Gamma}\}$, and so $O_1=O_2$. After possibly replacing $\chi_2$, we may assume that $\chi_1=\chi_2$, which we denote by $\chi$. In the non-square-regular case we may, after possibly twisting (Lemma \ref{twist_lemma}), assume that $\chi=\mathbf{1}$.
	
	For $i=1,2$ and $\psi\in \chi^{\tilde{\Pi}^{\bZ}}$, let $\sigma_i(\psi)\subset \pi_i$ denote the weight generated by $M_ie_{\psi}$, see Lemma \ref{generates_weight}, and put $r_i(\psi)=\dim_{k}\sigma_i(\psi)-1$. 
	We claim that $r_1(\psi)=r_2(\psi)$ for all $\psi\in \chi^{\tilde{\Pi}^{\bZ}}$. By Remark \ref{weight_remark}, it suffices to prove this for a single $\psi$. Choose $\psi$ such that $r_1(\psi)\geq r_1(\psi')$ for all $\psi'\in \chi^{\tilde{\Pi}^{\bZ}}$. By symmetry, we may assume that $r_1(\psi)\geq r_2(\psi)$. Assume by contradiction that $r_1(\psi)>r_2(\psi)$. Looking at the $\Gamma$-eigencharacter, we must have
	\[
	f(v_{\psi^s[0,1]})=\alpha w_{\psi^s[0,1]} + \beta w_{\psi^s} \text{ for some $\alpha,\beta\in k$ with $\alpha \beta \neq 0$,}
	\]
	where we put $w_{\psi^s}:=0$ in the square-regular case.
	By the recursion formula in Corollary \ref{recursion_formula}, we have
	\begin{equation}\label{recursion_used}
		X^{r_1(\psi)}F(v_{\psi^s[0,1]})=a v_{\psi} \text{ for some $a\in k^{\times}$.}
	\end{equation}
	Applying the $P$-equivariant isomorphism $f$ to this equation, we obtain
	\begin{equation}\label{equal2}
		af(v_{\psi})=\alpha X^{r_1(\psi)}F(w_{\psi^s[0,1]}) + \beta X^{r_1(\psi)}F(w_{\psi^s}),
	\end{equation}
	where the first summand vanishes since $r_1(\psi)>r_2(\psi)$ (using the analogue of equation (\ref{recursion_used}) for $w_{\psi^s[0,1]}$ and the fact that $Xw_{\psi}=0$). It follows that $\beta\neq 0$ and
	\begin{equation}\label{equal1}
		X^{r_1(\psi)}F(w_{\psi^s})=\beta^{-1}a f(v_{\psi})\neq 0.
	\end{equation}
	This already gives a contradiction in the square-regular case (as we put $w_{\psi^s}=0$ in this case). It remains to treat the non-square-regular case.  Since the expression in (\ref{equal1}) is non-zero, we deduce that $r_1(\psi)\leq r_2(\psi[1,0])$; and since it is killed by $X$, we then deduce that $r_1(\psi)=r_2(\psi[1,0])$ (both of these observations follow from Corollary \ref{recursion_formula}). Now either $\psi=\psi^s$ or $\psi[1,0]=\psi[1,0]^s$, i.e.\ either $r_i(\psi)\in \{0, p-1\}$ (independent of $i$) or $r_i(\psi[1,0])\in \{0, p-1\}$ (independent of $i$), but not both. Hence, the equality $r_1(\psi)=r_2(\psi[1,0])$ is impossible. This proves the claim.
	
	Using the claim we obtain that $f(v_{\psi})=\gamma w_{\psi}$ for some $\gamma\in k^{\times}$. Indeed, in the square-regular case, there is nothing to prove. In the non-square-regular case, note that $r_2(\psi[1,0])=r_1(\psi[1,0])<r_1(\psi)=r_2(\psi)$, so the second summand in equation (\ref{equal2}) vanishes, while the first one is a non-zero scalar multiple of $w_{\psi}$.
	
	For $i=1,2$, let us now write
	\[
	M_i=\begin{cases}
		\operatorname{SS}(\chi,z_i) \text{ if $O$ is square-regular}\\
		\operatorname{SS}(\lambda_i,z_i) \text{ if $\mathbf{1}\in O$,}
	\end{cases}
	\]
	where $z_i\in k^{\times}$ and $\lambda_i\in \{0,-1\}$ as in Definition \ref{def_simple_reg}, resp.\ \ref{def_simple_nreg}. It follows from Corollary \ref{generates_weight} (ii) that $\lambda_1=\lambda_2$ since $r_1(\chi)=r_2(\chi)$. In order to conclude that $M_1=M_2$, it remains to show that $z_1=z_2$. By Lemma \ref{prep_lem} (ii), we have
	\[
	X^{e(\psi)} F^4(v_{\psi})= (-1)^{\frac{p-1}{2}} z_i^{-1} \left(\prod_{\psi'\in \chi^{\tilde{\Pi}^{\bZ}}} r_1(\psi')!\right) v_{\psi},
	\]
	where the exponent $e(\psi)$ only depends on $r_1(\psi)$  (which determines all the other $r_1(\psi')$). A similar equality also holds true for $w_{\psi}$. Applying the $P$-equivariant map $f$ and using the equalities $r_1(\psi')=r_2(\psi')$ for all $\psi'\in \chi^{\tilde{\Pi}^{\bZ}}$ allows us to conclude $z_1=z_2$.
\end{proof}

\begin{cor}
	Let $\pi$ be a genuine supersingular $\wt{G}$-representation. Then each element in $\pi/X\pi$ is killed by a power of $F$.
\end{cor}

\begin{proof}
	Write $\pi=\cT(M)$ for a supersingular $\cH$-module $M$. By Proposition \ref{supersingular_irred}, we have $\pi=\pi_M[F^{-1}]$ and the relations in Corollary \ref{recursion_formula} imply that $\pi_M=X\pi_M$, which finishes the proof.
\end{proof}

\section{Spherical Hecke algebras}

In this section, we study the genuine mod-$p$ representation theory of $\wt{G}$ via spherical Hecke algebras by following the approach of Barthel-Livn\'{e} as outlined in the introduction. The spherical Hecke algebra, defined to be the endomorphism ring of the compact induction of some weight, turns out to be a polynomial ring in a single operator $\wt{T}$ (not to be confused with the metaplectic torus). Any smooth irreducible admissible genuine $\wt{G}$-representation is a quotient of the cokernel of a linear polynomial in $\wt{T}$, for a suitable weight. By following \cite{Breuil_Paskunas}, we show that these cokernels are of finite length and determine their Jordan-Hölder decomposition. As an application, we prove that a smooth $k$-representation of $\wt{G}$ is of finite length if and only if it is finitely generated and admissible, extending the known result for $G$.

We set things up a bit more generally to also include characteristic $0$ coefficients as in \cite{Breuil_II} allowing us to show in the next section that in case $k=\Fpbar$ the aforementioned cokernels arise as the mod-$p$ reduction of an invariant bounded lattice of a locally algebraic $\Qpbar$-representations of $\wt{G}$

\subsection{Computational preparations}

Let $R$ be a $\bZ_p$-algebra. We fix an integer $r\in \bZ_{\geq 0}$ and define the $R$-linear representation $\Sym^{r}(R^2)$ of $K$ just as in the case of $R=k$, Definition \ref{sym_def}, i.e.\ the underlying $R$-module is the space of homogeneous polynomials in $R[x,y]$ of degree $r$ with action defined by 
\[
\begin{mat}
	a & b\\ c & d
\end{mat}P(x,y)=P(ax+cy,bx+dy), \text{ for all $\begin{mat}
		a & b\\c & d
	\end{mat}\in K$ and $P(x,y)\in \Sym^{r}(R^2)$.}
\]
Extend the $K$-action to an action of  $KZ^{\sq}$ by letting the scalar matrix $p^2$ act trivially, and denote the resulting representation by $\underline{\Sym}^r(R^2)$. By Lemma \ref{centers} (i), we have $Z(\wt{G})=Z^{\sq}\times \mu_2$ and we write
\[
\tilde{\sigma}_r(R):=\underline{\Sym}^r(R^2)\boxtimes \iota
\]
for the corresponding genuine $\wt{K}Z(\wt{G})$-representation. We identify the \textit{spherical Hecke algebra} $\End_{\wt{G}}(\cInd^{\wt{G}}_{\wt{K}Z(\wt{G})}(\tilde{\sigma}_r(R)))$ with the algebra
\small
\begin{equation}\label{hecke-algebra_descr}
	\biggl\{\varphi\colon \wt{G}\to \End_R(\tilde{\sigma}_r(R)) : \begin{array}{l}
		\diamond \hspace{0.1cm} \varphi(k_2 gk_2)=k_2\circ \varphi(g) \circ k_1, \forall g\in \wt{G}, k_1,k_2\in \wt{K}Z(\wt{G})\\
		\diamond \hspace{0.1cm} \supp(\varphi) \text{ is compact} 
	\end{array} \biggr\},
\end{equation}
\normalsize
whose multiplication is given by convolution.

The following is an analogue of \cite[Lemme 2.1.4.1]{Breuil_II} and its proof is similar.

\begin{def-lem}\label{def_T} There exists a unique function $\wt{\mathbf{T}}_R$ in the algebra (\ref{hecke-algebra_descr}) with 
	\begin{enumerate}
		\item[$\bullet$] $\supp(\wt{\mathbf{T}}_R)=\wt{K}Z(\wt{G})\left(\begin{mat}
			1 & 0\\0 & p^{-2}
		\end{mat},1\right)\wt{K}Z(\wt{G})$;
		\item[$\bullet$] $\wt{\mathbf{T}}_R\left(\left(\begin{mat}
			1 & 0\\0 & p^{-2}
		\end{mat},1\right)\right)(x^{r-i}y^{i})=p^{2(r-i)}x^{r-i}y^{i}$ for all $0\leq i \leq r$.
	\end{enumerate}
	The corresponding endomorphism in $\End_{\wt{G}}(\cInd^{\wt{G}}_{\wt{K}Z(\wt{G})}(\tilde{\sigma}_r(R)))$ is denoted by $\wt{T}_R$. If $R=k$, we simply write $\wt{T}$ instead of $\wt{T}_{k}$.
\end{def-lem}

\begin{prop}\label{Hecke_poly}
	If $R$ is an $\bF_p$-algebra and $0\leq r\leq p-1$, then the $R$-algebra $\End_{\wt{G}}(\cInd^{\wt{G}}_{\wt{K}Z(\wt{G})}(\tilde{\sigma}_r(R)))$ is a commutative polynomial ring in the operator $\wt{T}_R$.
\end{prop}

\begin{proof}
	Since $\cInd^{\wt{G}}_{\wt{K}Z(\wt{G})}(\tilde{\sigma}_r(\bF_p))$ is finitely generated over $\bF_p[\wt{G}]$, one may reduce to the case $R=\bF_p$, see also \cite[Lemma 5.1]{paskunas}.
	Using the description of the Hecke algebra given by (\ref{hecke-algebra_descr}) and the Cartan decomposition, the assertion can be proved as in the case of $G=\GL_2(\bQ_p)$, see \cite[Theorem 21]{herzig}. We only explain why a function supported on $\left(\begin{mat}
		p^n & 0\\0 & p^{m}
	\end{mat},1\right)$, for $n\geq m\geq 0$, can only exist if $n$ and $m$ are even. Let $F$ be such a function and put $f=F\left(\left(\begin{mat}
		p^n & 0\\0 & p^{m}
	\end{mat},1\right)\right)\in \End_{\bF_p}(\tilde{\sigma}_r(\bF_p))$. First assume that $n>m$. For $x\in \bZ_p$, the equation
	\[
	\left(\begin{mat}
		p^n & 0\\ 0 & p^m
	\end{mat},1\right)\left(\begin{mat}
		1 & x\\0 & 1
	\end{mat},1\right) = \left(\begin{mat}
		1 & p^{n-m}x\\ 0 & 1
	\end{mat},1\right)\left(\begin{mat}
		p^n & 0\\0 & p^m
	\end{mat},1\right)
	\]
	implies that $f$ factors through the $1$-dimensional $U\cap K$-coinvariants $\tilde{\sigma}_r(k)_{U\cap K}$, while the equation
	\[
	\left(\begin{mat}
		1 & 0\\x & 1
	\end{mat},1\right) \left(\begin{mat}
		p^n & 0\\0 & p^m
	\end{mat},1\right) = \left(\begin{mat}
		p^n & 0\\0 & p^m
	\end{mat},1\right)\left(\begin{mat}
		1 & 0\\p^{n-m}x & 1
	\end{mat},1\right)
	\]
	implies that $f$ takes values in the $1$-dimensional $sUs\cap K$-invariants $\tilde{\sigma}_r(k)^{sUs\cap K}$. In conclusion, $f$ must be a scalar multiple of the canonical map 
	\[
	\tilde{\sigma}_r(\bF_p) \twoheadrightarrow \tilde{\sigma}_r(\bF_p)_{U\cap K} \xleftarrow{\cong} \tilde{\sigma}_r(\bF_p)^{sUs\cap K}\subset \tilde{\sigma}_r(\bF_p),
	\]
	where the isomorphism is the projection. In particular, $f$ must be $T\cap K$-equivariant. Now, by Example \ref{law_torus}, for all $\lambda,\mu\in \bF_p^{\times}$,
	\[
	\left(\begin{mat}
		p^n & 0\\ 0 & p^m
	\end{mat},1\right)\left(\begin{mat}
		[\lambda] & 0\\ 0 & [\mu]
	\end{mat},1\right) = \left(\begin{mat}
		[\lambda] & 0\\ 0 & [\mu]
	\end{mat},\lambda^{m\frac{p-1}{2}} \mu^{n\frac{p-1}{2}}\right) \left(\begin{mat}
		p^n & 0\\ 0 & p^m
	\end{mat},1\right),
	\]
	and so the $T\cap K$-equivariance of $f$ can only hold true if $f=0$ or $m$ and $n$ are even.
	
	If $n=m$, then Lemma \ref{Ztilde_action} implies that $f$ defines a $K$-equivariant map between $\tilde{\sigma}_r(k)$ and its twist by $\omega^{n\frac{p-1}{2}}\circ {\det}$. Such a map can only be non-zero if the twist is trivial, i.e.\ if $n$ is even.
\end{proof}

\begin{remark}
	If $R$ is a $\bQ_p$-algebra and $r=0$, then the proposition is still valid as follows for example from the Satake isomorphism \cite[Theorem 10.1]{McNamara}.
\end{remark}

\begin{definition}
	Given a vector $v\in \tilde{\sigma}_r(R)$ and an element $g\in \wt{G}$, denote by $[g,v]\in \cInd^{\wt{G}}_{\wt{K}Z(\wt{G})}(\tilde{\sigma}_r(R))$ the unique function with support $\wt{K}Z(\wt{G})g^{-1}$ and value $v$ at $g^{-1}$.
\end{definition}

Next, we define an explicit set of representatives of $\bZ_p/p^m\bZ_p$, for all $m\geq 0$, by letting $J_m=\left\{\sum_{j=0}^{m-1}p^{j}[\lambda_j] : \lambda_j\in \bF_p\right\}$ if $m\geq 1$ and setting $J_0=\{0\}$. Moreover, for $m\geq 0$ and $\lambda\in J_m$,  we put 
$g^0_{m,\lambda}=\begin{mat}
	p^m & \lambda\\0 & 1
\end{mat}$ and $g^1_{m,\lambda} = \begin{mat}
	1 & 0\\ p\lambda & p^{m+1}
\end{mat}.$ 
It follows from \cite[Equation (1)]{Breuil_II} that
\begin{equation}\label{decomp_G}
	\wt{G}=\bigsqcup_{m\geq 0, \lambda\in J_m} \bigsqcup_{q,j=0,1}\wt{K}Z(\wt{G}) (\tilde{g}^{q}_{m,\lambda})^{-1}\tilde{\Pi}^{-2j}.
\end{equation}

\begin{definition}
	For $m\geq 0$ and $q,j=0,1$, define the subspaces $S_m(R)\subset B_m(R)$ of $\cInd^{\wt{G}}_{\wt{K}Z(\wt{G})}(\tilde{\sigma}_r(R))$ by
	\[
	S_m(R)=\left\{f : \supp(f) \subset \bigsqcup_{\lambda\in J_m}\bigsqcup_{q,j=0,1} \wt{K}Z(\wt{G}) (\tilde{g}^{q}_{m,\lambda})^{-1}\tilde{\Pi}^{-2j}\right\},
	\]
	and $B_m(R)=\sum_{n\leq m} S_n(R)$, respectively.
\end{definition}

Note that $B_m(R)\subset B_{m+1}(R)$ for all $m\geq 0$. The decomposition (\ref{decomp_G}) implies that $\cInd^{\wt{G}}_{\wt{K}Z(\wt{G})}(\tilde{\sigma}_r(R))=\sum_{m\geq 0} B_m(R)$. Given a function $f$ in this space, one may think of the minimal $m\geq 0$ satisfying $f\in B_m(R)$ as the radius of the support of $f$.

The decomposition
\[
\wt{K}Z(\wt{G})\left(\begin{mat}
	1 & 0\\0 & p^{-2}
\end{mat},1\right)\wt{K}Z(\wt{G})=\bigsqcup_{\lambda\in J_2} \wt{K}Z(\wt{G})(\tilde{g}^0_{2,\lambda})^{-1} \sqcup \bigsqcup_{\lambda\in J_1} \wt{K}Z(\wt{G})(\tilde{g}^{1}_{1,\lambda})^{-1},
\]
which is quickly verified using $K=I\sqcup IsI$, implies that the endomorphism $\wt{T}_R\in \End_{\wt{G}}\left(\cInd^{\wt{G}}_{\wt{K}Z(\wt{G})}(\tilde{\sigma}_r(R))\right)$ is explicitly given by
\begin{equation}\label{T([1,v])}
	\wt{T}_R([1,v])=\sum_{\lambda\in J_2} [\tilde{g}^0_{2,\lambda},\wt{\mathbf{T}}_R((\tilde{g}^0_{2,\lambda})^{-1})(v)] + \sum_{\mu \in J_1} [\tilde{g}^1_{1,\mu}, \wt{\mathbf{T}}_R((\tilde{g}^1_{1,\mu})^{-1})(v)],
\end{equation}
for all $v\in \tilde{\sigma}_r(R)$. Motivated by this description, we make the following definition in analogy to \cite[Lemme 2.2.1]{Breuil_II}.

\begin{definition}\label{def_T+} Let $m\geq 0$, $\mu\in J_m$ and $v\in \tilde{\sigma}_r(R)$. 	
	\begin{enumerate}
		\item[{\rm (i)}] Define $\wt{T}_R^+([\tilde{g}^0_{m,\mu},v])=\sum_{\lambda\in J_2} [\tilde{g}^0_{m+2,\mu+p^m \lambda}, \wt{\mathbf{T}}_R((\tilde{g}^0_{2,\lambda})^{-1})(v)]$.
		\item[{\rm (ii)}] Noting that $\tilde{g}^1_{m,\mu}=\tilde{\Pi} \tilde{g}^0_{m,\mu} \tilde{s}$, set $\wt{T}_R^{+}([\tilde{g}^1_{m,\mu},v])=\tilde{\Pi} \wt{T}_R^+([\tilde{g}^0_{m,\mu} ,\tilde{s}v])$.
		\item[{\rm (iii)}] For $q\in \{0,1\}$, put 	$\wt{T}_R^+([\tilde{\Pi}^2 \tilde{g}^{q}_{m,\mu},v])=\tilde{\Pi}^2 \wt{T}_R^+([\tilde{g}^q_{m,\mu},v]).$
	\end{enumerate}
	Extending $R$-linearly, cf.\ (\ref{decomp_G}), yields the endomorphism $\wt{T}^+_R$ on $\cInd^{\wt{G}}_{\wt{K}Z(\wt{G})}(\tilde{\sigma}_r(R))$.
\end{definition}

\noindent For $m\geq 0$, we thus have $\wt{T}^+_R(S_m(R))\subset S_{m+2}(R)$ and one checks that the difference $\wt{T}_R-\wt{T}^+_R$ preserves $B_m(R)$. We obtain:

\begin{observation}\label{observation}
	Let $f\in \cInd^{\wt{G}}_{\wt{K}Z(\wt{G})}(\tilde{\sigma}_r(R))$. If $\wt{T}_R(f)\in B_m(R)$ for some $m\geq 0$, then $f\in B_{m-2}(R)$. In particular, the image of $\wt{T}_R$ intersects $B_1(R)$ trivially, and for any $\lambda\in R$, the endomorphism $\wt{T}_R-\lambda$ is not surjective.
\end{observation}

\subsection{Jordan-Hölder sequence}
We now take $R=k$ and restrict $0\leq r\leq p-1$ so that the $K$-representation $\tilde{\sigma}_r(k)$ is irreducible. Following the exposition in \cite[§6]{Breuil_Paskunas}, we show that $\cInd^{\wt{G}}_{\wt{K}Z(\wt{G})}(\tilde{\sigma}_r(k))/(\wt{T}-\lambda)$, for $\lambda\in k$, has finite length as a $\wt{G}$-representation and compute its Jordan-Hölder decomposition. 

To ease notation let us put $\tilde{\sigma}_r:=\tilde{\sigma}_r(k)$ and let $\chi=\omega^r \otimes \mathbf{1}\in \Hhat$ so that $\chi\boxtimes \iota = \tilde{\sigma}_r^{I_1}$. We extend $\chi$ to a character of $HZ^{\sq}$ by letting $p^2$ act trivially.

\begin{prop}\label{cInd_Hecke}
	\begin{enumerate}
		\item[{\rm (i)}] If $\chi\neq \chi^s$, then
		\begin{align*}
			\cInd^{\wt{G}}_{\wt{K}Z(\wt{G})}(\tilde{\sigma}_r)&=\ker\left(T_s\colon \cInd^{\wt{G}}_{\wt{I}Z(\wt{G})}(\chi^s\boxtimes \iota)\to  \cInd^{\wt{G}}_{\wt{I}Z(\wt{G})}(\chi\boxtimes \iota)\right)\\
			&=T_s\cInd^{\wt{G}}_{\wt{I}Z(\wt{G})}(\chi\boxtimes \iota).
		\end{align*}
		\item[{\rm (ii)}] If $r=p-1$, then
		\begin{align*}
			\cInd^{\wt{G}}_{\wt{K}Z(\wt{G})}(\tilde{\sigma}_{p-1})&=\ker\left(T_s+1\colon \cInd^{\wt{G}}_{\wt{I}Z(\wt{G})}(\mathbf{1}\boxtimes \iota)\to  \cInd^{\wt{G}}_{\wt{I}Z(\wt{G})}(\mathbf{1}\boxtimes \iota)\right)\\
			&=T_s\cInd^{\wt{G}}_{\wt{I}Z(\wt{G})}(\mathbf{1}\boxtimes \iota).
		\end{align*}
		\item[{\rm (iii)}] If $r=0$, then
		\begin{align*}
			\cInd^{\wt{G}}_{\wt{K}Z(\wt{G})}(\tilde{\sigma}_0)&=\ker\left(T_s\colon \cInd^{\wt{G}}_{\wt{I}Z(\wt{G})}(\mathbf{1}\boxtimes \iota)\to  \cInd^{\wt{G}}_{\wt{I}Z(\wt{G})}(\mathbf{1}\boxtimes \iota)\right)\\
			&=(T_s+1)\cInd^{\wt{G}}_{\wt{I}Z(\wt{G})}(\mathbf{1}\boxtimes \iota).
		\end{align*}
	\end{enumerate}
\end{prop}

\begin{proof}
	Since the functor $\cInd^{\wt{G}}_{\wt{K}Z(\wt{G})})(-)$ is exact, it suffices to prove the equalities for $\wt{G}$ replaced by $\wt{K}Z(\wt{G})\cong KZ^{\sq}\times \mu_2$. In this case, the assertions follow from Proposition \ref{JH_weight} and Lemma \ref{identities_Hecke_operators_I} (v). 
\end{proof}

\begin{cor}\label{I_1-inv_cInd} We have equalities of right $\cH$-modules:
	\begin{enumerate}
		\item[{\rm (i)}] If $\chi\neq \chi^s$ or $r=p-1$, 
		\[
		\cI\left(\cInd^{\wt{G}}_{\wt{K}Z(\wt{G})}(\tilde{\sigma}_r)\right)=T_se_{\chi}\cH_{p^2=1}.
		\]
		\item[{\rm (ii)}] If $r=0$, 
		\[
		\cI\left(\cInd^{\wt{G}}_{\wt{K}Z(\wt{G})}(\tilde{\sigma}_0)\right)= (T_s+1)e_{\mathbf{1}}\cH_{p^2=1}.
		\]
	\end{enumerate}
\end{cor}

\begin{proof}
	By Frobenius reciprocity, $e_{\chi}\cH_{p^2=1}\cong \cI\left(\cInd^{\wt{G}}_{\wt{I}Z(\wt{G})}(\chi\boxtimes \iota)\right)$. Taking $I_1$-invariants in Proposition \ref{cInd_Hecke} reduces us to proving the equalities
	\begin{enumerate}
		\item[{\rm (i)}] $\ker(T_s\colon e_{\chi^s}\cH_{p^2=1}\to e_{\chi}\cH_{p^2=1})=T_s e_{\chi}\cH_{p^2=1}$ for $\chi\neq \chi^s$;
		\item[{\rm (ii)}] $\ker(T_s+1\colon e_{\mathbf{1}}\cH_{p^2=1}\to e_{\mathbf{1}}\cH_{p^2=1})=T_se_{\mathbf{1}}\cH_{p^2=1}$;
		\item[{\rm (iii)}] $\ker(T_s\colon e_{\mathbf{1}}\cH_{p^2=1}\to e_{\mathbf{1}}\cH_{p^2=1})=(T_s+1)e_{\mathbf{1}}\cH_{p^2=1}$.
	\end{enumerate}
	This is the content of Lemma \ref{T_s_for_later}.
\end{proof}

\begin{remark}
	We will fix the identifications of the corollary so that the function $[1,x^r]$ corresponds to $T_se_{\chi}$ (resp.\ $(T_s+1)e_{\mathbf{1}}$).
\end{remark}

Let now $O=O_{\chi}$ be the orbit of $\chi$. As in Section \ref{Extensions}, we put
\[
\mathscr{Z}_O=\begin{cases}
	(T_{\Pi}T_s)^2+(T_sT_{\Pi})^2 \text{ if $O$ is square-regular}\\
	\sum_{i=0}^3 T_{\Pi}^{-i}\left((T_{\Pi}T_s)^2+(T_sT_{\Pi})^2 + T_{\Pi}T_sT_{\Pi}\right)e_{\mathbf{1}}T_{\Pi}^{i} \text{ if $O$ is non-square-regular},
\end{cases}
\]
which defines a central element in $\cH(O)$.

\begin{lem}\label{T=Z}
	Under the identifications in Corollary \ref{I_1-inv_cInd}, the action of the (restriction to $I_1$-invariants of the) Hecke operator $\wt{T}$ is given by the action of $\mathscr{Z}_O$.
\end{lem}

\begin{proof}
	Since both $\wt{T}$ and $\mathscr{Z}_O$ are equivariant for the right $\cH$-action, it suffices to prove that $\wt{T}([1,x^r])=\mathscr{Z}_O([1,x^r])$. The identifications of Corollary \ref{I_1-inv_cInd} are induced by the isomorphism $e_{\chi}\cH_{p^2=1}\cong \cInd^{\wt{G}}_{\wt{I}Z(\wt{G})}(\chi\boxtimes \iota)$ obtained by Frobenius reciprocity, i.e.\ if $\varphi\in \cInd^{\wt{G}}_{\wt{I}_1Z(\wt{G})}(\mathbf{1}\boxtimes \iota)$ denotes the unique function with support $\wt{I}_1Z(\wt{G})$ and value $1$ at the identity, then the image of an element $f\in e_{\chi}\cH_{p^2=1}$ is given by $f(\varphi)$. 
	
	Together with the relations $(\tilde{g}^0_{2,\lambda})^{-1}=\tilde{s} \left(\begin{mat}
		1 & 0\\0 & p^{-2}
	\end{mat},1\right) \left(\begin{mat}
		0 & 1\\ 1 & -\lambda
	\end{mat},1\right)$ and $(\tilde{g}^1_{1,\mu})^{-1}=\left(\begin{mat}
		1 & 0\\0 & p^{-2}
	\end{mat},1\right)\left(\begin{mat}
		1 & 0\\ -p\mu & 1
	\end{mat},1\right)$, equation (\ref{T([1,v])}) implies that in characteristic $p$,
	\[
	\wt{T}([1,x^{r}])=\sum_{\lambda\in J_2} [\tilde{g}^0_{2,\lambda},x^r] + \sum_{\mu\in J_1} [\tilde{g}^1_{1,\mu},p^{2r}x^r].
	\]
	Assume that $\chi\neq \chi^s$. The element $[1,x^r]$ is identified with the element $T_se_{\chi}\in T_se_{\chi}\cH_{p^2=1}$. By Lemma \ref{identities_Hecke_operators_I} (v), $T_s^2e_{\chi}=0$ and so the operator $\mathscr{Z}_O$ acts on $T_se_{\chi}$ via $(T_sT_{\Pi})^2$. By Lemma \ref{identities_Hecke_operators_II} (vii), $(T_{\Pi}T_s)^2=T_{(\Pi s)^2}$. The decomposition $I_1 \begin{mat}
		1 & 0\\0 & p^2
	\end{mat}I_1=\sqcup_{\lambda\in J_2} I_1 (\Pi s)^2\begin{mat}
		1 & \lambda\\0 & 1
	\end{mat}$ implies that
	\begin{equation}\label{eq0}
		(T_{\Pi}T_s)^2e_{\chi}(\varphi)=\sum_{\lambda\in J_2} \begin{mat}
			1 & \lambda\\0 & 1
		\end{mat} \left(\begin{mat}
			1 & 0\\0 & p^{-2}
		\end{mat},1\right) e_{\chi}(\varphi).
	\end{equation}
	Since $T_{\Pi^4}$ acts trivially, we may replace $\left(\begin{mat}
		1 & 0\\0 & p^{-2}
	\end{mat},1\right)$ by $\left(\begin{mat}
		p^2 & 0\\0 & 1
	\end{mat},1\right)$. Applying the $\wt{G}$-equivariant homomorphism $T_s$ to the above equation and using that $T_s(T_{\Pi}T_s)^2=(T_sT_{\Pi})^2T_s$, see Lemma \ref{identities_Hecke_operators_II} (viii), gives the assertion in case $\chi\neq \chi^s$.
	
	Assume now that $r=p-1$ or $r=0$. By Lemma \ref{identities_Hecke_operators_II} (vi) and (ix), we have $T_{\Pi}T_sT_{\Pi}=T_{\Pi s\Pi}$. The decomposition $I_1 (\Pi s \Pi)I_1=\sqcup_{\lambda\in J_1} I_1 (\Pi s \Pi)\begin{mat}
		1 & 0\\p\lambda & 1
	\end{mat}$ yields
	\begin{equation}\label{eq1}
		T_{\Pi s\Pi} e_{\mathbf{1}}(\varphi)=\sum_{\lambda\in J_1} \begin{mat}
			1 &0\\p\lambda & 1
		\end{mat}\left(\begin{mat}
			0 & p^{-2}\\1 & 0
		\end{mat},1\right) e_{\mathbf{1}}(\varphi).
	\end{equation}
	Applying the $\wt{G}$-equivariant map $T_s$ to it gives the formula 
	\begin{equation}\label{eq2}
		(T_sT_{\Pi})^2e_{\mathbf{1}}(\varphi)=\sum_{\lambda\in J_1} \begin{mat}
			1 &0\\p\lambda & 1
		\end{mat}\left(\begin{mat}
			0 & p^{-2}\\1 & 0
		\end{mat},1\right) T_se_{\mathbf{1}}(\varphi).
	\end{equation}
	
	The equalities $T_s(T_s+1)e_{\mathbf{1}}=0$ of Lemma \ref{identities_Hecke_operators_I} (v) allows one to deduce from (\ref{eq0}), (\ref{eq1}) and (\ref{eq2}) that 
	\[
	\mathscr{Z}_OT_se_{\mathbf{1}}(\varphi)=\sum_{\lambda\in J_2} \begin{mat}
		1 & \lambda\\0 & 1
	\end{mat} \left(\begin{mat}
		1 & 0\\0 & p^{-2}
	\end{mat},1\right) T_se_{\mathbf{1}}(\varphi),
	\]
	and
	\begin{align*}
		\mathscr{Z}_O(T_s+1)e_{\mathbf{1}}(\varphi)&=\sum_{\lambda\in J_2} \begin{mat}
			1 & \lambda\\0 & 1
		\end{mat} \left(\begin{mat}
			1 & 0\\0 & p^{-2}
		\end{mat},1\right) (T_s+1)e_{\mathbf{1}}(\varphi)\\
		& +\sum_{\lambda\in J_1} \begin{mat}
			1 &0\\p\lambda & 1
		\end{mat}\left(\begin{mat}
			0 & p^{-2}\\1 & 0
		\end{mat},1\right) (T_s+1)e_{\mathbf{1}}(\varphi) 
	\end{align*}
	Applying $T_{\Pi^4}=1$, the first equality settles the case $r=p-1$, while the second equality covers the case $r=0$ by noting that $\tilde{s}$ fixes $(T_s+1)e_{\mathbf{1}}(\varphi)$.
\end{proof}

Recall from Definition \ref{M(chi,lambda)_def} that for each $\psi\in O$ and the element $\lambda\in k$, we have attached a right $\cH_{p^2=1}$-module $M(\psi,\lambda)$ as well as the module $M(\lambda)$ in the non-regular case. The previous corollary immediately gives:
	
\begin{cor}\label{I_1-inv_coker} As right $\cH_{p^2=1}$-modules,
	\[
	\cI\left(\cInd^{\wt{G}}_{\wt{K}Z(\wt{G})}(\tilde{\sigma}_r)\right)/(\wt{T}-\lambda)=\begin{cases}
		M(\chi,\lambda) \text{ if $r\neq 0$}\\
		M(\lambda) \text{ if $r=0$.}
	\end{cases}
	\]
\end{cor}

We remind the reader that we parametrized the supersingular Hecke modules in definitions \ref{def_simple_reg} and \ref{def_simple_nreg}, respectively.

\begin{thm}\label{thm_JH}
	\begin{enumerate}
		\item[{\rm (i)}] If $\lambda\neq 0$, then
		\[
		\cInd^{\wt{G}}_{\wt{K}Z(\wt{G})}(\tilde{\sigma}_r)/(\wt{T}-\lambda) \cong \Ind^{\wt{G}}_{\wt{B}_2}(\psi\boxtimes \iota),
		\]
		where $\psi\colon T_2\to k^{\times}$ is the smooth character defined by the requirements $\psi|_H=\chi^s$, $\psi\left(\begin{mat}
			p^{-2}& 0\\0 & 1
		\end{mat}\right)=\lambda$ and $\psi\left(\begin{mat}
			p^{-2} & 0\\0 & p^{-2}
		\end{mat}\right)=1$.
		
		\item[{\rm (ii)}] Assume now that $\lambda=0$.
		\begin{itemize}
			\item[{\rm (a)}] If $O$ is square-regular, then there is a non-split short exact sequence
			\[
			0\to \cT(\operatorname{SS}(\chi[1,0],1))\to \cInd^{\wt{G}}_{\wt{K}Z(\wt{G})}(\tilde{\sigma}_r)/(\wt{T})\to \cT(\operatorname{SS}(\chi,1))\to 0.
			\]
			\item[{\rm (b)}] For $r=p-1$ and $r=0$, there are non-split short exact sequences
			\[
			0\to \cT(\operatorname{SS}(0,1))\to \cInd^{\wt{G}}_{\wt{K}Z(\wt{G})}(\tilde{\sigma}_{p-1})/(\wt{T})\to \cT(\operatorname{SS}(-1,1))\to 0
			\]
			and
			\[
			0\to \cT(\operatorname{SS}(-1,1)) \to \cInd^{\wt{G}}_{\wt{K}Z(\wt{G})}(\tilde{\sigma}_{0})/(\wt{T})\to \cT(\operatorname{SS}(0,1))\to 0,
			\]
			respectively, while for $r=\frac{p-1}{2}$ there is a direct sum decomposition
			\[
			\cInd^{\wt{G}}_{\wt{K}Z(\wt{G})}(\tilde{\sigma}_{\frac{p-1}{2}})/(\wt{T})\cong \cT(\operatorname{SS}(0,1)) \oplus \cT(\operatorname{SS}(-1,1)).
			\]
		\end{itemize}
	\end{enumerate}
\end{thm}

\begin{proof}
	(i) Assume that $\bar{a}_p\neq 0$. By equation (\ref{ps_reg}) (resp.\ (\ref{ps_nreg})), Proposition \ref{M(chi,lambda)_regular} (resp.\ \ref{M(chi,lambda)_non_regular}) and Corollary \ref{I_1-inv_coker}, we have
	\[
	\cI(\Ind^{\wt{G}}_{\wt{B}_2}(\psi\boxtimes \iota))\cong 	\cI\left(\cInd^{\wt{G}}_{\wt{K}Z(\wt{G})}(\tilde{\sigma}_r)\right)/(\wt{T}-\lambda).
	\]
	(In the square-regular case, use $\operatorname{PS}(\chi^s,x,0,z)=\operatorname{PS}(\chi,0,x,z)$, see Remark \ref{inter}). The evident map from this object to the $I_1$-invariants of the cokernel of $\wt{T}-\lambda$ is non-trivial by Observation \ref{observation} (as the element $[1,x^r]$ is $I_1$-invariant and generates the compact induction as a $\wt{G}$-representation). Hence, by adjunction, we obtain a non-trivial $\wt{G}$-equivariant map
	\[
	\Ind^{\wt{G}}_{\wt{B}_2}(\psi\boxtimes \iota)=\cT\cI(\Ind^{\wt{G}}_{\wt{B}_2}(\psi\boxtimes \iota)) \to \cInd^{\wt{G}}_{\wt{K}Z(\wt{G})}(\tilde{\sigma}_r)/(\wt{T}-\lambda),
	\]
	where the first equality is due to Theorem \ref{thm_ps}. The source being irreducible, this map is necessarily injective, while the right-hand side being generated by the image of the $I_1$-invariant function $[1,x^r]$ implies that it is surjective.
	
	(ii) Assume now that $\lambda=0$. By Propositions \ref{M(chi,lambda)_regular} (resp.\ \ref{M(chi,lambda)_non_regular}) and Corollary \ref{I_1-inv_coker}, there are short exact sequences of the desired form if the middle term is replaced by $	\cT\left(\cI\left(\cInd^{\wt{G}}_{\wt{K}Z(\wt{G})}(\tilde{\sigma}_r)\right)/(\wt{T})\right)$. For this we use the right exactness of the functor $\cT$ and the fact that $\cT$ applied to a supersingular modules yields an irreducible representation by Theorem \ref{thm_ss}. It remains to show that the canonical map
	\[
	\cT\left(\cI\left(\cInd^{\wt{G}}_{\wt{K}Z(\wt{G})}(\tilde{\sigma}_r)\right)/(\wt{T})\right)\to \cInd^{\wt{G}}_{\wt{K}Z(\wt{G})}(\tilde{\sigma}_r)/(\wt{T})
	\]
	is an isomorphism. The consideration of the element $[1,x^r]$ again shows the surjectivity. By what we have just explained, it suffices to prove that each supersingular subrepresentation (or equivalently its $I_1$-invariants) of the source survives in the target.
	
	In the square-regular case, the proof of Proposition \ref{M(chi,lambda)_regular} shows that the element $T_se_{\chi}T_{\Pi}T_s$ is a basis of $\operatorname{SS}(\chi[1,0],1)(e_{\chi}+e_{\chi[1,0]})$. In the non-square-regular case, the proof of Proposition \ref{M(chi,lambda)_non_regular} shows that the element $T_se_{\mathbf{1}}T_{\Pi}T_sT_{\Pi}$ is a basis element of $\operatorname{SS}(0,1)e_{\mathbf{1}}\subset M(\mathbf{1},0)e_{\mathbf{1}}$, while $(T_s+1)e_{\mathbf{1}}T_{\Pi}T_sT_{\Pi}$ is a basis of $\operatorname{SS}(-1,1)e_{\mathbf{1}}\subset M(0)$; it also shows that in case $\chi=\omega^{\frac{p-1}{2}}\otimes \mathbf{1}$ (i.e.\ $r=\frac{p-1}{2}$), the element $T_se_{\chi}T_{\Pi}T_s$ is a basis for the subspace $\operatorname{SS}(-1,1)e_{\mathbf{1}}$, while $T_se_{\chi}T_{\Pi}(T_s+1)$ is a basis element of the subspace $\operatorname{SS}(0,1)e_{\mathbf{1}}$. By the Morita equivalences \ref{Morita_regular} and \ref{Morita_non-regular}, respectively, it is enough to prove that these basis elements survive in the cokernel of $\wt{T}$. We thus need to show that
	\begin{equation}\label{not_in_T}
		\begin{cases}[1,x^r]T_{\Pi}T_s \notin \operatorname{Im}(\wt{T}) \text{ if $0<r<p-1$}\\
			[1,x^r]T_{\Pi}T_sT_{\Pi} \notin \operatorname{Im}(\wt{T}) \text{ if $r=0,p-1$}\\
			[1,x^r]T_{\Pi}(T_s+1)\notin \operatorname{Im}(\wt{T}) \text{ if $r=\frac{p-1}{2}$.}\end{cases}
	\end{equation}
	Since these elements lie in $B_1(k)$, an application of Observation \ref{observation} finishes the proof.
\end{proof}

\begin{remark}\label{T(ss)_non_zero}
	The last part of the proof in particular shows that a supersingular $\cH_{p^2=1}$-module $M$ admits a non-trivial $\cH$-equivariant morphism to the non-zero (Observation \ref{observation}) space $\cI\left(\cInd^{\wt{G}}_{\wt{K}Z(\wt{G})}(\tilde{\sigma})/(\wt{T})\right)$ for some weight $\tilde{\sigma}$. In particular, by adjunction, $\cT(M)$ is non-zero.
\end{remark}

\begin{definition}
	For each $0\leq r \leq p-1$ with $r\neq \frac{p-1}{2}$, let $\tilde{\pi}(r)$ be the unique irreducible (supersingular) quotient of $\cInd^{\wt{G}}_{\wt{K}Z(\wt{G})}(\tilde{\sigma}_r)/(\wt{T})$. If $\eta\colon \bQ_p^{\times}\to k^{\times}$ is a smooth character, we put $\tilde{\pi}(r,\eta)=\tilde{\pi}(r)\otimes \eta\circ {\det}$.
\end{definition}

\begin{cor}\label{ss_param}
	Each genuine supersingular $\wt{G}$-representation is of the form $\tilde{\pi}(r,\eta)$ for some $0\leq r\leq p-1$, $r\neq \frac{p-1}{2}$, and a smooth character $\eta\colon \bQ_p^{\times}\to k^{\times}$.
	
	Moreover, $\tilde{\pi}(r)\cong \tilde{\pi}(r',\eta)$ if and only if one of the following holds
	\begin{align*}
		&r'=r, \text{ $\eta_{\operatorname{nr}}^4=1$ and $(\eta \eta_{\operatorname{nr}}^{-1})^2=1$;}\\
		&r'=\begin{cases}
			\frac{p-1}{2}-r \text{ if $0<r<\frac{p-1}{2}$}\\
			\frac{3}{2}(p-1)-r \text{ if $\frac{p-1}{2}<r<p-1$}\end{cases} \text{ and \phantom{m} $\eta_{\operatorname{nr}}^4=1$, $(\eta \eta_{\operatorname{nr}}^{-1} \omega^{-r})^2=1$,}
	\end{align*}
	where $\eta_{\operatorname{nr}}$ denotes the unramified part of $\eta$.
\end{cor}

\begin{proof}
	Corollary \ref{ss_socle} gives an explicit description of the socle of a genuine supersingular representation. In particular, after twisting by a character $\eta\circ \det$ it contains a weight $\tilde{\sigma}_r$ for some $0\leq r \leq p-1$ with $r\neq \frac{p-1}{2}$. We may also assume that the central element $p^2$ acts trivially. Since the weight appears with multiplicity one, the operator $\wt{T}$ must act on the corresponding weight space via a scalar. By Theorem \ref{thm_JH} (i), this scalar must be $0$, i.e.\ the supersingular representation is isomorphic to $\tilde{\pi}(r)$. This proves the first part.
	
	Any unramified character $\eta$ of order at most $4$ satisfies $(\eta\circ \det)|_{Z^{\sq}}=1$ and so
	\[
	\cInd^{\wt{G}}_{\wt{K}Z(\wt{G})}(\tilde{\sigma}_r)/(\wt{T}) \otimes (\eta \circ {\det})\cong 	\cInd^{\wt{G}}_{\wt{K}Z(\wt{G})}(\tilde{\sigma}_r)/(\wt{T}).
	\]
	By Lemma \ref{Ztilde_action}, the element $\tilde{\Pi}^2$ defines a $\wt{G}$-equivariant isomorphism
	\[
	\cInd^{\wt{G}}_{\wt{K}Z(\wt{G})}(\tilde{\sigma}_r)/(\wt{T}) \cong 	\cInd^{\wt{G}}_{\wt{K}Z(\wt{G})}(\tilde{\sigma}_r)/(\wt{T}) \otimes (\mu_{-1}^{\frac{p-1}{2}}\omega^{\frac{p-1}{2}})\circ {\det}.
	\]
	Putting these two observations together, we obtain that $\tilde{\pi}(r)\cong \tilde{\pi}(r,\eta)$ for $\eta$ satisfying $\eta_{\operatorname{nr}}^4=1$ and $(\eta \eta_{\operatorname{nr}}^{-1})^2=1$ (in fact this holds for the full cokernel of $\wt{T}$). By Corollary \ref{ss_socle} and Remark \ref{weight_remark}, we know that $\tilde{\sigma}_{r'}\otimes \omega^r\circ {\det}$ lies in the $K$-socle of $\tilde{\pi}(r)$ and thus $\tilde{\pi}(r)\cong \tilde{\pi}(r')\otimes \omega^r\circ {\det}$. This shows that the described intertwinings hold true.
	
	That these exhaust all intertwinings can be deduced from looking at the $K$-socle and the central character.
\end{proof}

\subsection{Howe's theorem}

The goal of this section is to explain a proof of the following theorem.

\begin{thm}\label{fg+adm=fl}
	A smooth $k$-representation of $\wt{G}$ is of finite length if and only if it is finitely generated and admissible.
\end{thm}

For the full subcategory $\Mod^{\sm}_G(k)$  of $\Mod^{\sm}_{\wt{G}}(k)$ this is a combination of \cite[Theorem 2.3.8]{Emerton_OP1} and \cite{Berger_central}, so we will only focus on the subcategory of genuine objects. The strategy is the same as that for the group $G$, but we still need to justify certain auxiliary results. 

Theorem \ref{thm_JH} implies that the cokernels of linear polynomials in the Hecke operator $\wt{T}$ are of finite length and the proof of \cite[Theorem 2.3.8]{Emerton_OP1} then goes through to show that any finitely generated admissible object of $\Mod^{\sm}_{\wt{G},\iota}(k)$ is indeed of finite length. 

For the converse, we need to show that a smooth irreducible genuine representation of $\wt{G}$ is admissible. As a first step, we need to check that the analogue of (the proof of) \cite[Proposition 32]{Barthel-Livne} goes through to prove the following result.

\begin{prop}\label{quotient_of_coker}
	Let $\ell$ be a field of characteristic $p$. Let $\pi$ be a smooth irreducible genuine $\ell$-representation of $\wt{G}$ admitting a central character. Then $\pi$ is a quotient of $\cInd^{\wt{G}}_{\wt{K}Z(\wt{G})}(\sigma \boxtimes \iota)/P(\wt{T})$ for some weight $\sigma$ and some irreducible polynomial $P(\wt{T})\in \ell[\wt{T}]$.
\end{prop}

The proof of this proposition requires two lemmas, the second of which is a metaplectic version of \cite[Proposition 18]{Barthel-Livne}. 

\begin{lem}\label{injective}
	Assume that $\chi:=\sigma^{I_1}$ is square-regular. The Hecke operator $T_{\Pi s}$ acts injectively on $\cInd^{\wt{G}}_{\wt{K}Z(\wt{G})}(\sigma\boxtimes \iota)^{I_1}e_{\chi[1,0]}$.
\end{lem}

\begin{proof}
	After twisting by a character, we may assume that $\sigma=\Sym^{r}(\ell^2)$, for some $0\leq r\leq p-1$, and the central element $p^2$ acts trivially. By Corollary \ref{I_1-inv_cInd}, we have $\cInd^{\wt{G}}_{\wt{K}Z(\wt{G})}(\sigma\boxtimes \iota)^{I_1}e_{\chi[1,0]}=T_se_{\chi}\cH_{p^2=1}e_{\chi[1,0]}$. For proving the desired injectivity, we may act by $T_{\Pi}$ from the left and it suffices to prove injectivity of the map $T_{\Pi s}^2=T_{(\Pi s)^2}$. In view of the isomorphisms in propositions \ref{Hecke_algebra_regular_char} and \ref{Hecke_module_regular_char}, we need to show that the element $Y\in \cH(\chi[1,0])_{p^2=1}\cong \ell[X,Y]/(XY)$ acts injectively on $T_{\Pi s}\cH(\chi[1,0],\chi)_{p^2=1}\cong Y\ell[X,Y]/(XY)$, which is clear. We note that the results we used hold true over an arbitrary field of characteristic $p$; alternatively, the lemma can be checked after passing to an algebraic closure.
\end{proof}

\begin{lem}\label{finite_codim}
	Let $\psi\colon HZ^{\sq}\to k^{\times}$ be a smooth character. Any non-zero subspace $W\subset \Hom_{\wt{G}}(\cInd^{\wt{G}}_{\wt{I}Z(\wt{G})}(\psi\boxtimes \iota),\cInd^{\wt{G}}_{\wt{K}Z(\wt{G})}(\sigma\boxtimes \iota))$ stable under the action of $\End_{\wt{G}}(\cInd^{\wt{G}}_{\wt{I}Z(\wt{G})}(\psi\boxtimes \iota))$ has finite codimension.
\end{lem}

\begin{proof}
	After twisting we may assume that $\sigma=\Sym^r(\ell^2)$, for some $0\leq r\leq p-1$, and the central element $p^2$ acts trivially. Write $\chi=\sigma^{I_1}$, viewed as a character of $HZ^{\sq}$. The Hom-space in the statement is equal to the $\cH(\psi)_{p^2=1}$-module $\cInd^{\wt{G}}_{\wt{K}Z(\wt{G})}(\sigma\boxtimes \iota)^{I_1}e_{\psi}$.
	
	Assume that $\chi$ is square-regular. By Corollary \ref{I_1-inv_cInd} we then have $\cInd^{\wt{G}}_{\wt{K}Z(\wt{G})}(\sigma\boxtimes \iota)^{I_1}e_{\psi}=T_se_{\chi}\cH_{p^2=1}e_{\psi}$. Since the statement is empty if $\psi$ does not lie in the orbit of $\chi$, we do from now on assume that $\psi\in O_{\chi}$. Acting by $T_{\Pi}$ from the right, we may assume that $\psi=\chi[t,0]$ for $t=0,1$. If $t=0$, then acting from the left by $T_{\Pi}$, we need to check that a non-zero $\cH(\chi)_{p^2=1}$-submodule of $T_{\Pi s}\cH(\chi)_{p^2=1}$ has finite codimension. But by propositions \ref{Hecke_algebra_regular_char} and \ref{Hecke_module_regular_char}, this just means that a $\ell[X,Y]/(XY)$-submodule of $\ell[Y]$ (with $X=0$ on it) has finite codimension, i.e.\ that every non-zero ideal of the polynomial ring $\ell[Y]$ has finite codimension. This is clear. The case $t=1$ is treated similarly.
	
	Assume now that $\chi$ is not square-regular, i.e.\ $r$ is a multiple of $\frac{p-1}{2}$. Without loss of generality, we may assume that $\psi=\mathbf{1}$. 
	
	If $r=p-1$, then by Corollary \ref{I_1-inv_cInd}, we are considering a non-zero $\cH(\mathbf{1})_{p^2=1}$-submodule $W$ of $T_s\cH(\mathbf{1})_{p^2=1}$. Using the description of $\cH(\mathbf{1})_{p^2=1}$ as in Proposition \ref{H(1)}, we decompose $W$ as
	\[
	WS_{0,0} \oplus W(S_{0,0}+1)\subset S_{0,0}\cH(\mathbf{1})_{p^2=1}S_{0,0}\oplus S_{0,0}\cH(\mathbf{1})_{p^2=1}(S_{0,0}+1)=S_{0,0}\cH(\mathbf{1})_{p^2=1}.
	\]
	We note that each summand on the left-hand side is non-zero (act by $S_{0,2}$). Now, as $\ell$-algebras, $\ell[Y]\cong S_{0,0}\cH(\mathbf{1})_{p^2=1}S_{0,0}$ via $1\mapsto -S_{0,0}$ and $Y\mapsto S_{0,0}S_{0,2}S_{0,0}$, so we are done with the inclusion of the first direct summand just like in the regular case. For the second one, note that we again have $\ell[Y]\cong (S_{0,0}+1)\cH(\mathbf{1})_{p^2=1}(S_{0,0}+1)$, where now $1\mapsto S_{0,0}+1$ and $Y\mapsto (S_{0,0}+1)S_{0,2}(S_{0,0}+1)$. As a module over this ring, the summand $S_{0,0}\cH(\mathbf{1})_{p^2=1}(S_{0,0}+1)$ is generated by the element $S_{0,0}S_{0,2}(S_{0,0}+1)$ and in particular any non-zero submodule does indeed have finite codimension.
	
	If $r=0$, then by Corollary \ref{I_1-inv_cInd}, we are considering a non-zero $\cH(\mathbf{1})_{p^2=1}$-submodule $W$ of $(S_{0,0}+1)\cH(\mathbf{1})_{p^2=1}$. Interchanging $S_{0,0}+1$ and $-S_{0,0}$ defines an automorphism of $\cH(\mathbf{1})_{p^2=1}$, which reduces to the previously treated case $r=p-1$.
	
	In case $r=\frac{p-1}{2}$, we make use of Corollary \ref{I_1-inv_cInd} once again and act by $T_{\Pi}$ from the left to reduce to showing that a non-zero $\cH(\mathbf{1})_{p^2=1}$-submodule $W$ of $T_{\Pi}T_sT_{\Pi} \cH(\mathbf{1})_{p^2=1}=S_{0,2}\cH(\mathbf{1})_{p^2=1}$ has finite codimension (here we used that $T_{\Pi}e_{\mathbf{1}}=e_{\chi} T_{\Pi}$, Lemma \ref{identities_Hecke_operators_II} (v)), which can now be proved as for $r=p-1$ by decomposing the problem according to the idempotent $-S_{0,0}$.
\end{proof}
	
\begin{proof}[Proof of Proposition \ref{quotient_of_coker}]
	Let $\pi$ be as in the statement. Picking a weight $\sigma\subset \pi|_K$ and letting the central element $p^2$ act on it according to the central character of $\pi$, we obtain a surjection $\varphi\colon \cInd^{\wt{G}}_{\wt{K}Z(\wt{G})}(\sigma\boxtimes \iota)\twoheadrightarrow \pi$. Since the compact induction is not irreducible (e.g.\ Theorem \ref{thm_JH}), the kernel is non-zero and thus $\ker(\varphi)^{I_1}e_{\psi}\neq 0$ for some $\psi\in O_{\chi}$. Either acting by $T_{\Pi}$ and using Lemma \ref{identities_Hecke_operators_II} (v) or acting by $T_{\Pi s}$ and using Lemma \ref{identities_Hecke_operators_II} (v), (vii) and Lemma \ref{identities_Hecke_operators_I} (iv), we may assume that $\psi=\chi$; here note that $T_{\Pi}$ (Lemma \ref{identities_Hecke_operators_II} (i)) and $T_{\Pi s}$ (Lemma \ref{injective}) act injectively.
	
	The inclusion $\chi=\sigma^{I_1}\subset \sigma$ gives rise to a surjective map $\cInd^{\wt{G}}_{\wt{I}Z(\wt{G})}(\chi\boxtimes \iota)\twoheadrightarrow \cInd^{\wt{G}}_{\wt{K}Z(\wt{G})}(\sigma\boxtimes \iota)$, which in turn induces a commutative diagram
	\[
	\xymatrix{
	\Hom_{\wt{G}}(\cInd^{\wt{G}}_{\wt{I}Z(\wt{G})}(\chi\boxtimes \iota), \cInd^{\wt{G}}_{\wt{K}Z(\wt{G})}(\sigma\boxtimes \iota)) \ar[r] & \Hom_{\wt{G}}(\cInd^{\wt{G}}_{\wt{I}Z(\wt{G})}(\chi\boxtimes \iota), \pi)\\
	\End_{\wt{G}}(\cInd^{\wt{G}}_{\wt{K}Z(\wt{G})}(\sigma\boxtimes \iota)) \ar[u]\ar[r] & \Hom_{\wt{G}}(\cInd^{\wt{G}}_{\wt{K}Z(\wt{G})}(\sigma\boxtimes \iota), \pi)\ar[u]
}
	\]
	with injective vertical maps, where the horizontal maps are induced by composing with $\varphi$. Lemma \ref{finite_codim} implies that the top arrow has finite image and thus the lower arrow does as well. But the image of the morphism at the bottom is non-zero (it contains $\varphi$) and it is stable under the action of the spherical Hecke algebra $\ell[\wt{T}]$ appearing as the source of the arrow. Picking a non-zero element in this finite dimensional $\ell[\wt{T}]$-submodule killed by some irreducible polynomial finishes the proof.
\end{proof}

\begin{prop}
	A smooth irreducible genuine $k$-representation of $\wt{G}$ is admissible.
\end{prop}

\begin{proof}
	As in \cite{Berger_central}, this can now be deduced from Proposition \ref{quotient_of_coker} and Theorem \ref{thm_JH}.
\end{proof}

\section{Relation to locally algebraic representations}

Fix an algebraic closure $\Qpbar$ of $\bQ_p$ with ring of integers $\Zpbar$. Keeping the notation of the previous section and putting $k=\Fpbar$, we would like to relate the smooth $\Fpbar$-representations $\cInd^{\wt{G}}_{\wt{K}Z(\wt{G})}(\tilde{\sigma}_r(\Fpbar))/(\wt{T}-\bar{a}_p)$, for $\bar{a}_p\in \Fpbar$ and $0\leq r\leq p-1$, to the reduction modulo the maximal ideal in $\Zpbar$ of a $\wt{G}$-stable bounded lattice of a locally algebraic $\Qpbar$-representation (or equivalently its unitary completion) of $\wt{G}$. Choosing a lift $a_p\in \Zpbar$ of $\bar{a}_p$, a natural candidate to consider is the representation
\[
\wt{\Pi}_{r,a_p}:=\cInd^{\wt{G}}_{\wt{K}Z(\wt{G})}(\tilde{\sigma}_r(\Qpbar))/(\wt{T}_{\Qpbar}-a_p).
\]

Although it will not be necessary for us, we first describe $\wt{\Pi}_{r,a_p}$ in terms of (constituents of) locally algebraic unramified principal series, which can be deduced from the corresponding statement \cite[Proposition 3.2.1]{Breuil_II} for the group $G$ by utilizing Savin's local Shimura correspondence.

\subsection{Savin's local Shimura correspondence}

Let us denote by $\cH_{\iota}(\wt{G},\wt{I})=\End_{\wt{G}}(\cInd^{\wt{G}}_{\wt{I}}(\mathbf{1}\boxtimes \iota))$ the genuine Iwahori Hecke algebra, where $\mathbf{1}$ is the trivial $\Qpbar^{\times}$-valuded character, and let $\cH(G,I)=\End_{G}(\cInd^G_I(\mathbf{1}))$ be the usual Iwahori Hecke algebra of $G$. The work \cite[Proposition 3.1.2]{Savin} of Savin provides us with a description of $\cH_{\iota}(\wt{G},\wt{I})$ in terms of generators and relations. Comparing with the description of $\cH(G,I)$, see for example \cite{Rostami}, we obtain an explicit isomorphism
\begin{equation}\label{iso_IwHecke}
	\cH_{\iota}(\wt{G},I)\cong \cH(G,I)
\end{equation}
of $\Qpbar$-algebras, see also \cite[§4]{Savin}. In representation theoretic terms, this translates into an equivalence of categories
\small
\begin{equation}\label{equiv_I}
	\left\{\begin{array}{c}\text{Smooth genuine}\\ \text{$\Qpbar$-representations of $\wt{G}$}\\ \text{generated by their $I$-invariants}\end{array}\right\} \cong \left\{\begin{array}{c}\text{Smooth }\\ \text{$\Qpbar$-representations of $G$}\\ \text{generated by their $I$-invariants}\end{array}\right\}.
\end{equation}
\normalsize

\begin{lem}\label{Savin-cInd}
	The equivalence (\ref{equiv_I}) maps $\cInd^{\wt{G}}_{\wt{K}Z(\wt{G})}(\mathbf{1}\boxtimes \iota)$ to $\cInd^G_{KZ}(\mathbf{1})$ and the endomorphism $\wt{T}_{\Qpbar}$ of the former to the endomorphism $p^{1/2}T_{\Qpbar}$ of the latter, where $T_{\Qpbar}$ corresponds to the double coset $KZ\begin{mat}
		1 & 0\\0&p^{-1}
	\end{mat}KZ$.
\end{lem}

\begin{proof}
	It is enough to check the assertions on $I$-invariants. Let $\wt{T}_s\in \cH_{\iota}(\wt{G},\wt{I})$ be the function with support $\wt{I}\tilde{s}\wt{I}$ sending $\tilde{s}$ to $1$ and let $T_s\in \cH(G,I)$ be the indicator function of $IsI$. For $n\geq m$, let $\tilde{\mathbf{1}}_{2n,2m}\in \cH_{\iota}(\wt{G},\wt{I})$ be the function with support $\wt{I}\left(\begin{mat}
		p^{2n} & 0\\ 0 & p^{2m}
	\end{mat},1\right)\wt{I}$ and value $1$ at $\left(\begin{mat}
	p^{2n} & 0\\ 0 & p^{2m}
\end{mat},1\right)$, and let $\mathbf{1}_{n,m}\in \cH(G,I)$ be the indicator function of $I\begin{mat}
p^n & 0\\ 0 & p^m
\end{mat}I$. The isomorphism (\ref{iso_IwHecke}) then maps $\wt{T}_s$ to $T_s$ and $\tilde{\mathbf{1}}_{2n,2m}$ to $p^{\frac{n-m}{2}}\mathbf{1}_{n,m}$.

Considered as a subspace $\cInd^G_K(\mathbf{1})^{I}\subset \cInd^G_I(\mathbf{1})^{I}\cong \cH(G,I)$, we have the equality $\cInd^G_K(\mathbf{1})^{I}=(1+T_s)\cH(G,I)$ since $1+T_s$ is the indicator function of $K$. One similarly has $\cInd^{\wt{G}}_{\wt{K}}(\mathbf{1}\boxtimes \iota)^{I}=(1+\wt{T}_s)\cH_{\iota}(\wt{G},\wt{I})$. This proves the first part. 
Concerning the Hecke operators, one may check that $T_{\Qpbar}=p^{-1}(1+T_s)\mathbf{1}_{0,-1}(1+T_s)$ in $\cH(G,I)/(\mathbf{1}_{1,1}-1)$ and $\wt{T}_{\Qpbar} = p^{-1}(1+\wt{T}_s)\tilde{\mathbf{1}}_{0,-2}(1+\wt{T}_s)$ in $\cH_{\iota}(\wt{G},\wt{I})/(\tilde{\mathbf{1}}_{2,2}-1)$.
\end{proof}

\subsection{Genuine principal series in characteristic $0$}

Fix an isomorphism $\Qpbar \cong \mathbf{C}$ and let $|\cdot|\colon \bQ_p\to \mathbf{Q}_{\geq 0}$ be the $p$-adic norm so that $|p|=p^{-1}$. Via the fixed isomorphism, we will view fractional powers of $|\cdot|$ as $\Qpbar^{\times}$-valued characters of $\bQ_p^{\times}$.

\begin{definition} 
	\begin{enumerate}
		\item[{\rm (i)}] For two smooth characters $\chi_1,\chi_2\colon A\to \Qpbar^{\times}$, define the normalized principal series
		\[
		\iota^{\wt{G}}_{\wt{B}_2}(\chi_1\otimes \chi_2)=\Ind^{\wt{G}}_{\wt{B}_2}(\chi_1|\cdot|^{1/2}\otimes \chi_2|\cdot|^{-1/2} \boxtimes \iota),
		\]
		which only depends on the restriction of $\chi_1,\chi_2$ to the subgroup $S$ of squares in $\bQ_p^{\times}$.
		\item[{\rm (ii)}] For two smooth characters $\chi_1,\chi_2\colon \bQ_p^{\times}\to \Qpbar^{\times}$, let
		\[
		\iota^G_B(\chi_1\otimes \chi_2)=\Ind^G_B(\chi_1|\cdot|^{1/2}\otimes \chi_2|\cdot|^{-1/2})
		\]
		be the usual normalized principal series of $G$ attached to $(\chi_1,\chi_2)$.
	\end{enumerate}
\end{definition}

Given a character $\chi\colon A\to \Qpbar^{\times}$, denote by $\chi'\colon \bQ_p^{\times}\to \Qpbar^{\times}$ the character defined by
\[
\chi'(x)=\chi(x^2), \text{ for all $x\in \bQ_p^{\times}$.}
\]

\begin{lem}\label{Savin_PS}
	Let $\chi_1,\chi_2\colon A\to \Qpbar^{\times}$ be smooth characters trivial on $T\cap K$. The equivalence (\ref{equiv_I}) maps $\iota^{\wt{G}}_{\wt{B}}(\chi_1\otimes \chi_2)$ to $\iota^{G}_B(\chi_1'\otimes \chi_2')$.
\end{lem}

\begin{proof}
	First note that the genuine principal series must correspond to an (unramified) principal series, say $\iota^G_B(\psi_1\otimes \psi_2)$ for two smooth unramified characters $\psi_1,\psi_2$. A basis for the space $\iota^{G}_B(\psi_1\otimes \psi_2)^{I}$ is given by the two functions $\varphi,\varphi_s$ defined by the conditions $\supp(\varphi)=BI$, $\supp(\varphi_s)=BsI$ and $\varphi(1)=\varphi_s(s)=1$. A basis for the space $\iota^{\wt{G}}_{\wt{B}}(\chi_1\otimes \chi_2)^{I}$ is given by the two functions $\tilde{\varphi},\tilde{\varphi}_s$ defined similarly but requiring that $\tilde{\varphi}(1)=\tilde{\varphi}_s(\tilde{s})\in \Ind^{\wt{T}}_{\wt{T}_2}(\chi_1|\cdot|^{1/2}\otimes \chi_2|\cdot|^{-1/2}\boxtimes \iota)^{T\cap K}$ is the standard basis vector $f_{0,0}$ as in Definition \ref{def_std_basis}. Letting $\mathbf{1}_{0,-1}\in \cH(G,I)$ and $\tilde{\mathbf{1}}_{0,-2}\in \cH_{\iota}(\wt{G},\wt{I})$ be as in the proof of Lemma \ref{Savin-cInd}, one checks that with respect to the chosen bases,
	\[
	\mathbf{1}_{0,-1}=\begin{mat}
		p^{1/2}\psi_2(p) & p^{1/2}(p-1)\psi_2(p)\\0 & p^{1/2}\psi_1(p)
	\end{mat} \text{ and } \tilde{\mathbf{1}}_{0,-2}=\begin{mat}
	p\chi_2(p^2) & p(p-1)\chi_2(p^2)\\0 & p\chi_1(p^2)
\end{mat}.
	\]
	Since the isomorphism (\ref{iso_IwHecke}) identifies $\tilde{\mathbf{1}}_{0,-2}$ with $p^{1/2}\mathbf{1}_{0,-1}$, we deduce that $\{\psi_1,\psi_2\}=\{\chi_1',\chi_2'\}$. If $\psi_1\psi_2^{-1}\neq |\cdot|^{\pm 1}$, then $\iota^G_B(\psi_1\otimes \psi_2)\cong \iota^G_B(\psi_2\otimes \psi_1)$, so we are done. Assume now by contradiction that $\psi_1\psi_2^{-1}=|\cdot|$ and $\chi_1'(\chi_2')^{-1}=|\cdot|^{-1}$. By \cite[Theorem I.2.9]{Kazhdan_Patterson}, $\iota^{\wt{G}}_{\wt{B}}(\chi_1'\otimes \chi_2')$ contains a unique irreducible subrepresentation admitting a $K$-invariant vector (and so the quotient cannot admit one), while $\iota^G_B(\psi_1\otimes \psi_2)$ is a non-split extensions of an unramified character by the corresponding twist of the Steinberg representation (not admitting a $K$-invariant vector). Since the equivalence (\ref{equiv_I}) preserves the property of having $K$-invariant vectors (the isomorphism (\ref{iso_IwHecke}) maps the genuine indicator function of $\wt{K}$ to the indicator function of $K$ as observered in the proof of Lemma \ref{Savin-cInd}), we obtain the desired contradiction.
\end{proof}

\begin{def-prop}
	Let $\chi_1,\chi_2\colon A\to \Qpbar^{\times}$ be smooth characters.
	
	\begin{enumerate}
		\item[{\rm (i)}] If $\chi_1'(\chi_2')^{-1} \neq |\cdot|^{\pm 1}$, then $\iota^{\wt{G}}_{\wt{B}_2}(\chi_1\otimes \chi_2)$ is irreducible.
		
		\noindent Moreover, if $\psi_1,\psi_2\colon A\to \Qpbar^{\times}$ is another pair of smooth characters, then $\iota^{\wt{G}}_{\wt{B}_2}(\chi_1\otimes \chi_2)\cong \iota^{\wt{G}}_{\wt{B}_2}(\psi_1\otimes \psi_2)$ if and only if $(\psi_1',\psi_2')=(\chi_1',\chi_2')$ or $(\psi_1',\psi_2')=(\chi_2',\chi_1')$.
		\item[{\rm (ii)}] If $\chi_1'(\chi_2')^{-1} = |\cdot|^{-1}$, then there are non-split short exact sequence
		\[
		0\to \tilde{\pi}(\chi_1,\chi_2)\to \iota^{\wt{G}}_{\wt{B}_2}(\chi_1\otimes \chi_2)\to \tilde{\pi}(\chi_2,\chi_1)\to 0
		\]
		and
		\[
		0\to \tilde{\pi}(\chi_2,\chi_1)\to \iota^{\wt{G}}_{\wt{B}_2}(\chi_2\otimes \chi_1)\to \tilde{\pi}(\chi_1,\chi_2)\to 0,
		\]
		with $\tilde{\pi}(\chi_1,\chi_2)$, $\tilde{\pi}(\chi_2,\chi_1)$ irreducible. Moreover, $\tilde{\pi}(\chi_1,\chi_2)^K\neq 0$ if and only if $\chi_2'$ is unramified, in which case $\tilde{\pi}(\chi_1,\chi_2)$ and $\tilde{\pi}(\chi_2,\chi_1)$ correspond to $\chi_2'|\cdot|^{-1/2}\circ \det$ and $\operatorname{St}\otimes \chi_2'|\cdot|^{-1/2}\circ \det$, respectively, under the equivalence (\ref{equiv_I}).
	\end{enumerate}
\end{def-prop}

\begin{proof}
	(i) The irreducibility is \cite[Corollary I.2.8]{Kazhdan_Patterson}, while the intertwinings follow from the computation of the Jacquet module of a genuine principal series, see Proposition I.2.1 of \textit{loc.\ cit.}
	
	(ii) After twisting, we may assume that $\chi_2'$ is unramified, in which case the assertion follows from the equivalence (\ref{equiv_I}) and Lemma \ref{Savin_PS}.
\end{proof}

For $y\in \Qpbar^{\times}$, let $\tilde{\mu}_y\colon A\to \Qpbar^{\times}$ denote the unique character trivial on $T\cap K$ and mapping $p^2$ to $y$.

\begin{cor}\label{loc_alg_descr}
	Let $r\in \bZ_{\geq 0}$ and $a_p\in \Qpbar$. Let $\alpha,\beta\in\Qpbar$ be the roots of the polynomial $X^2-a_pX+p^{2(r+1)}$.
	
	\begin{enumerate}
		\item[{\rm (i)}] If $a_p\notin\{\pm p^{r+1/2}(1+p)\}$, then
		\[
		\wt{\Pi}_{r,a_p}\cong \Sym^r(\Qpbar^2)\otimes \Ind^{\wt{G}}_{\wt{B}_2}(\tilde{\mu}_{\alpha^{-1}}\otimes \tilde{\mu}_{p^2\beta^{-1}}\boxtimes \iota).
		\]
		\item[{\rm (ii)}] If $a_p\in\{\pm p^{r+1/2}(1+p)\}$, then there is a short exact sequence
		\[
		0\to \Sym^r(\Qpbar^2)\otimes \tilde{\pi}(\tilde{\mu}_{\varepsilon},\tilde{\mu}_{\varepsilon}|\cdot|^{-1/2})\to \wt{\Pi}_{r,a_p}\to \Sym^r(\Qpbar^2)\otimes \tilde{\pi}(\tilde{\mu}_{\varepsilon}|\cdot|^{-1/2},\tilde{\mu}_{\varepsilon})\to 0,
		\]
		where $\varepsilon =\frac{p^{-2r}a_p}{p(p+1)}$.
	\end{enumerate}
\end{cor}

\begin{proof}
	The case $r=0$ follows from the equivalence (\ref{equiv_I}), Lemma \ref{Savin-cInd}, Lemma \ref{Savin_PS} and the corresponding statement for the group $G$, which is a special case of \cite[Proposition 3.2.1]{Breuil_II}.
	The general case is now a consequence of this by noting that the isomorphism
	\begin{align*}
		(\Sym^r(\Qpbar^2)\otimes |\cdot|^{r/2}\circ \det) \otimes \cInd^{\wt{G}}_{\wt{K}Z(\wt{G})}(\mathbf{1}\boxtimes \iota)\cong \cInd^{\wt{G}}_{\wt{K}Z(\wt{G})}(\tilde{\sigma}_r(\Qpbar))\\
		v\otimes f\mapsto [g\mapsto f(g)(g.v)]
	\end{align*}
	identifies $\operatorname{id}\otimes \wt{T}_{\Qpbar}$ with $p^{-r}\wt{T}_{\Qpbar}$.
\end{proof}

\subsection{Mod-$p$ reduction}
Let now $0\leq r\leq p-1$ and $a_p\in \Zpbar$.

\begin{lem}\label{vanishes_for_all_lambda}
	Suppose that for $0\leq i \leq p-1$ we are given elements $a_i\in \ol{\bQ}_p$ such that $\sum_{i=0}^{p-1} a_i [\lambda]^{i}\in \ol{\bZ}_p$ for all $\lambda\in \bF_p$. Then $a_i\in \ol{\bZ}_p$
\end{lem}

\begin{proof}
	Choose $L/\bQ_p$ finite large enough to contain all the $a_i$. Let $\mathcal{O}_L$ be the ring of integers with uniformizer $\varpi$ and residue field $k_L$. Let $N\gg 0$ such that $b_i:=\varpi^{N}a_i\in \mathcal{O}_L$ for all $0\leq i\leq p-1$. By assumption, $\sum_{i=0}^{p-1} b_i [\lambda]^{i}\in \varpi^{N}\mathcal{O}_L$ for all $\lambda\in \bF_p$. We claim that $b_i\in \varpi^{N}\mathcal{O}_L$. By induction, we may assume that $N=1$. In this case, the mod-$\varpi$ reduction $\bar{f}(X)\in k_L[X]$ of the polynomial $\sum_{i=0}^{p-1}b_iX^{i}\in \mathcal{O}_L[X]$ vanishes on the set $\bF_p$ of cardinality $p$, which forces $\bar{f}(X)=0$ for degree reasons, i.e.\ $b_i\in \varpi \mathcal{O}_L$ as desired.
\end{proof}

By our restriction on $r$ being at most $p-1$, the following assertion can be proved in the same way as Breuil shows Corollaire 4.1.2 in \cite{Breuil_II} by utilizing the previous lemma, which allows to obtain appropriate versions of Proposition 3.3.3 and Th\'eor\'eme 4.1.1 in \textit{loc.\ cit.}

\begin{prop}\label{prop_lattice}
		The $\wt{G}$-equivariant map
	\begin{equation}\label{inclusion_lattice}
		\cInd^{\wt{G}}_{\wt{K}Z(\wt{G})}(\tilde{\sigma}_r(\ol{\bZ}_p))/(\wt{T}_{\ol{\bZ}_p}-a_p) \hookrightarrow 	\cInd^{\wt{G}}_{\wt{K}Z(\wt{G})}(\tilde{\sigma}_r(\ol{\bQ}_p))/(\wt{T}_{\ol{\bQ}_p}-a_p)=\wt{\Pi}_{r,a_p}
	\end{equation}
	is injective and defines a $\wt{G}$-stable bounded lattice in the $\Qpbar[\wt{G}]$-module $\wt{\Pi}_{r,a_p}$, i.e.\ a $\ol{\bZ}_p[\wt{G}]$-submodule generating the whole space over $\ol{\bQ}_p$ but not containing a $\ol{\bQ}_p$-line.
\end{prop}

\begin{remark}
	\begin{enumerate}
		\item[{\rm (i)}] We have
		\[
		\cInd^{\wt{G}}_{\wt{K}Z(\wt{G})}(\tilde{\sigma}_r(\ol{\bZ}_p))/(\wt{T}_{\ol{\bZ}_p}-a_p)\otimes_{\Zpbar} \Fpbar = \cInd^{\wt{G}}_{\wt{K}Z(\wt{G})}(\tilde{\sigma}_r(\Fpbar))/(\wt{T}-\bar{a}_p),
		\]
		where $\bar{a}_p$ is the image of $a_p$ under the projection $\Zpbar\twoheadrightarrow \Fpbar$, so from Theorem \ref{thm_JH} we deduce a description of the (semi-simplified) mod-$p$ reduction of (any $\wt{G}$-stable bounded lattice in) $\wt{\Pi}_{r,a_p}$.
		\item[{\rm (ii)}] In fact, Breuil allows $0\leq r\leq 2(p-1)$. But under the more restrictive assumption $r\leq p-1$ only the $\wt{T}^+$-part of the Hecke operator, see Definition \ref{def_T+}, needs to be considered. We have not attempted to investigate the more complicated cases $p\leq r\leq 2(p-1)$ in our setting.
	\end{enumerate}
\end{remark}

\bibliography{metaplectic_reps_arxiv}
\bibliographystyle{plain}	
	
\end{document}